\pdfoutput=1

\documentclass[a4paper,12pt]{article}

\usepackage[utf8]{inputenc}
\usepackage[T1]{fontenc}
\usepackage[UKenglish]{babel}
\usepackage{lmodern}
\usepackage{fullpage}
\usepackage{amsmath,amssymb,amsthm}
\usepackage{mathtools}
\usepackage{tikz-cd}
\usepackage{hyperref}

\DeclareFontFamily{OMX}{lmex}{}
\DeclareFontShape{OMX}{lmex}{m}{n}{<->lmex10}{}

\tikzcdset{arrow style=Latin Modern}

\theoremstyle{plain}
\newtheorem{theo}{Theorem}[section]
\newtheorem{prop}[theo]{Proposition}
\newtheorem{conj}[theo]{Conjecture}
\newtheorem{coro}[theo]{Corollary}
\newtheorem{lemm}[theo]{Lemma}
\newtheorem*{slem}{Sublemma}

\theoremstyle{definition}
\newtheorem{defi}[theo]{Definition}
\newtheorem{nota}[theo]{Notation}

\theoremstyle{remark}
\newtheorem{rema}[theo]{Remark}

\DeclareMathOperator{\card}{card}
\DeclareMathOperator{\supp}{supp}

\DeclareMathOperator{\Hom}{Hom}
\DeclareMathOperator{\Ext}{Ext}

\DeclareMathOperator{\Set}{Set}
\DeclareMathOperator{\Mod}{Mod}
\DeclareMathOperator{\Ban}{Ban}

\DeclareMathOperator{\Lie}{Lie}
\DeclareMathOperator{\Fil}{Fil}
\DeclareMathOperator{\Gr}{Gr}
\DeclareMathOperator{\Indt}{Ind-\!}
\DeclareMathOperator{\Prot}{Pro-\!}
\DeclareMathOperator{\Res}{Res}
\DeclareMathOperator{\Nrm}{Nrm}
\DeclareMathOperator{\cind}{c-ind}
\DeclareMathOperator{\Ind}{Ind}
\DeclareMathOperator{\Ord}{Ord}

\newcommand{\middlevert}{\;\middle|\;}
\newcommand{\llbrack}{[\![}
\newcommand{\rrbrack}{]\!]}
\newcommand{\iso}{\overset{\sim}{\longrightarrow}}

\newcommand{\vertsim}{\rotatebox{-90}{\(\sim\)}}
\newcommand{\h}{\overset{\mathrm{H}}{\cdot}}
\newcommand{\Ev}{\mathrm{Ev}}
\newcommand{\ev}{\mathrm{ev}}
\newcommand{\cont}{\mathrm{cont}}
\newcommand{\sm}{\mathrm{sm}}
\newcommand{\adm}{\mathrm{adm}}
\newcommand{\ladm}{\mathrm{l.adm}}
\newcommand{\admu}{\mathrm{adm,u}}
\newcommand{\lfin}{\mathrm{l.fin}}
\newcommand{\fl}{\mathrm{fl}}
\newcommand{\op}{\mathrm{op}}

\newcommand{\der}{\mathrm{der}}

\newcommand{\ord}{\mathrm{ord}}

\newcommand{\N}{\mathbb{N}}
\newcommand{\Z}{\mathbb{Z}}
\newcommand{\Q}{\mathbb{Q}}
\newcommand{\Fp}{\mathbb{F}_p}
\newcommand{\Fpbar}{\overline{\Fp}}
\newcommand{\Zp}{\Z_p}
\newcommand{\Qp}{\Q_p}

\newcommand{\GL}{\mathrm{GL}}

\newcommand{\Rc}[1][\bullet]{\mathrm{R}^{#1}}
\newcommand{\ROrd}[1][\bullet]{\Rc[{#1}]\!\Ord}
\newcommand{\Hc}[1][\bullet]{\mathrm{H}^{#1}}
\newcommand{\HOrd}[1][\bullet]{\Hc[{#1}]\!\Ord}
\newcommand{\Lh}[1][\bullet]{\mathrm{L}_{#1}}
\newcommand{\Hh}[1][\bullet]{\mathrm{H}_{#1}}

\newcommand{\Hb}{\mathbf{H}}
\newcommand{\Gb}{\mathbf{G}}
\newcommand{\Sb}{\mathbf{S}}
\newcommand{\Zc}{\mathcal{Z}}
\newcommand{\Zcb}{\boldsymbol{\Zc}}
\newcommand{\Nc}{\mathcal{N}}
\newcommand{\Ncb}{\boldsymbol{\Nc}}
\newcommand{\Tb}{\mathbf{T}}
\newcommand{\Zb}{\mathbf{Z}}
\newcommand{\Bb}{\mathbf{B}}
\newcommand{\Ub}{\mathbf{U}}
\newcommand{\Pb}{\mathbf{P}}
\newcommand{\Lb}{\mathbf{L}}
\newcommand{\Nb}{\mathbf{N}}
\newcommand{\Qb}{\mathbf{Q}}

\newcommand{\Gbt}{\widetilde{\Gb}}
\newcommand{\Gt}{\widetilde{G}}
\newcommand{\Pbt}{\widetilde{\Pb}}
\newcommand{\Pt}{\widetilde{P}}
\newcommand{\Lbt}{\widetilde\Lb}
\newcommand{\Lt}{\widetilde{L}}
\newcommand{\lt}{\tilde{l}}
\newcommand{\Nbt}{\widetilde{\Nb}}
\newcommand{\Nt}{\widetilde{N}}
\newcommand{\nt}{\tilde{n}}
\newcommand{\Zbt}{\widetilde{\Zb}}
\newcommand{\Zt}{\widetilde{Z}}
\newcommand{\zt}{\tilde{z}}
\newcommand{\Sbt}{\widetilde{\Sb}}
\newcommand{\St}{\widetilde{S}}
\newcommand{\Bbt}{\widetilde{\Bb}}
\newcommand{\Bt}{\widetilde{B}}
\newcommand{\Ubt}{\widetilde{\Ub}}
\newcommand{\Ut}{\widetilde{U}}
\newcommand{\Wt}{\widetilde{W}}
\newcommand{\wt}{\tilde{w}}
\newcommand{\ellt}{\tilde{\ell}}
\newcommand{\dt}{\tilde{d}}
\newcommand{\alphat}{\tilde{\alpha}}
\newcommand{\deltat}{\tilde{\delta}}
\newcommand{\lambdat}{\tilde{\lambda}}
\newcommand{\mut}{\tilde{\mu}}
\newcommand{\sigmat}{\tilde{\sigma}}

\newcommand{\X}{\mathrm{X}}
\newcommand{\Ec}{\mathcal{E}}
\newcommand{\Oc}{\mathcal{O}}
\newcommand{\Ic}{\mathcal{I}}
\newcommand{\Cc}{\mathcal{C}}
\newcommand{\Clis}{\Cc^\sm}
\newcommand{\Clisc}{\Clis_\mathrm{c}}
\newcommand{\Cf}{\mathfrak{C}}

\newcommand{\IW}{{}^IW}
\newcommand{\Iw}{{}^I\!w}
\newcommand{\IpW}{{}^{I'}W}
\newcommand{\Ipw}{{}^{I'}\!\!w}
\newcommand{\JIW}{{}^{J \cap \Iw^J{}^{-1}(I)}W}
\newcommand{\JIw}{{}^{J \cap \Iw^J{}^{-1}(I)}w}
\newcommand{\JIpW}{{}^{J \cap \Ipw^J{}^{-1}(I')}W}

\hypersetup{pdfinfo={
	Title={Parabolic induction and extensions},
	Author={Julien Hauseux},
	Subject={22E50},
	Keywords={p-adic reductive groups, mod p representations, parabolic induction, extensions, derived ordinary parts, Bruhat filtration}
}}

\title{Parabolic induction and extensions}
\author{Julien Hauseux\thanks{This research was partly supported by EPSRC grant EP/L025302/1.}}
\date{}

\begin{document}

\maketitle

\begin{abstract}
Let $G$ be a $p$-adic reductive group.
We determine the extensions between admissible smooth mod $p$ representations of $G$ parabolically induced from supersingular representations of Levi subgroups of $G$, in terms of extensions between representations of Levi subgroups of $G$ and parabolic induction.
This proves for the most part a conjecture formulated by the author in a previous article and gives some strong evidence for the remaining part.
In order to do so, we use the derived functors of the left and right adjoints of the parabolic induction functor, both related to Emerton's $\delta$-functor of derived ordinary parts.
We compute the latter on parabolically induced representations of $G$ by pushing to their limits the methods initiated and expanded by the author in previous articles.
\end{abstract}

\tableofcontents

\section{Introduction}

The study of representations of a $p$-adic reductive group $G$ over a field of characteristic $p$ has a strong motivation in the search for a possible mod $p$ Langlands correspondence for $G$.
Recently, Abe, Henniart, Herzig and Vignéras (\cite{AHHV}) gave a complete classification of the irreducible admissible smooth representations of $G$ over an algebraically closed field of characteristic $p$ in terms of supersingular representations of the Levi subgroups of $G$ and parabolic induction, generalising the results of Barthel and Livné for $\GL_2$ (\cite{BL}), Herzig for $\GL_n$ (\cite{Her}) and Abe for a split $G$ (\cite{Abe}).

Two major difficulties come into play when trying to extend the mod $p$ Langlands correspondence beyond $\GL_2(\Qp)$.
First, the supersingular representations of $G$ remain completely unknown, except for some reductive groups of relative semisimple rank $1$ over $\Qp$ (\cite{Abd,Che,KK}) using the classification of Breuil for $\GL_2(\Qp)$ (\cite{Br1}).
Second, it is expected that such a correspondence would involve representations of $G$ with many irreducible constituents (see e.g. \cite{BH}).
This phenomenon already appears for $\GL_2(\Qp)$ when the Galois representation is an extension between two characters, in which case the associated representation of $\GL_2(\Qp)$ is an extension between two principal series (\cite{Col}).
This raises the question of the extensions between representations of $G$.

\medskip

In this article, we intend to compute the extensions between admissible smooth mod~$p$ representations of $G$ parabolically induced from supersingular representations of Levi subgroups of $G$, in terms of extensions between representations of Levi subgroups of $G$ and parabolic induction.

In order to do so, we use the derived functors of the left and right adjoints of the parabolic induction functor, namely the Jacquet functor and the ordinary parts functor (\cite{Em1}), both related to Emerton's $\delta$-functor of derived ordinary parts (\cite{Em2}).
We compute the latter on parabolically induced representations of $G$ by pushing to their limits the methods initiated in \cite{JH} and expanded in \cite{JHB}.

These computations have also been used to study the deformations of parabolically induced admissible smooth mod $p$ representations of $G$ in a joint work with T.~Schmidt and C.~Sorensen (\cite{JHSS}).

\subsection*{Presentation of the main results}

We let $F/\Qp$ and $k/\Fp$ be finite extensions.
We fix a connected reductive algebraic \mbox{$F$-group} $\Gb$, a minimal parabolic subgroup $\Bb \subseteq \Gb$ and a maximal split torus $\Sb \subseteq \Bb$.
We write the corresponding groups of $F$-points $G$, $B$, $S$, etc.
We let $\Delta$ denote the set of simple roots of $\Sb$ in $\Bb$.
To each $\alpha \in \Delta$ correspond a simple reflection $s_\alpha$ and a root subgroup $\Ub_\alpha \subset \Bb$.
We put $\Delta^1 \coloneqq \{\alpha \in \Delta \mid \dim_F \Ub_\alpha = 1\}$.

Let $\Pb = \Lb \Nb$ be a standard parabolic subgroup.
We write $\Delta_\Lb \subseteq \Delta$ for the corresponding subset and we put $\Delta_\Lb^\perp \coloneqq \{\alpha \in \Delta \mid \langle \alpha, \beta^\vee \rangle = 0 \ \forall \beta \in \Delta_\Lb\}$ and $\Delta_\Lb^{\perp,1} \coloneqq \Delta_\Lb^\perp \cap \Delta^1$.
For $\alpha \in \Delta_\Lb^{\perp,1}$, conjugation by (any representative of) $s_\alpha$ stabilises $\Lb$ and $\alpha$ extends to an algebraic character of $\Lb$ (see the proof of Lemma \ref{lemm:alpha}).

We let $\Pb^-$ denote the opposite parabolic subgroup.
Recall the parabolic induction functor $\Ind_{P^-}^G$ from the category of admissible smooth representations of $L$ over $k$ to the category of admissible smooth representations of $G$ over $k$, which is $k$-linear, fully faithful and exact (\cite{Em1}).
In particular, it induces a $k$-linear injection $\Ext_L^1 \hookrightarrow \Ext_G^1$.

Let $\sigma$ be an admissible smooth representation of $L$ over $k$.
For $\alpha \in \Delta_\Lb^{\perp,1}$, we consider the admissible smooth representation $\sigma^\alpha \otimes (\omega^{-1} \circ \alpha)$ of $L$ over $k$ where $\sigma^\alpha$ is the \mbox{$s_\alpha$-conjugate} of $\sigma$ and $\omega : F^\times \to \Fp^\times \subseteq k^\times$ is the mod $p$ cyclotomic character.
We say that $\sigma$ is \emph{supersingular} if it is absolutely irreducible and $\Fpbar \otimes_k \sigma$ is supersingular (\cite{AHHV}).

In \cite{JHB}, we formulated the following conjecture.
In cases (iii) and (iv), `Otherwise' means that the conditions of case (ii) are not all satisfied.

\begin{conj}[{\cite[Conjecture 3.17]{JHB}}] \label{intro:conj:Ext1}
Assume $\Gb$ split with connected centre and simply connected derived subgroup.
Let $\Pb = \Lb \Nb,\Pb' = \Lb' \Nb'$ be standard parabolic subgroups and $\sigma,\sigma'$ be supersingular representations of $L,L'$ respectively over $k$.
Assume $\Ind_{P^-}^G \sigma, \Ind_{P'^-}^G \sigma'$ irreducible or $p \neq 2$.
\begin{enumerate}
\item If $\Pb' \not \subseteq \Pb$ and $\Pb \not \subseteq \Pb'$, then $\Ext_G^1 \left( \Ind_{P'^-}^G \sigma', \Ind_{P^-}^G \sigma \right) = 0$.
\item If $F=\Qp$, $\Pb'=\Pb$ and $\sigma' \cong \sigma^\alpha \otimes \left( \omega^{-1} \circ \alpha \right) \not \cong \sigma$ for some $\alpha \in \Delta_\Lb^\perp$, then
\begin{equation*}
\dim_k \Ext_G^1 \left( \Ind_{P^-}^G \sigma', \Ind_{P^-}^G \sigma \right) = 1.
\end{equation*}
\item Otherwise if $\Pb' \subseteq \Pb$, then the functor $\Ind_{P^-}^G$ induces a $k$-linear isomorphism
\begin{equation*}
\Ext_L^1 \left( \Ind_{L \cap P'^-}^L \sigma', \sigma \right) \iso \Ext_G^1 \left( \Ind_{P'^-}^G \sigma', \Ind_{P^-}^G \sigma \right).
\end{equation*}
\item Otherwise if $\Pb \subseteq \Pb'$, then the functor $\Ind_{P'^-}^G$ induces a $k$-linear isomorphism
\begin{equation*}
\Ext_{L'}^1 \left( \sigma', \Ind_{L' \cap P^-}^{L'} \sigma \right) \iso \Ext_G^1 \left( \Ind_{P'^-}^G \sigma', \Ind_{P^-}^G \sigma \right).
\end{equation*}
\end{enumerate}
\end{conj}

We prove cases (ii), (iii) and (iv) of this conjecture and give some strong evidence for case (i).
We actually work without any assumption on $\Gb$ and our results hold true for broader classes of representations (see §~\ref{ssec:Ext} for more precise statements).
We also prove similar results for unitary continuous $p$-adic representations (see §~\ref{ssec:Extcont}).

We treat the cases $F=\Qp$ and $F \neq \Qp$ separately.
They are in fact the degree $1$ case of a more general (but conditional to a conjecture of Emerton) result on the $k$-linear morphism $\Ext_L^n \to \Ext_G^n$ induced by $\Ind_{P^-}^G$ in all degrees $n \leq [F:\Qp]$ (see Remark \ref{rema:Extn}).

\begin{theo}[Theorem \ref{theo:Ext1Qp}] \label{intro:theo:Ext1Qp}
Assume $F=\Qp$.
Let $\Pb = \Lb \Nb,\Pb' = \Lb' \Nb'$ be standard parabolic subgroups and $\sigma,\sigma'$ be supersingular representations of $L,L'$ respectively over $k$.
\begin{enumerate}
\item If $\Pb'=\Pb$ and $\sigma' \not \cong \sigma^\alpha \otimes (\omega^{-1} \circ \alpha)$ for all $\alpha \in \Delta_\Lb^{\perp,1}$, then the functor $\Ind_{P^-}^G$ induces a $k$-linear isomorphism
\begin{equation*}
\Ext_L^1 \left( \sigma', \sigma \right) \iso \Ext_G^1 \left( \Ind_{P^-}^G \sigma', \Ind_{P^-}^G \sigma \right).
\end{equation*}
\item If $\Pb' \subsetneq \Pb$, then the functor $\Ind_{P^-}^G$ induces a $k$-linear isomorphism
\begin{equation*}
\Ext_L^1 \left( \Ind_{L \cap P'^-}^L \sigma', \sigma \right) \iso \Ext_G^1 \left( \Ind_{P'^-}^G \sigma', \Ind_{P^-}^G \sigma \right).
\end{equation*}
\item If $\Pb \subsetneq \Pb'$, then the functor $\Ind_{P'^-}^G$ induces a $k$-linear isomorphism
\begin{equation*}
\Ext_{L'}^1 \left( \sigma', \Ind_{L' \cap P^-}^{L'} \sigma \right) \iso \Ext_G^1 \left( \Ind_{P'^-}^G \sigma', \Ind_{P^-}^G \sigma \right).
\end{equation*}
\end{enumerate}
\end{theo}

If $\Pb'=\Pb$, we do not know the exact dimension of the cokernel of the $k$-linear injection $\Ext_L^1(\sigma',\sigma) \hookrightarrow \Ext_G^1(\Ind_{P^-}^G \sigma',\Ind_{P^-}^G \sigma)$ induced by $\Ind_{P^-}^G$ in general, but we prove that it is at most $\card\{\alpha \in \Delta_\Lb^{\perp,1} \mid \sigma' \cong \sigma^\alpha \otimes (\omega^{-1} \circ \alpha)\}$ (see Remark \ref{rema:Ext1Qp} for more details).
Further, we compute it when $\Gb$ is split with connected centre (see Theorem \ref{intro:theo:Ext1Zcnx} below). 
Note that in cases (ii) and (iii), the source of the isomorphism can be non-zero (\cite{Hu}).

\begin{theo}[Theorem \ref{theo:Ext1F}] \label{intro:theo:Ext1F}
Assume $F \neq \Qp$.
Let $\Pb = \Lb \Nb$ be a standard parabolic subgroup.
The functor $\Ind_{P^-}^G$ induces a $k$-linear isomorphism
\begin{equation*}
\Ext_L^1 \left( \sigma', \sigma \right) \iso \Ext_G^1 \left( \Ind_{P^-}^G \sigma', \Ind_{P^-}^G \sigma \right)
\end{equation*}
for all admissible smooth representations $\sigma,\sigma'$ of $L$ over $k$.
\end{theo}

In particular, Theorem \ref{intro:theo:Ext1Qp} (ii) and (iii) hold true for any admissible smooth representations $\sigma,\sigma'$ of $L,L'$ respectively over $k$ when $F \neq \Qp$ (see Corollary \ref{coro:Ext1F}).

We complete Theorem \ref{intro:theo:Ext1Qp} (i) when $\Gb$ is split with connected centre (see also Remark \ref{rema:ExtnZcnx} for a more general, but conditional to a conjecture of Emerton, result on the $k$-linear morphism $\Ext_L^{[F:\Qp]} \to \Ext_G^{[F:\Qp]}$ induced by $\Ind_{P^-}^G$).

\begin{theo}[Theorem \ref{theo:Ext1Zcnx}] \label{intro:theo:Ext1Zcnx}
Assume $F=\Qp$ and $\Gb$ split with connected centre.
Let $\Pb = \Lb \Nb$ be a standard parabolic subgroup and $\sigma,\sigma'$ be supersingular representations of $L$ over $k$.
\begin{enumerate}
\item If $\sigma' \cong \sigma^\alpha \otimes (\omega^{-1} \circ \alpha) \not \cong \sigma$ for some $\alpha \in \Delta_\Lb^\perp$, then $\Ext_L^1(\sigma',\sigma)=0$ and
\begin{equation*}
\dim_k \Ext_G^1 \left( \Ind_{P^-}^G \sigma', \Ind_{P^-}^G \sigma \right) = 1.
\end{equation*}
\item If either $\sigma' \cong \sigma$ and $p \neq 2$, or $\sigma' \not \cong \sigma^\alpha \otimes (\omega^{-1} \circ \alpha)$ for any $\alpha \in \Delta_\Lb^\perp$, then the functor $\Ind_{P^-}^G$ induces a $k$-linear isomorphism
\begin{equation*}
\Ext_L^1 \left( \sigma', \sigma \right) \iso \Ext_G^1 \left( \Ind_{P^-}^G \sigma', \Ind_{P^-}^G \sigma \right).
\end{equation*}
\item If $p=2$, then the functor $\Ind_{P^-}^G$ induces a $k$-linear injection
\begin{equation*}
\Ext_L^1 \left( \sigma', \sigma \right) \hookrightarrow \Ext_G^1 \left( \Ind_{P^-}^G \sigma', \Ind_{P^-}^G \sigma \right)
\end{equation*}
whose cokernel is of dimension $\card \{ \alpha \in \Delta_\Lb^\perp \mid \sigma' \cong \sigma^\alpha \}$.
\end{enumerate}
\end{theo}

Finally, we treat the case where there is no inclusion between the two parabolic subgroups, assuming a special case of Conjecture \ref{intro:conj:HOrd} below (see also Remark \ref{rema:IwJ=1}).

\begin{prop}[Proposition \ref{prop:Ext1}]
Let $\Pb = \Lb \Nb,\Pb' = \Lb' \Nb'$ be standard parabolic subgroups and $\sigma,\sigma'$ be supersingular representations of $L,L'$ respectively over $k$.
Assume Conjecture \ref{intro:conj:HOrd} is true for $A=k$, $n=1$ and $\Iw^J=1$.
If $\Pb' \not \subseteq \Pb$ and $\Pb \not \subseteq \Pb'$, then
\begin{equation*}
\Ext_G^1 \left( \Ind_{P'^-}^G \sigma', \Ind_{P^-}^G \sigma \right) = 0.
\end{equation*}
\end{prop}

As a consequence, Conjecture \ref{intro:conj:Ext1} is true under the same assumption when $\Gb$ is split with connected centre (without assuming the derived subgroup of $\Gb$ simply connected).

\begin{coro}[Corollary \ref{coro:Ext1Zcnx}]
Assume $\Gb$ split with connected centre.
If Conjecture \ref{intro:conj:HOrd} is true for $A=k$, $n=1$ and $\Iw^J=1$, then Conjecture \ref{intro:conj:Ext1} is true.
\end{coro}

\subsection*{Strategy of proof and methods used}

Let $E/\Qp$ be a finite extension with ring of integers $\Oc$ and residue field $k$.
We work more generally with smooth representations over an Artinian local $\Oc$-algebra $A$ with residue field $k$.

The main tools to compute extensions between parabolically induced representations are two exact sequences related to Emerton's $\delta$-functor of derived ordinary parts (see below \eqref{intro:SEExt1} which is due to Emerton and \eqref{intro:SEExt2} which is a new feature of this article).

Using these, most of the previous results reduce to computing the derived ordinary parts of parabolically induced representations.
We formulate a conjecture on these computations (see Conjecture \ref{intro:conj:HOrd} below).
We prove it in low degree (see Theorem \ref{intro:theo:HOrd} below) and give some strong evidence for it in general.

We proceed in two steps: first we construct filtrations of parabolically induced representations related to the Bruhat decomposition; second we partially compute the derived ordinary parts of the associated graded representations using some dévissages.

\subsubsection*{Derived ordinary parts and extensions}

Let $\Pb \subseteq \Gb$ be a parabolic subgroup and $\Lb \subseteq \Pb$ be a Levi factor.
We let $\Pb^- \subseteq \Gb$ denote the parabolic subgroup opposed to $\Pb$ with respect to $\Lb$.
Emerton (\cite{Em1,Em2}) constructed a cohomological $\delta$-functor $\HOrd_P$ from the category of admissible smooth representations of $G$ over $A$ to the category of admissible smooth representations of $L$ over $A$, which is the right adjoint functor $\Ord_P$ of $\Ind_{P^-}^G$ in degree $0$.
From this, he derived a natural exact sequence of $A$-modules
\begin{equation} \label{intro:SEExt1}
0 \to \Ext_L^1 \left( \sigma, \Ord_P \pi \right) \to \Ext_G^1 \left( \Ind_{P^-}^G \sigma, \pi \right) \to \Hom_L \left( \sigma, \HOrd[1]_P \pi \right)
\end{equation}
for all admissible smooth representations $\sigma$ and $\pi$ of $L$ and $G$ respectively over $A$.

In §~\ref{ssec:second}, we construct a second exact sequence in which parabolic induction is on the right.
The construction is somewhat dual to that of \eqref{intro:SEExt1} but not exactly (see Remark \ref{rema:dual} (ii)).
We let $d$ denote the integer $\dim_F \Nb$ and $\delta$ denote the algebraic character of the adjoint representation of $\Lb$ on $\det_F (\Lie \Nb)$.
The key fact is that the $A$-linear functors
\begin{equation*}
\Hh\! \left( N, - \right) \coloneqq \HOrd[{[F:\Qp]} d-\bullet]_P \otimes \left( \omega \circ \delta \right).
\end{equation*}
form a homological $\delta$-functor from the category of admissible smooth representations of $G$ over $A$ to the category of admissible smooth representations of $L$ over $A$, which is isomorphic to the left adjoint functor $(-)_N$ of $\Ind_P^G$ in degree $0$ (hence the notation).
From this and using a result of Oort (\cite{Oort}) to compute extensions using pro-categories (see §~\ref{ssec:pro}), we derive a natural exact sequence of $A$-modules
\begin{equation} \label{intro:SEExt2}
0 \to \Ext_L^1 \left( \pi_N, \sigma \right) \to \Ext_G^1 \left( \pi, \Ind_P^G \sigma \right) \to \Hom_L \left( \Hh[1]\! \left( N, \pi \right), \sigma \right)
\end{equation}
for all admissible smooth representations $\pi$ and $\sigma$ of $G$ and $L$ respectively over $A$.

\subsubsection*{Computation of derived ordinary parts}

We let $W$ be the Weyl group of $(\Gb,\Sb)$.
For $I \subseteq \Delta$, we write $\Pb_I = \Lb_I \Nb_I$ for the corresponding standard parabolic subgroup, $\Bb_I \subseteq \Lb_I$ for the minimal parabolic subgroup $\Bb \cap \Lb_I$ and $W_I \subseteq W$ for the subgroup generated by $(s_\alpha)_{\alpha \in I}$.

Let $I,J \subseteq \Delta$, $\sigma$ be a locally admissible smooth representation of $L_I$ over $A$ and $n \in \N$.
We intend to compute the smooth representation of $L_J$ over $A$
\begin{equation*}
\HOrd[n]_{P_J} \left( \Ind_{P_I^-}^G \sigma \right).
\end{equation*}

In §~\ref{ssec:fil}, we use the generalised Bruhat decomposition $G = \bigsqcup_{\Iw^J \in \IW^J} P_I^- \Iw^J P_J$ where $\IW^J$ is the system of representatives of minimal length of the double cosets $W_I \backslash W / W_J$ (see §~\ref{ssec:double}) to define a natural filtration $\Fil_{P_J}^\bullet(\Ind_{P_I^-}^G \sigma)$ of $\Ind_{P_I^-}^G \sigma$ by $A[P_J]$-submodules indexed by $\IW^J$ (with the Bruhat order).
We also adapt the notion of graded representation associated to such a filtration (in particular, the grading has values in $\IW^J$) and we prove that for all $\Iw^J \in \IW^J$, there is a natural $A[P_J]$-linear isomorphism
\begin{equation*}
\Gr_{P_J}^{\Iw^J} \left( \Ind_{P_I^-}^G \sigma \right) \cong \cind_{P_I^-}^{P_I^- \Iw^J P_J} \sigma.
\end{equation*}
We prove that $\Fil_{P_J}^\bullet(\Ind_{P_I^-}^G \sigma)$ induces a filtration of $\HOrd[n]_{P_J} (\Ind_{P_I^-}^G \sigma)$ by $A[L_J]$-sub\-mod\-ules indexed by $\IW^J$ (see Proposition \ref{prop:FilPJ}).

Finally, we associate to each $\Iw^J \in \IW^J$ an integer $d_{\Iw^J}$ and an algebraic character $\delta_{\Iw^J}$ of $\Lb_{J \cap \Iw^J{}^{-1}(I)}$ (see Notation \ref{nota:dw} and Remark \ref{rema:dIwJ}), and we formulate the following conjecture.

\begin{conj}[Conjecture \ref{conj:HOrd}] \label{intro:conj:HOrd}
Let $\sigma$ be a locally admissible smooth representation of $L_I$ over $A$, $\Iw^J \in \IW^J$ and $n \in \N$.
There is a natural $A[L_J]$-linear isomorphism
\begin{multline*}
\HOrd[n]_{P_J} \left( \cind_{P_I^-}^{P_I^- \Iw^J P_J} \sigma \right) \\
\cong \Ind_{L_J \cap P_{J \cap \Iw^J{}^{-1}(I)}^-}^{L_J} \left( \left( \HOrd[n-{[F:\Qp]} d_{\Iw^J}]_{L_I \cap P_{I \cap \Iw^J(J)}} \sigma \right)^{\Iw^J} \otimes \left( \omega^{-1} \circ \delta_{\Iw^J} \right) \right).
\end{multline*}
\end{conj}

We give some strong evidence for this conjecture (see Theorem \ref{theo:HOrd}): we prove that these two representations have natural filtrations by $A[B_J]$-submodules indexed by $\JIW_J$ (the system of representatives of minimal length of the right cosets $W_{J \cap \Iw^J{}^{-1}(I)} \backslash W_J$) such that the associated graded representations are naturally isomorphic (see the subsection below).

We prove this conjecture in several cases (see Proposition \ref{prop:HOrd1}): whenever the right-hand side is either zero or a trivially induced representation, in which cases the aforementioned filtrations of both sides are trivial; when $n=0$, in which case we deduce the result from the computation of $\Ord_{P_J}(\Ind_{P_I^-}^G \sigma)$ in \cite{AHV}.
This allows us to compute $\HOrd_{P_J}(\Ind_{P_I^-}^G \sigma)$ in low degree when there is an inclusion between $I$ and $J$ (see Proposition \ref{prop:HOrd2}).
In particular, we obtain the following result in the case $I=J$.

\begin{theo}[Corollary \ref{coro:HOrd}] \label{intro:theo:HOrd}
Let $\Pb = \Lb \Nb$ be a standard parabolic subgroup and $\sigma$ be a locally admissible smooth representation of $L$ over $A$.
\begin{enumerate}
\item For all $n \in \N$ such that $0<n<[F:\Qp]$, we have $\HOrd[n]_P (\Ind_{P^-}^G \sigma) = 0$.
\item If $\Ord_{L \cap s_\alpha P s_\alpha^{-1}} \sigma = 0$ for all $\alpha \in \Delta^1 \backslash (\Delta_\Lb \cup \Delta_\Lb^\perp)$, then there is a natural $A[L]$-linear isomorphism
\begin{equation*}
\HOrd[{[F:\Qp]}]_P \left( \Ind_{P^-}^G \sigma \right) \cong \bigoplus_{\alpha \in \Delta_\Lb^{\perp,1}} \sigma^\alpha \otimes \left( \omega^{-1} \circ \alpha \right).
\end{equation*}
\end{enumerate}
\end{theo}

Note that for all $\alpha \in \Delta \backslash \Delta_\Lb$, $\Lb \cap s_\alpha \Pb s_\alpha^{-1}$ is the standard parabolic subgroup of $\Lb$ corresponding to $\Delta_\Lb \cap s_\alpha (\Delta_\Lb)$ and it is proper if and only if $\alpha \not \in \Delta_\Lb^\perp$.
In particular, the condition in (ii) is satisfied when $\sigma$ is supersingular.

In §~\ref{ssec:HJ}, we adapt the previous results in order to partially compute $\Hh[n](N_J,\Ind_{P_I}^G \sigma)$.
In particular, we obtain an analogue of Theorem \ref{intro:theo:HOrd} (see Corollary \ref{coro:HJ}).

\subsubsection*{Filtrations and dévissages}

Let $\Iw^J \in \IW^J$.
We explain the partial computation of the smooth representation of $L_J$ over $A$
\begin{equation*}
\HOrd[n]_{P_J} \left( \cind_{P_I^-}^{P_I^- \Iw^J P_J} \sigma \right).
\end{equation*}

In §~\ref{ssec:fil}, we use again the Bruhat decomposition to construct a natural filtration $\Fil_B^\bullet(\cind_{P_I^-}^{P_I^- \Iw^J P_J} \sigma)$ by $A[B]$-submodules indexed by $\JIW_J$, and we prove that for all $w_J \in \JIW_J$ there is a natural $A[B]$-linear isomorphism
\begin{equation*}
\Gr_B^{w_J} \left( \cind_{P_I^-}^{P_I^- \Iw^J P_J} \sigma \right) \cong \cind_{P_I^-}^{P_I^- \Iw^J w_J B} \sigma.
\end{equation*}
We prove that $\Fil_B^\bullet(\cind_{P_I^-}^{P_I^- \Iw^J P_J} \sigma)$ induces a filtration of $\HOrd[n]_{P_J}(\cind_{P_I^-}^{P_I^- \Iw^J P_J} \sigma)$ by $A[B_J]$-submodules indexed by $\JIW_J$ (see Proposition \ref{prop:FilB}).
Likewise, we construct a natural filtration $\Fil_{B_J}^\bullet(\Ind_{L_J \cap P_{J \cap \Iw^J{}^{-1}(I)}^-}^{L_J} \sigmat)$ by $A[B_J]$-submodules indexed by $\JIW_J$ for any smooth representation $\sigmat$ of $L_{J \cap \Iw^J{}^{-1}(I)}$ over $A$.

Let $w_J \in \JIW_J$ and set $\Iw \coloneqq \Iw^J w_J$ and $\pi_{\Iw} \coloneqq \cind_{P_I^-}^{P_I^- \Iw B} \sigma$.
We want to compute the $A$-module $\Hc[n](N_{J,0},\pi_{\Iw})$ endowed with the Hecke action of $B_J^+$ (see §~\ref{ssec:dev}), where $N_{J,0} \subseteq N_J$ is a compact open subgroup and $B_J^+ \subseteq B_J$ is the open submonoid stabilising $N_0$ by conjugation (we use similar notation for subgroups of $\Nb_J$ and $\Bb_J$ by taking intersections with $N_{J,0}$ and $B_J^+$ respectively).

In §~\ref{ssec:isocind}, we define closed subgroups $\Nb_{J,\Iw} \subseteq \Nb_J$ and $\Bb_{J,w_J} \subseteq \Bb_J$ such that there is a semidirect product $\Bb_{J,w_J} \ltimes \Nb_{J,\Iw}$ and we give an explicit description of the actions of $N_{J,\Iw}$ and $B_{J,w_J}$ on $\pi_{\Iw}$ for all $w_J \in \JIW_J$.
Then, we compute the $A$-module $\Hc[n](N_{J,\Iw,0},\pi_{\Iw})$ with the Hecke action of $B_{J,w_J}^+$ (see Proposition \ref{prop:calc}).

The idea is to use a semidirect product $\Nb_{J,\Iw} = \Nb''_{J,\Iw} \ltimes \Nb'_{J,\Iw}$ (also defined in §~\ref{ssec:isocind}) where $\Nb'_{J,\Iw} \subseteq \Nb_{J,\Iw}$ is a closed subgroup stable under conjugation by $\Bb_{J,w_J}$ such that $\pi_{\Iw}$ is ${N'_{J,\Iw,0}}$-acyclic and there is an $A[B_{J,w_J}^+]$-linear surjection with a locally nilpotent kernel from $\pi_{\Iw}^{N'_{J,\Iw,0}}$ onto $\Gr_{B_J}^{w_J}(\Ind_{L_J \cap P_{J \cap \Iw^J{}^{-1}(I)}^-}^{L_J}(\sigma_{|L_{I \cap \Iw^J(J)}})^{\Iw^J})$.
Then, taking the $N''_{J,\Iw,0}$-cohomology changes $\sigma_{|L_{I \cap \Iw^J(J)}}$ into $\HOrd[n]_{L_I \cap P_{I \cap \Iw^J(J)}} \sigma$ in the target and the inflation map is an $A[B_{J,w_J}^+]$-linear isomorphism between the source and $\Hc[n](N_{J,\Iw,0},\pi_{\Iw})$.

Finally, by a technical result on dévissages (see Proposition \ref{prop:dev}) and a finiteness property of the $A$-modules $\Hc(N_{J,\Iw,0},\pi_{\Iw})$, we can compute the $A$-module $\Hc[n](N_{J,0},\pi_{\Iw})$ with the Hecke action of $B_{J,w_J}^+$ from $\Hc[n](N_{J,\Iw,0},\pi_{\Iw})$.
It is this dévissage that introduces the degree shift and the twist (i.e. $d_{\Iw^J}$ and $\delta_{\Iw^J}$) in the formulas.

\subsection*{Notation and terminology}

Let $F/\Qp$ be a finite extension.
A linear algebraic $F$-group will be denoted by a boldface letter like $\Hb$ and the group of its $F$-points $\Hb(F)$ will be denoted by the corresponding ordinary letter $H$.
We will also write $\Hb^\der$ for its derived subgroup and $\Hb^\circ$ for its identity component.
The group of algebraic characters of $\Hb$ will be denoted by $\X^*(\Hb)$, the group of algebraic cocharacters of $\Hb$ will be denoted by $\X_*(\Hb)$, and we will write $\langle-,-\rangle : \X^*(\Hb) \times \X_*(\Hb) \to \Z$ for the natural pairing.
We now turn to reductive groups.
The main reference for these is \cite{BT}.

Let $\Gb$ be a connected reductive algebraic $F$-group.
We write $\Zb$ for the centre of $\Gb$.
Let $\Sb \subseteq \Gb$ be a maximal split torus.
We write $\Zcb$ (resp. $\Ncb$) for the centraliser (resp. normaliser) of $\Sb$ in $\Gb$ and $W$ for the Weyl group $\Ncb/\Zcb=\Nc/\Zc$.
We write $\Phi \subseteq \X^*(\Sb)$ for the set of roots of $\Sb$ in $\Gb$ and $\Phi_0 \subseteq \Phi$ for the subset of reduced roots.
To each $\alpha \in \Phi$ correspond a coroot $\alpha^\vee \in \X_*(\Sb)$, a reflection $s_\alpha \in W$ and a root subgroup $\Ub_\alpha \subset \Gb$ (which is denoted by $\Ub_{(\alpha)}$ in \cite{BT}).
For $\alpha,\beta \in \Phi$, we write $\alpha \perp \beta$ if and only if $\langle \alpha,\beta^\vee \rangle = 0$.
For $I \subseteq \Delta$, we put $I^\perp \coloneqq \{ \alpha \in \Delta \mid \alpha \perp \beta \ \forall \beta \in I \}$.

Let $\Bb \subseteq \Gb$ be a minimal parabolic subgroup containing $\Sb$.
We write $\Ub$ for the unipotent radical of $\Bb$ (so that $\Bb=\Zcb\Ub$), $\Phi^+ \subseteq \Phi$ for the subset of roots of $\Sb$ in $\Ub$ and $\Delta \subseteq \Phi^+$ for the subset of simple roots.
We set $\Phi_0^+ \coloneqq \Phi_0 \cap \Phi^+$.
A simple reflection is a reflection $s_\alpha \in W$ with $\alpha \in \Delta$.
A reduced decomposition of $w \in W$ is any decomposition into simple reflections $w=s_1 \dots s_n$ with $n \in \N$ minimal, which is called the length of $w$ and denoted by $\ell(w)$.
We write $w_0$ for the element of maximal length in $W$.

We say that $\Pb = \Lb \Nb$ is a standard parabolic subgroup if $\Pb \subseteq \Gb$ is a parabolic subgroup containing $\Bb$ with unipotent radical $\Nb$ and $\Lb \subseteq \Pb$ is the Levi factor containing $\Sb$ (we say that $\Lb$ is a standard Levi subgroup).
In this case, we write $\Pb^-$ for the parabolic subgroup of $\Gb$ opposed to $\Pb$ with respect to $\Lb$ (i.e. $\Pb \cap \Pb^-=\Lb$) and $\Nb^-$ for the unipotent radical of $\Pb^-$.
We write $\Zb_\Lb$ for the centre of $\Lb$, $\Bb_\Lb \subseteq \Lb$ for the minimal parabolic subgroup $\Bb \cap \Lb$, $\Ub_\Lb \subseteq \Bb_\Lb$ for the unipotent radical $\Ub \cap \Lb$ (so that $\Bb_\Lb=\Zcb\Ub_\Lb$ and $\Ub = \Ub_\Lb \ltimes \Nb$) and $\Delta_\Lb \subseteq \Delta$ for the subset of simple roots of $\Sb$ in $\Ub_\Lb$.

Each parabolic subgroup of $\Gb$ is conjugate to exactly one standard parabolic subgroup and the map $\Pb = \Lb \Nb \mapsto \Delta_\Lb$ yields a bijection between standard parabolic subgroups of $\Gb$ and subsets of $\Delta$.
For $I \subseteq \Delta$, we write $\Pb_I = \Lb_I \Nb_I$ for the corresponding standard parabolic subgroup (i.e. $\Delta_{\Lb_I}=I$), $\Zb_I$, $\Bb_I$, $\Ub_I$ instead of $\Zb_{\Lb_I}$, $\Bb_{\Lb_I}$, $\Ub_{\Lb_I}$ respectively, $W_I \subseteq W$ for the subgroup generated by $(s_\alpha)_{\alpha \in I}$ (so that $\Pb_I = \Bb W_I \Bb$), $w_{I,0}$ for the element of maximal length in $W_I$, $\Phi_I \subseteq \Phi$ for the subset of roots of $\Sb$ in $\Lb_I$ and $\Phi_I^+\subseteq \Phi^+$ for the subset of roots of $\Sb$ in $\Ub_I$.

\medskip

Let $E/\Qp$ be a finite extension with ring of integers $\Oc$ and residue field $k$.
We let $A$ be an Artinian local $\Oc$-algebra with residue field $k$.
We write $\varepsilon : F^\times \to \Zp^\times \subseteq \Oc^\times$ for the $p$-adic cyclotomic character (defined by $\varepsilon(x) = \Nrm_{F/\Qp}(x) \vert \Nrm_{F/\Qp}(x) \vert_p$ for all $x \in F^\times$) and $\omega : F^\times \to A^\times$ for its image in $A^\times$.

We use the terminology and notation of \cite[§~2]{Em1} for representations of a $p$-adic Lie group $H$ over $A$.
An $H$-representation is a smooth representation of $H$ over $A$ and a morphism between $H$-representations is $A$-linear.
We write $\Mod_H^\sm(A)$ for the category of $H$-representations and $H$-equivariant morphisms, and $\Mod_H^\adm(A)$ (resp. $\Mod_H^\ladm(A)$, $\Mod_H^\sm(A)^{Z_H-\lfin}$) for the full subcategory of admissible (resp. locally admissible, locally $Z_H$-finite) $H$-representations (here $Z_H$ denotes the centre of $H$).

Assume $H \subseteq G$ is closed and $\pi$ is an $H$-representation.
For $g \in G$, we write $\pi^g$ for the $g^{-1}Hg$-representation with the same underlying $A$-module as $\pi$ on which $g^{-1}hg$ acts as $h$ for all $h \in H$.
If $g \in H$, then $g^{-1}Hg=H$ and the action of $g$ on $\pi$ induces a natural $H$-equivariant isomorphism $\pi \iso \pi^g$.

Assume furthermore $\Zc \subseteq H$.
For $w \in W$, we write $\pi^w$ for the $n^{-1}Hn$-representation $\pi^n$ where $n \in \Nc$ is any representative of $w$ (neither $n^{-1}Hn$ nor $\pi^n$ depend on the choice of $n$ up to isomorphism).
For $\alpha \in \Delta$, we simply write $\pi^\alpha$ instead of $\pi^{s_\alpha}$.

For a topological space $X$ and an $A$-module $V$, we write $\Clis(X,V)$ for the $A$-module of locally constant functions $f : X \to V$ and $\Clisc(X,V)$ for the $A$-submodule consisting of those functions with compact support (the support of $f$ is the open and closed subset $\supp f \coloneqq f^{-1}(V \backslash \{0\}) \subseteq X$).

\subsection*{Acknowledgements}

I am grateful to C.~Breuil whose ideas motivated the present results, and for several comments on a preliminary version of this paper.
I thank N.~Abe for his help with the proof of Proposition \ref{prop:Ext1prelim}, as well as G.~Henniart, V.~Paškūnas and M.-F.~Vignéras for answering my questions.
I also thank M.-F.~Vignéras for her encouragement.
Finally, I want to thank the anonymous referee for his/her very careful reading of the manuscript.

\numberwithin{theo}{subsection}

\section{Generalised Bruhat filtrations}

The aim of this section is to define filtrations of parabolically induced representations and describe the associated graded representations.
In §~\ref{ssec:double}, we review some properties of the representatives of minimal length of certain double cosets in $W$ and some variants of the Bruhat decomposition.
In §~\ref{ssec:fil}, we define the notion of filtration indexed by a poset and we construct filtrations of induced representations indexed by subsets of $W$ with the Bruhat order using the previous decompositions.
In §~\ref{ssec:isocind}, we define several subgroups of $\Ub$ that we use to describe the graded representations associated to the previous filtrations as spaces of locally constant functions with compact support.

\subsection{Doubles cosets} \label{ssec:double}

First, we recall some facts about certain right cosets in $W$ (cf. \cite[Proposition 3.9]{BTC}).
For any $I \subseteq \Delta$, we define a system of representatives of the right cosets $W_I \backslash W$ by setting
\begin{equation*}
\IW \coloneqq \left\{ w \in W \middlevert w \text{ is of minimal length in } W_I w \right\}.
\end{equation*}
For all $w \in W$, there exists a unique decomposition $w = w_I \Iw$ with $w_I \in W_I$ and $\Iw \in \IW$.
This decomposition is characterised by the equality
\begin{equation*}
\Phi_I^+ \cap w \left( \Phi^+ \right) = \Phi_I^+ \cap w_I \left( \Phi_I^+ \right).
\end{equation*}
In particular, we have $\Iw^{-1}(\Phi_I^+) \subseteq \Phi^+$.
Furthermore, we have $\ell(w) = \ell(w_I) + \ell(\Iw)$.

We now recall some properties of certain double cosets in $W$ (cf. \cite[Lemma 5.4]{DM}).
For any $I,J \subseteq \Delta$, we define a system of representatives of the double cosets $W_I \backslash W / W_J$ by setting
\begin{equation*}
\IW^J \coloneqq \IW \cap \left( ^JW \right)^{-1}.
\end{equation*}
For all $\Iw \in \IW$, there exists a unique decomposition $\Iw = \Iw^J w_J$ with $\Iw^J \in \IW^J$ and $w_J \in W_J$.
In fact $w_J \in \JIW_J$.
This decomposition is characterised by the equality
\begin{equation} \label{KostantIJ}
\Phi_J^+ \cap \Iw^{-1} \left( \Phi^+ \right) = \Phi_J^+ \cap w_J^{-1} \left( \Phi_J^+ \right).
\end{equation}
In particular, we have $\Iw^J(\Phi_J^+) \subseteq \Phi^+$.
Furthermore, we have $\ell(\Iw) = \ell(\Iw^J) + \ell(w_J)$.
Conversely, for all $\Iw^J \in \IW^J$ and $w_J \in W_J$, we have $\Iw^J w_J \in \IW$ if and only if $w_J \in \JIW_J$.
Note that the projections $W \twoheadrightarrow \IW$ and $\IW \twoheadrightarrow \IW^J$ respect the Bruhat order\footnote{The Bruhat order on $W$ is defined by $w \leq w'$ if and only if there exist a reduced decomposition $w'=s_1 \dots s_{\ell(w')}$ and integers $1 \leq i_1 < \dots < i_{\ell(w)} \leq \ell(w')$ such that $w = s_{i_1} \dots s_{i_{\ell(w)}}$.} (cf. \cite[Proposition 2.5.1]{BB})

\begin{lemm} \label{lemm:IwJ}
We have the following equalities in $\Gb$.
\begin{enumerate}
\item $\Lb_I \cap \Iw^J \Ub_J \Iw^J{}^{-1} = \Ub_{I \cap \Iw^J(J)}$
\item $\Lb_I \cap \Iw^J \Lb_J \Iw^J{}^{-1} = \Lb_{I \cap \Iw^J(J)}$
\item $\Lb_I \cap \Iw^J \Nb_J \Iw^J{}^{-1} = \Lb_I \cap \Nb_{I \cap \Iw^J(J)}$
\item $\Lb_I \cap \Iw^J \Pb_J \Iw^J{}^{-1} = \Lb_I \cap \Pb_{I \cap \Iw^J(J)}$
\end{enumerate}
\end{lemm}

\begin{proof}
First, we prove the following equalities in $\Phi$:
\begin{gather}
\Phi_I \cap \Iw^J(J) = I \cap \Iw^J(J), \label{eq0} \\
\Phi_I \cap \Iw^J \left( \Phi_J^+ \right) = \Phi_{I \cap \Iw^J(J)}^+. \label{eqi}
\end{gather}
We prove the non-trivial inclusion of \eqref{eq0}.
Assume $\Phi_I \cap \Iw^J(J) \neq \emptyset$ and let $\alpha \in \Phi_I \cap \Iw^J(J)$.
Since $\Iw^J(J) \subseteq \Phi^+$, $\alpha \in \Phi_I^+$ so that there exists $(r_\beta)_{\beta \in I} \in \N^I$ such that $\alpha = \sum_{\beta \in I} r_\beta \beta$.
Then $\Iw^J{}^{-1}(\alpha) = \sum_{\beta \in I} r_\beta \Iw^J{}^{-1}(\beta) \in \Delta$.
Since $\Iw^J{}^{-1}(\beta) \in \Phi^+$ for all $\beta \in I$, $r_\beta=0$ for all $\beta \in I \backslash \{\alpha\}$ and $r_\alpha=1$.
Thus $\alpha \in I$.
We prove the non-trivial inclusion of \eqref{eqi}.
Assume $\Phi_I \cap \Iw^J (\Phi_J^+) \neq \emptyset$ and let $\alpha \in \Phi_I \cap \Iw^J (\Phi_J^+)$.
There exists $(r_\beta)_{\beta \in J} \in \N^J$ such that $\alpha = \sum_{\beta \in J} r_\beta \Iw^J(\beta)$.
Since $\Iw^J(\beta) \in \Phi^+$ for all $\beta \in J$, $\Iw^J(\beta) \in \Phi_I^+$ so that $\Iw^J(\beta) \in I$ by \eqref{eq0} for all $\beta \in J$ such that $r_\beta > 0$.
Thus $\alpha \in \Phi_{I \cap \Iw^J(J)}^+$.

Now, by considering the Lie algebras, \eqref{eqi} yields (i), \eqref{eqi} and its opposite yield (ii), the equality $\Phi_I \cap \Iw^J(\Phi^+ \backslash \Phi_J^+)=\Phi_I^+ \backslash \Phi_{I \cap \Iw^J(J)}^+$ (which follows from \eqref{eqi} and the fact that $\Phi_I \cap \Iw^J(\Phi^+)=\Phi_I^+$ since $\Iw^J \in \IW$) yields (iii), and we deduce (iv) from (ii) and (iii).
\end{proof}

Finally, we give certain decompositions in double cosets (for the notion of `lower set', see foonote \ref{lowerset} p.~\pageref{lowerset}).

\begin{lemm} \phantomsection \label{lemm:Bruhat}
\begin{enumerate}
\item We have $G = \bigsqcup_{\Iw^J \in \IW^J} P_I^- \Iw^J P_J$ and for any lower set $\IW^J_1 \subseteq \allowbreak \IW^J$, the subset $P_I^- \IW^J_1P_J \subseteq G$ is open.
\item We have $P_I^- \Iw^J P_J = \bigsqcup_{w_J \in \JIW_J} P_I^- \Iw^J w_J B$ and for any lower set $W'_J \subseteq \allowbreak \JIW_J$, the subset $P_I^- \Iw^J W'_J B \subseteq P_I^- \Iw^J P_J$ is open.
\item We have $L_J = \bigsqcup_{w_J \in \JIW_J} L_J \cap P_{J \cap \Iw^J{}^{-1}(I)}^- w_J B_J$ and for any lower set $W'_J \subseteq \allowbreak \JIW_J$, the subset $L_J \cap P_{J \cap \Iw^J{}^{-1}(I)}^- W'_J B_J \subseteq L_J$ is open.
\end{enumerate}
\end{lemm}

\begin{proof}
We have $G = \bigsqcup_{\Iw \in \IW} P_I^- \Iw B$ and for any $\Iw \in \IW$, the closure of $P_I^- \Iw B$ in $G$ is $\bigsqcup_{\Iw' \geq \Iw} P_I^- \Iw' B$ (this can be deduced from the Bruhat decomposition, cf. e.g. \cite[§~2.3]{JH}).
Furthermore, for any $\Iw^J \in \IW^J$ we have
\begin{align*}
P_I^- \Iw^J P_J &= \bigcup_{w_J \in W_J} P_I^- w_0 B w_0 \Iw^J w_{J,0} B w_J B \\
&= \bigcup_{w_J \in W_J} P_I^- w_0 B w_0 \Iw^J w_{J,0} w_J B \\
&= \bigcup_{w_J \in W_J} P_I^- \Iw^J w_J B
\end{align*}
(the first and third equalities follow from the inclusion $w_0Bw_0 = B^- \subseteq P_I^-$ and the decomposition $P_J = B W_J B$, and the second equality follows from \cite[Lemme 3.4 (iv)]{BTC} and \cite[Proposition 2.5.4]{BB}).
From this we deduce (ii), and also (i) using the fact that the projection $\IW \twoheadrightarrow \IW^J$ is order-preserving.
Finally, (iii) is (i) for the double cosets $L_J \cap P_{J \cap \Iw^J{}^{-1}(I)}^- \backslash L_J / B_J$ instead of $P_I^- \backslash G / P_J$.
\end{proof}

\begin{rema}[Case $w_J=1$] \label{rema:bigcell}
Note that $P_I^- \Iw^J B$ is $P_{J \cap \Iw^J{}^{-1}(I)}$-invariant by right translation.
In general, the stabiliser of $P_I^- \Iw B$ in $G$ for the action by right translation is the (non-standard) parabolic subgroup $B \Iw^{-1} W_I \Iw B$.
Likewise, $L_J \cap P_{J \cap \Iw^J{}^{-1}(I)}^- B_J$ is $L_J \cap P_{J \cap \Iw^J{}^{-1}(I)}$-invariant by right translation.
\end{rema}

\subsection{Definition of filtrations} \label{ssec:fil}

\subsubsection*{Filtration indexed by a poset}

Let $H$ be a $p$-adic Lie group, $\pi$ be an $H$-representation and $(\Wt,\leq)$ be a poset.
A filtration of $\pi$ indexed by $\Wt$ is a morphism of complete lattices $\Fil_H^\bullet \pi$ from the complete lattice of lower sets\footnote{A lower set of $\Wt$ is a subset $\Wt'$ such that $\wt \leq \wt' \Rightarrow \wt \in \Wt'$ for any $\wt \in \Wt$ and $\wt' \in \Wt'$. \label{lowerset}} of $\Wt$ to the complete lattice of $H$-subrepresentations of $\pi$, i.e. an $H$-subrepresentation $\Fil_H^{\Wt'} \pi \subseteq \pi$ for each lower set $\Wt' \subseteq \Wt$ such that for any family $(\Wt_i)_{i \in \Ic}$ of lower sets of $\Wt$, we have the following equalities in $\pi$:
\begin{gather*}
\Fil_H^{\bigcap_{i \in \Ic} \Wt_i} \pi = \bigcap_{i \in \Ic} \Fil_H^{\Wt_i} \pi, \\
\Fil_H^{\bigcup_{i \in \Ic} \Wt_i} \pi = \sum_{i \in \Ic} \Fil_H^{\Wt_i} \pi.
\end{gather*}
When $\Wt$ is finite, these two equalities are equivalent (by induction) to the following conditions: $\Fil_H^\bullet \pi$ is inclusion-preserving with $\Fil_H^\emptyset \pi = 0$ and $\Fil_H^{\Wt} \pi = \pi$ (i.e. the empty family case), and for any lower sets $\Wt_1, \Wt_2 \subseteq \Wt$ the short sequence of $H$-representations
\begin{equation*}
0 \to \Fil_H^{\Wt_1 \cap \Wt_2} \pi \to \Fil_H^{\Wt_1} \pi \oplus \Fil_H^{\Wt_2} \pi \to \Fil_H^{\Wt_1 \cup \Wt_2} \pi \to 0,
\end{equation*}
defined by $v \mapsto (v,-v)$ and $(v_1,v_2) \mapsto v_1+v_2$, is exact.

Each $\wt \in \Wt$ defines a principal lower set $\{\wt' \in \Wt \mid \wt' \leq \wt\}$ and we write $\Fil_H^{\wt} \pi$ for the corresponding $H$-subrepresentation of $\pi$.
Note that for any lower set $\Wt' \subseteq \Wt$, we have the following equality in $\pi$:
\begin{equation*}
\Fil_H^{\Wt'} \pi = \sum_{\wt' \in \Wt'} \Fil_H^{\wt'} \pi.
\end{equation*}
In particular, we can recover the whole filtration from the $H$-subrepresentations of $\pi$ corresponding to the elements of $\Wt$, hence the terminology.
We define the graded representation $\Gr_H^\bullet \pi$ associated to the filtration $\Fil_H^\bullet \pi$ by setting
\begin{equation*}
\Gr_H^{\wt} \pi \coloneqq \Fil_H^{\wt} \pi / \sum_{\wt' < \wt} \Fil_H^{\wt'} \pi
\end{equation*}
for each $\wt \in \Wt$.

Let $\ellt : \Wt \to \Z$ be a monotonic map (i.e. $\wt \leq \wt' \Rightarrow \ellt(\wt) \leq \ellt(\wt')$ for any $\wt,\wt' \in \Wt$).
For each $n \in \Z$, we set
\begin{equation*}
\Fil_H^{\ellt,n} \pi \coloneqq \sum_{\ellt(\wt) \leq n} \Fil_H^{\wt} \pi.
\end{equation*}
We obtain a filtration of $\pi$ indexed by $\Z$ (in the usual sense).

\begin{lemm} \label{lemm:lin}
Assume $\ellt : \Wt \to \Z$ is strictly monotonic (i.e. $\wt < \wt' \Rightarrow \ellt(\wt) < \ellt(\wt')$ for any $\wt,\wt' \in \Wt$).
For all $n \in \Z$, there is a natural $H$-equivariant isomorphism
\begin{equation*}
\Gr_H^{\ellt,n} \pi \cong \bigoplus_{\ellt(\wt)=n} \Gr_H^{\wt} \pi.
\end{equation*}
\end{lemm}

\begin{proof}
Let $n \in \Z$.
By definition of $\Fil_H^{\ellt,n} \pi$ and $\Gr_H^{\ellt,n} \pi$, there are natural $H$-equivariant surjections
\begin{equation} \label{projGr}
\bigoplus_{\ellt(\wt) \leq n} \Fil_H^{\wt} \pi \twoheadrightarrow \Fil_H^{\ellt,n} \pi \twoheadrightarrow \Gr_H^{\ellt,n} \pi.
\end{equation}
The kernel of \eqref{projGr} contains $\bigoplus_{\ellt(\wt) \leq n} \Fil_H^{\wt} \pi \cap \Fil_H^{\ellt,n-1} \pi$, and $\Fil_H^{\wt} \pi \cap \Fil_H^{\ellt,n-1} \pi = \Fil_H^{\wt} \pi$ for all $\wt \in \Wt$ such that $\ellt(\wt)<n$.
Now, for any $\wt_0 \in \Wt$ such that $\ellt(\wt_0)=n$, we have the following equality in $\pi$:
\begin{equation*}
\Fil_H^{\wt_0} \pi \cap \sum_{\underset{\wt \neq \wt_0}{\ellt(\wt) \leq n}} \Fil_H^{\wt} \pi = \sum_{\wt < \wt_0} \Fil_H^{\wt} \pi,
\end{equation*}
which results from the following equality in $\Wt$:
\begin{equation*}
\left\{ \wt' \in \Wt \middlevert \wt' \leq \wt_0 \right\} \cap \bigcup_{\underset{\wt \neq \wt_0}{\ellt(\wt) \leq n}} \left\{ \wt' \in \Wt \middlevert \wt' \leq \wt \right\} = \bigcup_{\wt<\wt_0} \left\{ \wt' \in \Wt \middlevert \wt' \leq \wt \right\},
\end{equation*}
which in turn follows from the fact that $\wt_0 \not \leq \wt$ for all $\wt \in \Wt \backslash \{\wt_0\}$ such that $\ellt(\wt) \leq n$ by strict monotonicity of $\ellt$.
We deduce that the kernel of \eqref{projGr} is $\bigoplus_{\ellt(\wt) \leq n} \Fil_H^{\wt} \pi \cap \Fil_H^{\ellt,n-1} \pi$, and that $\Fil_H^{\wt} \pi \cap \Fil_H^{\ellt,n-1} \pi = \sum_{\wt'<\wt} \Fil_H^{\wt'} \pi$ for all $\wt \in \Wt$ such that $\ellt(\wt)=n$.
We conclude that \eqref{projGr} induces an isomorphism as in the statement.
\end{proof}

\subsubsection*{Filtrations of induced representations}

Let $I,J \subseteq \Delta$ and $\sigma$ be an $L_I$-representation.
Recall that for any locally closed subset $X \subseteq G$ and for any open subset $Y \subseteq X$, both $P_I^-$-invariant by left translation, there is a natural short exact sequence of $A$-modules
\begin{equation*}
0 \to \cind_{P_I^-}^Y \sigma \to \cind_{P_I^-}^X \sigma \to \cind_{P_I^-}^{X \backslash Y} \sigma \to 0
\end{equation*}
(cf. \cite[Proposition 1.8]{BZ}, see also the proof of \cite[Proposition 2.1.3]{JH}).
Note that there is a natural $A$-linear isomorphism $\cind_{P_I^-}^G \sigma \iso \Ind_{P_I^-}^G \sigma$ since $P_I^- \backslash G$ is compact.

For each lower set $\IW^J_1 \subseteq \IW^J$, we define a $P_J$-subrepresentation of $\Ind_{P_I^-}^G \sigma$ by setting
\begin{equation*}
\Fil_{P_J}^{\IW^J_1} \left( \Ind_{P_I^-}^G \sigma \right) \coloneqq \cind_{P_I^-}^{P_I^- \IW^J_1 P_J} \sigma.
\end{equation*}
Using Lemma \ref{lemm:Bruhat} (i), we obtain a filtration of $\Ind_{P_I^-}^G \sigma$ indexed by $\IW^J$ such that for all $\Iw^J \in \IW^J$, there is a natural $P_J$-equivariant isomorphism
\begin{equation} \label{GrPJ}
\Gr_{P_J}^{\Iw^J} \left( \Ind_{P_I^-}^G \sigma \right) \cong \cind_{P_I^-}^{P_I^- \Iw^J P_J} \sigma.
\end{equation}

Let $\Iw^J \in \IW^J$.
For each lower set $W'_J \subseteq \JIW_J$, we define a $B$-sub\-rep\-res\-ent\-a\-tion of $\cind_{P_I^-}^{P_I^- \Iw^J P_J} \sigma$ by setting
\begin{equation*}
\Fil_B^{W'_J} \left( \cind_{P_I^-}^{P_I^- \Iw^J P_J} \sigma \right) \coloneqq \cind_{P_I^-}^{P_I^- \Iw^J W'_J B} \sigma.
\end{equation*}
Using Lemma \ref{lemm:Bruhat} (ii), we obtain a filtration of $\cind_{P_I^-}^{P_I^- \Iw^J P_J} \sigma$ indexed by $\JIW_J$ such that for all $w_J \in \JIW_J$, there is a natural $B$-equivariant isomorphism
\begin{equation} \label{GrB}
\Gr_B^{w_J} \left( \cind_{P_I^-}^{P_I^- \Iw^J P_J} \sigma \right) \cong \cind_{P_I^-}^{P_I^- \Iw^J w_J B} \sigma.
\end{equation}

Likewise, for any $L_{J \cap \Iw^J{}^{-1}(I)}$-representation $\sigmat$ and using Lemma \ref{lemm:Bruhat} (iii), we define for each lower set $W'_J \subseteq \JIW_J$ a $B_J$-subrepresentation of $\Ind_{L_J \cap P_{J \cap \Iw^J{}^{-1}(I)}^-}^{L_J} \sigmat$ by setting
\begin{equation*}
\Fil_{B_J}^{W'_J} \left( \Ind_{L_J \cap P_{J \cap \Iw^J{}^{-1}(I)}^-}^{L_J} \sigmat \right) \coloneqq \cind_{L_J \cap P_{J \cap \Iw^J{}^{-1}(I)}^-}^{L_J \cap P_{J \cap \Iw^J{}^{-1}(I)}^- W'_J B_J} \sigmat
\end{equation*}
and we obtain a filtration of $\Ind_{L_J \cap P_{J \cap \Iw^J{}^{-1}(I)}^-}^{L_J} \sigmat$ indexed by $\JIW_J$ such that for all $w_J \in \JIW_J$, there is a natural $B_J$-equivariant isomorphism
\begin{equation} \label{GrBJ}
\Gr_{B_J}^{w_J} \left( \Ind_{L_J \cap P_{J \cap \Iw^J{}^{-1}(I)}^-}^{L_J} \sigmat \right) \cong \cind_{L_J \cap P_{J \cap \Iw^J{}^{-1}(I)}^-}^{L_J \cap P_{J \cap \Iw^J{}^{-1}(I)}^- w_J B_J} \sigmat.
\end{equation}

\begin{rema}[Case $w_J=1$] \label{rema:Fil1}
Note that $\Gr_B^1 (\cind_{P_I^-}^{P_I^- \Iw^J P_J} \sigma) \cong \cind_{P_I^-}^{P_I^- \Iw^J B} \sigma$ is a $P_{J \cap \Iw^J{}^{-1}(I)}$-subrepresentation of $\cind_{P_I^-}^{P_I^- \Iw^J P_J} \sigma$ and likewise $\Gr_{B_J}^1 (\Ind_{L_J \cap P_{J \cap \Iw^J{}^{-1}(I)}^-}^{L_J} \sigmat) \cong \allowbreak \cind_{L_J \cap P_{J \cap \Iw^J{}^{-1}(I)}^-}^{L_J \cap P_{J \cap \Iw^J{}^{-1}(I)}^- B_J} \sigmat$ is an $L_J \cap P_{J \cap \Iw^J{}^{-1}(I)}$-subrepresentation of $\Ind_{L_J \cap P_{J \cap \Iw^J{}^{-1}(I)}^-}^{L_J} \sigmat$ (see Remark \ref{rema:bigcell}).
\end{rema}

\subsection{Computation of the associated graded representations} \label{ssec:isocind}

For each $w \in W$, we define a closed subgroup of $\Ub$ stable under conjugation by $\Zcb$ by setting
\begin{equation*}
\Ub_w \coloneqq \Ub \cap w^{-1} \Ub w
\end{equation*}
and we let $\Bb_w \subseteq \Bb$ be the closed subgroup $\Zcb \Ub_w$.
For any order on $\Phi^+ \cap w^{-1}(\Phi_0^+)$, the product induces an isomorphism of $F$-varieties
\begin{equation} \label{isoprod}
\prod_{\alpha \in \Phi^+ \cap w^{-1}(\Phi_0^+)} \Ub_\alpha \iso \Ub_w.
\end{equation}

\medskip

Let $I \subseteq \Delta$ and $\Iw \in \IW$.
We define closed subgroups of $\Ub_{\Iw}$ stable under conjugation by $\Zcb$ by setting
\begin{gather*}
\Ub'_{\Iw} \coloneqq \Ub \cap \Iw^{-1} \Nb_I \Iw \\
\Ub''_{\Iw} \coloneqq \Ub \cap \Iw^{-1} \Ub_I \Iw
\end{gather*}
and we let $\Bb''_{\Iw} \subseteq \Bb_{\Iw}$ be the closed subgroup $\Zcb \Ub''_{\Iw}$.
We have semidirect products $\Ub_{\Iw} = \Ub''_{\Iw} \ltimes \Ub'_{\Iw}$ and $\Bb_{\Iw} = \Bb''_{\Iw} \ltimes \Ub'_{\Iw}$.

Let $\sigma$ be an $L_I$-representation.
The product induces an isomorphism of $F$-varieties
\begin{equation*}
P_I^- \times \left\{ \Iw \right\} \times U'_{\Iw} \iso P_I^- \Iw B,
\end{equation*}
hence an $A$-linear isomorphism
\begin{equation} \label{isocind}
\cind_{P_I^-}^{P_I^- \Iw B} \sigma \cong \Clisc\! \left( U'_{\Iw}, \sigma^{\Iw} \right)
\end{equation}
via which $U'_{\Iw}$ acts on $\Clisc(U'_{\Iw},\sigma^{\Iw})$ by right translation and the action of $b'' \in B''_{\Iw}$ on $f \in \Clisc(U'_{\Iw},\sigma^{\Iw})$ is given by
\begin{equation*}
\left( b'' \cdot f \right) \left( u' \right) = b'' \cdot f \left( b''^{-1} u' b'' \right)
\end{equation*}
for all $u' \in U'_{\Iw}$.

\medskip

Let $J \subseteq \Delta$.
We write $\Iw = \Iw^J w_J$ with $\Iw^J \in \IW^J$ and $w_J \in W_J$.
We define closed subgroups of $\Nb_J$ and $\Ub_J$ stable under conjugation by $\Zcb$ by setting
\begin{gather*}
\Nb_{J,\Iw} \coloneqq \Nb_J \cap \Ub_{\Iw} = \Nb_J \cap \Iw^{-1} \Ub \Iw \\
\Ub_{J,w_J} \coloneqq \Ub_J \cap \Ub_{\Iw} = \Ub_J \cap \Iw^{-1} \Ub \Iw = \Ub_J \cap w_J^{-1} \Ub_J w_J
\end{gather*}
the last equality resulting from \eqref{KostantIJ}, and we let $\Bb_{J,w_J} \subseteq \Bb_J$ be the closed subgroup $\Zcb \Ub_{J,w_J}$.
We have semidirect products $\Ub_{\Iw} = \Ub_{J,w_J} \ltimes \Nb_{J,\Iw}$ and $\Bb_{\Iw} = \Bb_{J,w_J} \ltimes \Nb_{J,\Iw}$.
We define closed subgroups of $\Nb_{J,\Iw}$ and $\Ub_{J,w_J}$ stable under conjugation by $\Zcb$ by setting
\begin{gather*}
\Nb'_{J,\Iw} \coloneqq \Nb_J \cap \Ub'_{\Iw} = \Nb_J \cap \Iw^{-1} \Nb_I \Iw \\
\Nb''_{J,\Iw} \coloneqq \Nb_J \cap \Ub''_{\Iw} = \Nb_J \cap \Iw^{-1} \Ub_I \Iw \\
\Ub'_{J,w_J} \coloneqq \Ub_J \cap \Ub'_{\Iw} = \Ub_J \cap \Iw^{-1} \Nb_I \Iw \\
\Ub''_{J,w_J} \coloneqq \Ub_J \cap \Ub''_{\Iw} = \Ub_J \cap \Iw^{-1} \Ub_I \Iw
\end{gather*}
and we let $\Bb''_{J,w_J} \subseteq \Bb_J$ be the closed subgroup $\Zcb \Ub''_{J,w_J}$.
We have semidirect products $\Nb_{J,\Iw} = \Nb''_{J,\Iw} \ltimes \Nb'_{J,\Iw}$, $\Ub_{J,w_J} = \Ub''_{J,w_J} \ltimes \Ub'_{J,w_J}$ and $\Bb_{J,w_J} = \Bb''_{J,w_J} \ltimes \Ub'_{J,w_J}$.
Note that $\Ub'_{J,w_J}$ and $\Ub''_{J,w_J}$ actually depend on $\Iw$ (not only on $w_J$).

Likewise, for any $L_{J \cap \Iw^J{}^{-1}(I)}$-representation $\sigmat$ and using Lemma \ref{lemm:IwJ} with $I$ and $J$ swapped and $\Iw^J$ inverted, the product induces an isomorphism of $F$-varieties
\begin{equation*}
L_J \cap P_{J \cap \Iw^J{}^{-1}(I)}^- \times \left\{ w_J \right\} \times U'_{J,w_J} \iso L_J \cap P_{J \cap \Iw^J{}^{-1}(I)}^- w_J B_J,
\end{equation*}
hence an $A$-linear isomorphism
\begin{equation} \label{isocindJ}
\cind_{L_J \cap P_{J \cap \Iw^J{}^{-1}(I)}^-}^{L_J \cap P_{J \cap \Iw^J{}^{-1}(I)}^- w_J B_J} \sigmat \cong \Clisc\! \left( U'_{J,w_J}, \sigmat^{w_J} \right)
\end{equation}
via which $U'_{J,w_J}$ acts on $\Clisc(U'_{J,w_J},\sigmat^{w_J})$ by right translation and the action of $b'' \in B''_{J,w_J}$ on $f \in \Clisc(U'_{J,w_J},\sigmat^{w_J})$ is given by
\begin{equation*}
\left( b'' \cdot f \right) \left( u' \right) = b'' \cdot f \left( b''^{-1} u' b'' \right)
\end{equation*}
for all $u' \in U'_{J,w_J}$.
In particular with $\sigmat = \sigma^{\Iw^J}$, we have defined a natural smooth $A$-linear action of $B_{J,w_J}$ on $\Clisc(U'_{J,w_J},\sigma^{\Iw})$.

We have a semidirect product $\Ub'_{\Iw} = \Ub'_{J,w_J} \ltimes \Nb'_{J,\Iw}$, so that \eqref{isocind} composed with the $A$-linear morphism defined by $f \mapsto (n' \mapsto (u' \mapsto f(u'n')))$ is an $A$-linear isomorphism
\begin{equation} \label{isocind2}
\cind_{P_I^-}^{P_I^- \Iw B} \sigma \cong \Clisc\! \left( N'_{J,\Iw}, \Clisc\! \left( U'_{J,w_J}, \sigma^{\Iw} \right) \right)
\end{equation}
via which $N'_{J,\Iw}$ acts on $\Clisc(N'_{J,\Iw},\Clisc(U'_{J,w_J},\sigma^{\Iw}))$ by right translation, the action of $b \in B_{J,w_J}$ on $f \in \Clisc(N'_{J,\Iw},\Clisc(U'_{J,w_J},\sigma^{\Iw}))$ is given by
\begin{equation*}
\left( b \cdot f \right) \left( n' \right) = b \cdot f \left( b^{-1} n' b \right)
\end{equation*}
for all $n' \in N'_{J,\Iw}$ and the action of $N''_{J,\Iw}$ on $\Clisc(N'_{J,\Iw},\Clisc(U'_{J,w_J},\sigma^{\Iw}))$ is given by the following result.

\begin{lemm} \label{lemm:action}
Let $f \in \Clisc(N'_{J,\Iw},\Clisc(U'_{J,w_J},\sigma^{\Iw}))$ and $n'' \in N''_{J,\Iw}$.
Via \eqref{isocind2}, the action of $n''$ on $f$ is given by
\begin{equation*}
\left( n'' \cdot f \right) \left( n' \right) \left( u' \right) = n'' \cdot f \left( u'^{-1} n''^{-1} u' n' n'' \right) \left( u' \right)
\end{equation*}
for all $n' \in N'_{J,\Iw}$ and $u' \in U'_{J,w_J}$.
\end{lemm}

\begin{proof}
Let $n' \in N'_{J,\Iw}$ and $u' \in U'_{J,w_J}$.
We have
\begin{equation*}
\Iw u' n' n'' = (\Iw n'' \Iw^{-1}) \Iw u' (u'^{-1} n''^{-1} u' n' n'').
\end{equation*}
Thus, it is enough to check that $u'^{-1} n''^{-1} u' n' n'' \in N'_{J,\Iw}$.
Since $u' \in U_J$ and $n',n'' \in N_J$, we have $(u'^{-1} n''^{-1} u') n' n'' \in N_J$.
Since $n'' \in \Iw^{-1} U_I \Iw$ and $n',u' \in \Iw^{-1} N_I \Iw$, we have $u'^{-1} (n''^{-1} (u' n') n'') \in \Iw^{-1} N_I \Iw$.
Hence the result.
\end{proof}

\begin{rema}[Case $w_J=1$] \label{rema:isocind}
We can also give the action of $L_{J \cap \Iw^J{}^{-1}(I)}$ (which normalises $U'_{J,1}$, $N'_{J,\Iw^J}$, $N''_{J,\Iw^J}$, and thus $U'_{\Iw^J}$, $N_{J,\Iw^J}$) on $\cind_{P_I^-}^{P_I^- \Iw^J B} \sigma$ and $\cind_{L_J \cap P_{J \cap \Iw^J{}^{-1}(I)}^-}^{L_J \cap P_{J \cap \Iw^J{}^{-1}(I)}^- B_J} \sigmat$ (see Remark \ref{rema:Fil1}) via \eqref{isocind2} and \eqref{isocindJ} respectively, by replacing $B_{J,w_J} = B''_{J,w_J} \ltimes U'_{J,w_J}$ by $L_J \cap P_{J \cap \Iw^J{}^{-1}(I)} = L_{J \cap \Iw^J{}^{-1}(I)} \ltimes U'_{J,1}$.
\end{rema}

We end this subsection with some more notation.

\begin{nota} \label{nota:dw}
For each $w \in W$, we let $d_w$ be the integer $\dim_F (\Ub/\Ub_w)$ and $\delta_w \in \X^*(\Sb)$ be the algebraic character of the adjoint representation of $\Sb$ on $\det_F ((\Lie \Ub)/(\Lie \Ub_w))$.
Note that $d_w \geq \ell(w)$ and $\delta_w$ extends to an algebraic character of $\Zcb$.
For $\alpha \in \Delta$, we have $d_{s_\alpha} = d_\alpha \coloneqq \dim_F \Ub_\alpha$ and $\delta_{s_\alpha} = \alpha^{d_\alpha}$.
Note that $d_\alpha=1$ if and only if $\alpha$ extends to an algebraic character of $\Zcb$.
We define a subset of $\Delta$ by setting
\begin{equation*}
\Delta^1 \coloneqq \left\{ \alpha \in \Delta \middlevert d_\alpha = 1 \right\}.
\end{equation*}
For $I \subseteq \Delta$, we put $I^1 \coloneqq I \cap \Delta^1$.
\end{nota}

\begin{rema} \label{rema:dIwJ}
For $\Iw^J \in \IW^J$, we have $\Ub_J \subseteq \Ub_{\Iw^J}$ and $\Lb_{J \cap \Iw^J{}^{-1}(I)}$ normalises $\Nb_{J,\Iw^J}$. Thus, the inclusion $\Nb_J \hookrightarrow \Ub$ induces an isomorphism of $F$-varieties
\begin{equation*}
\Nb_J / \Nb_{J,\Iw^J} \iso \Ub / \Ub_{\Iw^J}
\end{equation*}
and there is an adjoint action of $\Lb_{J \cap \Iw^J{}^{-1}(I)}$ on $(\Lie \Nb_J)/(\Lie \Nb_{J,\Iw^J})$.
Therefore, we have $d_{\Iw^J}=\dim_F (\Nb_J / \Nb_{J,\Iw^J})$ and $\delta_{\Iw^J}$ extends to an algebraic character of $\Lb_{J \cap \Iw^J{}^{-1}(I)}$.
\end{rema}

\section{Derived ordinary parts}

The aim of this section is to compute the derived ordinary parts of a parabolically induced representation.
In §~\ref{ssec:dev}, we show how to compute the cohomology of certain groups with a Hecke action from the cohomology of certain subgroups with the induced Hecke action, provided the latter satisfy some finiteness condition.
In §~\ref{ssec:calc}, we make a computation of cohomology and Hecke action on a compactly induced representation.
In §~\ref{ssec:HOrd}, we use the previous results to partially compute the derived ordinary parts of the graded representations associated to the Bruhat filtrations, we formulate a conjecture on the complete result and we prove it in many cases in low degree.

\subsection{Cohomology, Hecke action and dévissage} \label{ssec:dev}

Let $\Lbt/F$ be an algebraic group and $\Nbt/F$ be a unipotent algebraic group endowed with an action of $\Lbt$ that we identify with the conjugation in $\Lbt \ltimes \Nbt$.
We let $\dt$ denote the integer $\dim_F \Nbt$ and $\deltat \in \X^*(\Lbt)$ denote the algebraic character of the adjoint representation of $\Lbt$ on $\det_F (\Lie \Nbt)$.

Let $\Lt^+ \subseteq \Lt$ be an open submonoid and $\Nt_0 \subseteq \Nt$ be a standard\footnote{The exponential map $\exp : \Lie \Nbt \to \Nbt$ is an isomorphism of $F$-varieties (cf. \cite[Chapitre IV, §~2, Proposition 4.1]{DG}) and we say that $\Nt_0$ is \emph{standard} if $\Lie \Nt_0 \coloneqq \exp^{-1}(\Nt_0) \subseteq \Lie \Nt$ is a $\Zp$-Lie subalgebra. The identity of $\Nt$ admits a basis of neighbourhoods consisting of standard compact open subgroups (cf. \cite[Lemma 3.5.2]{Em2}).} compact open subgroup stable under conjugation by $\Lt^+$.
If $\pi$ is an $\Lt^+ \ltimes \Nt_0$-representation\footnote{Given a $p$-adic Lie group $H$ and an open submonoid $H^+ \subseteq H$, a representation of $H^+$ over $A$ is \emph{smooth} if its restriction to an open subgroup of $H$ contained in $H^+$ is smooth.}, then the $A$-modules of $\Nt_0$-cohomology $\Hc(\Nt_0,\pi)$ computed using locally constant cochains (or equivalently an $N_0$-injective resolution of $\pi$, cf. \cite[Proposition 2.2.6]{Em2}) are naturally endowed with the Hecke action of $\Lt^+$ (denoted $\h$), defined for every $\lt \in \Lt^+$ as the composite
\begin{equation*}
\Hc\! \left( \Nt_0, \pi \right) \to \Hc\! \left( \lt\Nt_0\lt^{-1}, \pi \right) \to \Hc\! \left( \Nt_0, \pi \right)
\end{equation*}
where the first morphism is induced by the action of $\lt$ on $\pi$ and the second morphism is the corestriction from $\lt\Nt_0\lt^{-1}$ to $\Nt_0$ (this defines a natural smooth $A$-linear action of $\Lt^+$ in degree $0$ by \cite[Lemma 3.1.4]{Em1}, that extends in higher degrees by universality of $\Hc(\Nt_0,-)$).
We obtain a universal $\delta$-functor $\Hc(\Nt_0,-) : \Mod_{\Lt^+ \ltimes \Nt_0}^\sm(A) \to \Mod_{\Lt^+}^\sm(A)$ (since an injective $\Lt^+ \ltimes \Nt_0$-representation is $\Nt_0$-acyclic, cf. \cite[Proposition 2.1.11]{Em2} and \cite[Lemme 3.1.1]{JH}).

Let $\Zbt \subseteq \Lbt$ be a central split torus and $\Zt^+ \subseteq \Zt$ be the open submonoid $\Zt \cap \Lt^+$.
Since $\Zbt$ is split, its adjoint representation on $\Lie \Nbt$ is a direct sum of weights.
We assume that there exists $\lambdat \in \X_*(\Zbt)$ such that $\langle \mut,\lambdat \rangle >0$ for any weight $\mut$ of $\Zbt$ in $\Lie \Nbt$.
We fix an element $\zt \coloneqq \lambdat(p^j) \in \Zt$ with $j \in \N$ large enough so that $\zt$ is \emph{strictly contracting} $\Nt_0$, i.e. $(\zt^i \Nt_0 \zt^{-i})_{i \in \N}$ is a basis of neighbourhoods of the identity in $\Nt_0$ (cf. \cite[Lemma 3.1.3]{Em2} using the fact that $\ord_p(\mut(\zt)) = \langle \mut,\lambdat \rangle j$ for any weight $\mut$ of $\Zbt$ in $\Lie \Nbt$).
In particular $\zt \in \Zt^+$.

If $\pi$ is a $\Zt^+$-representation, we say that $\pi$ is \emph{locally $\zt$-finite} if for every $v \in \pi$, the $A$-submodule $A[\zt] \cdot v$ is of finite type, and we say that the action of $\zt$ on $\pi$ is \emph{locally nilpotent} if for every $v \in \pi$, there exists $i \in \N$ such that $\zt^i \cdot v = 0$.

\begin{lemm} \label{lemm:Zfin}
Let $\pi$ be a locally $\zt$-finite $\Lt^+ \ltimes \Nt_0$-representation and $n \in \N$.
\begin{enumerate}
\item If $n = [F:\Qp] \dt$, then the action of $\zt$ on the kernel of the natural $\Lt^+$-equivariant surjection $\pi \otimes (\omega^{-1} \circ \deltat) \twoheadrightarrow \Hc[n](\Nt_0,\pi)$ is locally nilpotent.
\item If $n < [F:\Qp] \dt$, then the Hecke action of $\zt$ on $\Hc[n](\Nt_0,\pi)$ is locally nilpotent.
\end{enumerate}
\end{lemm}

\begin{proof}
We prove (i).
The natural $\Lt^+$-equivariant surjection in the statement is the composite
\begin{equation*}
\pi \otimes \left( \omega^{-1} \circ \deltat \right) \twoheadrightarrow \pi_{\Nt_0} \otimes \left( \omega^{-1} \circ \deltat \right) \cong \Hc[{[F:\Qp]} \dt]\! \left( \Nt_0, \pi \right)
\end{equation*}
where the first morphism is the natural projection onto the $\Nt_0$-coinvariants of $\pi$ and the second morphism is the natural isomorphism \cite[(2.2)]{JHB} which is due to Emerton (in loc. cit. $\alphat \in \X^*(\Res_{F/\Qp} \Lbt)$ is the algebraic character of the adjoint representation of $\Res_{F/\Qp} \Lbt$ on $\det_{\Qp} (\Lie(\Res_{F/\Qp} \Nbt))$ so that $\alphat = \Nrm_{F/\Qp} \circ \deltat$ as $\Qp^\times$-valued characters of $\Lt$, hence $\alphat^{-1} \lvert\alphat\rvert_p^{-1} = \omega^{-1} \circ \deltat$ as $\Qp^\times$-valued characters of $\Lt$).
For every $v \in \pi$, there exists $i \in \N$ such that $\zt^i \Nt_0 \zt^{-i}$ fixes $A[\zt] \cdot v$ (since $\pi$ is locally $\zt$-finite and $\zt$ is strictly contracting $\Nt_0$), so that for all $\nt \in \Nt_0$ we have
\begin{equation*}
\zt^i \cdot \left( \nt \cdot v - v \right) = \left( \zt^i \nt \zt^{-i} \right) \cdot \left( \zt^i \cdot v \right) - \left( \zt^i \cdot v \right) = 0.
\end{equation*}
Thus the action of $\zt$ on the kernel of the above surjection is locally nilpotent.

We prove (ii).
Let $(\mut_r)_{r \in \llbrack 0,m-1 \rrbrack}$ be an enumeration of the weights of $\Zbt$ in $\Lie \Nbt$ such that the sequence $(\langle \mut_r,\lambdat \rangle)_{r \in \llbrack 0,m-1 \rrbrack}$ is increasing.
If $\mut_i+\mut_j=\mut_r$ with $i,j,r \in \llbrack 0,m-1 \rrbrack$, then $r > \max \{i,j\}$ (since $\langle \mut_r,\lambdat \rangle > \max \{\langle \mut_i,\lambdat \rangle,\langle \mut_j,\lambdat \rangle\}$).
Thus for all $r \in \llbrack 0,m \rrbrack$, the direct sum of the weight spaces corresponding to $\mut_r,\dots,\mut_{m-1}$ is an ideal of $\Lie \Nbt$ stable under the adjoint action of $\Zbt$ and we let $\Nbt^{(r)} \subseteq \Nbt$ be the corresponding closed normal subgroup stable under conjugation by $\Zbt$, $\dt_r$ denote the integer $\dim_F \Nbt^{(r)}$, $\deltat_r \in \X^*(\Lbt)$ denote the algebraic character of the adjoint representation of $\Lbt$ on $\det_F(\Lie \Nbt^{(r)})$ and $\Nt_0^{(r)} \subseteq \Nt^{(r)}$ be the standard compact open subgroup $\Nt^{(r)} \cap \Nt_0$ stable under conjugation by $\Zt^+$.

Let $r \in \llbrack 0,m \rrbrack$.
We assume that $n < [F:\Qp] \dt_r$ and we prove that the Hecke action of $\zt$ on $\Hc[n](\Nt^{(r)}_0,\pi)$ is locally nilpotent by induction on $r$.
The result is trivial for $r=m$.
We assume $r<m$ and the result true for $r+1$.
We have a short exact sequence of topological groups
\begin{equation*}
1 \to \Nt_0^{(r+1)} \to \Nt_0^{(r)} \to \Nt_0^{(r)}/\Nt_0^{(r+1)} \to 1.
\end{equation*}
The Lyndon--Hochschild--Serre spectral sequence associated to this dévissage is naturally a spectral sequence of $\Lt^+$-representations (cf. \cite[(2.3)]{JHB})
\begin{equation} \label{HSSS}
\Hc[i]\! \left( \Nt_0^{(r)} / \Nt_0^{(r+1)}, \Hc[j]\! \left( \Nt_0^{(r+1)}, \pi \right) \right) \Rightarrow \Hc[i+j]\! \left( \Nt_0^{(r)}, \pi \right).
\end{equation}
Let $i,j \in \N$ such that $i+j=n$.
If $j < [F:\Qp] \dt_{r+1}$, then the Hecke action of $\zt$ on $\Hc[j](\Nt_0^{(r+1)},\pi)$ is locally nilpotent by the induction hypothesis, thus the Hecke action of $\zt$ on $\Hc[i](\Nt_0^{(r)}/\Nt_0^{(r+1)},\Hc[j](\Nt_0^{(r+1)},\pi))$ is also locally nilpotent (since the image of a locally constant cochain is finite by compactness).
If $j = [F:\Qp] \dt_{r+1}$, then $i < [F:\Qp] (\dt_r-\dt_{r+1})$ and we deduce from (i) with $\Nbt^{(r+1)}$ and $j$ instead of $\Nbt$ and $n$ respectively that $\Hc[j](\Nt_0^{(r+1)},\pi)$ is locally $\zt$-finite, thus the Hecke action of $\zt$ on $\Hc[i](\Nt_0^{(r)}/\Nt_0^{(r+1)},\Hc[j](\Nt_0^{(r+1)},\pi))$ is locally nilpotent by the sublemma below with $\mut=\mut_r$ and $\Nbt^{(r)}/\Nbt^{(r+1)}$, $\Hc[j](\Nt_0^{(r+1)},\pi)$, $i$ instead of $\Nbt$, $\pi$, $n$ respectively.
If $j > [F:\Qp] \dt_{r+1}$, then $\Hc[j](\Nt_0^{(r+1)},\pi) = 0$ by \cite[Lemma 3.5.4]{Em2}, thus $\Hc[i](\Nt_0^{(r)}/\Nt_0^{(r+1)},\Hc[j](\Nt_0^{(r+1)},\pi)) = 0$.
Using \eqref{HSSS}, we conclude that the action of $\zt$ on $\Hc[n](\Nt^{(r)}_0,\pi)$ is locally nilpotent.
\end{proof}

\begin{slem}
Let $\pi$ be a locally $\zt$-finite $\Zt^+ \ltimes \Nt_0$-representation, $\mut \in \X^*(\Zbt)$ and $n \in \N$.
Assume that the adjoint action of $\Zbt$ on $\Lie \Nbt$ factors through $\mut$.
If $n < [F:\Qp] \dt$, then the Hecke action of $\zt$ on $\Hc[n](\Nt_0,\pi)$ is locally nilpotent.
\end{slem}

\begin{proof}
Let $\Sbt \subseteq \Res_{F/\Qp} \Zbt$ be the maximal split subtorus, $\St \subseteq \Zt$ be the closed subgroup $\Sbt(\Qp)$ and $\St^+ \subseteq \St$ be the open submonoid $\St \cap \Zt^+$.
Every algebraic (co)character of $\Zbt$ induces by restriction of scalars a (co)character of $\Sbt$ (since the image of a split torus by a morphism of algebraic groups is a split torus, cf. \cite[§~1.4]{BT}).
In particular, the restriction of $\lambdat : F^\times \to \Zt$ to $\Qp^\times$ takes values in $\St$ and the restriction of $\mut : \Zt \to F^\times$ to $\St$ takes values in $\Qp^\times$.

We deduce on the one hand that $\zt \in \St^+$, and on the other hand that the adjoint action of $\Sbt$ on $\Lie(\Res_{F/\Qp} \Nbt)$ factors through an algebraic character so that any closed subgroup of $\Res_{F/\Qp} \Nbt$ is stable under conjugation by $\Sbt$.
Since $\Res_{F/\Qp} \Nbt$ is unipotent, there exists a composition series
\begin{equation*}
\Res_{F/\Qp} \Nbt = \Nbt^{(0)} \vartriangleright \Nbt^{(1)} \vartriangleright \dots \vartriangleright \Nbt^{([F:\Qp] \dt)} = 1
\end{equation*}
whose successive quotients are isomorphic to the additive group over $\Qp$ and for all $r \in \llbrack 0,[F:\Qp] \dt \rrbrack$, we let $\Nt^{(r)} \subseteq \Nt$ be the closed subgroup $\Nbt^{(r)}(\Qp)$ and $\Nt_0^{(r)} \subseteq \Nt^{(r)}$ be the standard compact open subgroup $\Nt^{(r)} \cap \Nt_0$ stable under conjugation by $\St^+$.

Let $r \in \llbrack 0,[F:\Qp] \dt \rrbrack$.
We assume that $n<[F:\Qp] \dt-r$ and we prove by induction on $r$ that the Hecke action of $\zt$ on $\Hc[n](\Nt_0^{(r)},\pi)$ is locally nilpotent.
The result is trivial for $r=[F:\Qp] \dt$.
We assume $r<[F:\Qp] \dt$ and the result true for $r+1$.
Since $\dim_{\Qp} (\Nbt^{(r)}/\Nbt^{(r+1)}) = 1$, we have a short exact sequence of $\St^+$-representations (cf. \cite[(2.4)]{JHB})
\begin{multline} \label{HSSE}
0 \to \Hc[1]\! \left( \Nt_0^{(r)} / \Nt_0^{(r+1)}, \Hc[n-1]\! \left( \Nt_0^{(r+1)}, \pi \right) \right) \to \Hc[n]\! \left( \Nt_0^{(r)}, \pi \right) \\
\to \Hc[n]\! \left( \Nt_0^{(r+1)}, \pi \right)^{\Nt_0^{(r)} / \Nt_0^{(r+1)}} \to 0.
\end{multline}
The Hecke action of $\zt$ on $\Hc[n-1](\Nt_0^{(r+1)},\pi)$ is locally nilpotent by the induction hypothesis, thus the Hecke action of $\zt$ on $\Hc[1](\Nt_0^{(r)}/\Nt_0^{(r+1)},\Hc[n-1](\Nt_0^{(r+1)},\pi))$ is also locally nilpotent.
If $n<[F:\Qp] \dt-(r+1)$, then the Hecke action of $\zt$ on $\Hc[n](\Nt_0^{(r+1)},\pi)$ is locally nilpotent by induction, thus the Hecke action of $\zt$ on $\Hc[n](\Nt_0^{(r+1)},\pi)^{\Nt_0^{(r)}/\Nt_0^{(r+1)}}$ is also locally nilpotent.
If $n=[F:\Qp] \dt-(r+1)$, then we have a natural $\St^+$-equivariant surjection $\pi \otimes \mut^{-n} \lvert\mut\rvert_p^{-n} \twoheadrightarrow \Hc[n](\Nt_0^{(r+1)},\pi)$ (cf. \cite[(2.2)]{JHB}) and we deduce that $\Hc[n](\Nt_0^{(r+1)},\pi)$ is locally $\zt$-finite.
In this case, we put $\Nt''_0 \coloneqq \Nt_0^{(r)}/\Nt_0^{(r+1)}$.
For every $v \in \Hc[n](\Nt_0^{(r+1)},\pi)$, there exists $i \in \N$ such that $\zt^i \Nt''_0 \zt^{-i}$ fixes $A[\zt] \cdot v$, so that for all $j \in \N$ we have
\begin{align*}
\zt^{i+j} \h v &= \sum_{\nt'' \in \Nt''_0 / \zt^{i+j} \Nt''_0 \zt^{-(i+j)}} \nt'' \cdot \left( \zt^{i+j} \cdot v \right) \\
&= \left( \zt^i \Nt''_0 \zt^{-i} : \zt^{i+j} \Nt''_0 \zt^{-(i+j)} \right) \sum_{\nt'' \in \Nt''_0 / \zt^i \Nt''_0 \zt^{-i}} \nt'' \cdot \left( \zt^{i+j} \cdot v \right) \\
&= \left( \Nt''_0 : \zt^j \Nt''_0 \zt^{-j} \right) \sum_{\nt'' \in \Nt''_0 / \zt^i \Nt''_0 \zt^{-i}} \nt'' \cdot \left( \zt^{i+j} \cdot v \right).
\end{align*}
Now $\Nt''_0$ is an infinite pro-$p$ group, $\zt$ is strictly contracting $\Nt''_0$ and $A$ is Artinian.
Thus $(\Nt''_0 : \zt^j \Nt''_0 \zt^{-j})$ is zero in $A$ for $j \in \N$ large enough.
Therefore, the Hecke action of $\zt$ on $\Hc[n](\Nt_0^{(r+1)},\pi)^{\Nt_0^{(r)}/\Nt_0^{(r+1)}}$ is locally nilpotent.
Using \eqref{HSSE}, we conclude that the Hecke action of $\zt$ on $\Hc[n](\Nt_0^{(r)},\pi)$ is locally nilpotent.
\end{proof}

Let $\Nbt' \subseteq \Nbt$ be a closed subgroup such that $\Lie \Nbt' \subseteq \Lie \Nbt$ is a direct sum of weight spaces of $\Zbt$.
We stress that $\Nbt'$ need not be normal.
Since $\Zbt$ is central in $\Lbt$, $\Lie \Nbt'$ is stable under the adjoint action of $\Lbt$, thus $\Nbt'$ is stable under conjugation by $\Lbt$.
We let $\dt'$ denote the integer $\dim_F \Nbt'$ and $\deltat' \in \X^*(\Lbt)$ denote the algebraic character of the adjoint representation of $\Lbt$ on $\det_F (\Lie \Nbt')$.
We let $\Nt'_0 \subseteq \Nt'$ be the standard compact open subgroup $\Nt' \cap \Nt_0$ stable under conjugation by $\Lt^+$.

\begin{prop} \label{prop:dev}
Let $\pi$ be an $\Lt^+ \ltimes \Nt_0$-representation.
For all $n \in \N$, there is a natural $\Lt^+$-equivariant morphism
\begin{equation*}
\Hc[n - {[F:\Qp]} \left( \dt - \dt' \right)]\! \left( \Nt'_0, \pi \right) \otimes \left( \omega^{-1} \circ \left( \deltat - \deltat' \right) \right) \to \Hc[n]\! \left( \Nt_0, \pi \right).
\end{equation*}
Furthermore, the Hecke action of $\zt$ on its kernel and cokernel is locally nilpotent if the $\Lt^+$-representations $\Hc(\Nt'_0,\pi)$ are locally $\zt$-finite.
\end{prop}

\begin{proof}
Let $(\mut_r)_{r \in \llbrack 0,m-m'-1 \rrbrack}$ be an enumeration of the weights of $\Zbt$ in $(\Lie \Nbt)/(\Lie \Nbt')$ such that the sequence $(\langle \mut_r,\lambdat \rangle)_{r \in \llbrack 0,m-m'-1 \rrbrack}$ is increasing and $(\mut_r)_{r \in \llbrack m-m',m-1 \rrbrack}$ be an enumeration of the weights of $\Zbt$ in $\Lie \Nbt'$ such that the sequence $(\langle \mut_r,\lambdat \rangle)_{r \in \llbrack m-m',m-1 \rrbrack}$ is increasing.
If $\mut_i+\mut_j=\mut_r$ with $i,j,r \in \llbrack 0,m-1 \rrbrack$, then $r > \min \{i,j\}$ (since $\langle \mut_r,\lambdat \rangle > \max \{\langle \mut_i,\lambdat \rangle,\langle \mut_j,\lambdat \rangle\}$).
Thus for all $r \in \llbrack 0,m-m' \rrbrack$, the direct sum of the weight spaces corresponding to $\mut_r,\dots,\mut_{m-1}$ is a Lie subalgebra of $\Lie \Nbt$ stable under the adjoint action of $\Zbt$ and we use the notations $\Nbt^{(r)}$, $\dt_r$, $\deltat_r$ and $\Nt_0^{(r)}$ as in the proof of Lemma \ref{lemm:Zfin} (ii).
Moreover for all $r \in \llbrack 0,m-m'-1 \rrbrack$, $\Lie \Nbt^{(r+1)}$ is an ideal of $\Lie \Nbt^{(r)}$ so that $\Nbt^{(r+1)}$ is a normal subgroup of $\Nbt^{(r)}$.

Let $r \in \llbrack 0,m-m' \rrbrack$.
We prove by induction on $r$ that for all $n \in \N$, there is a natural $\Lt^+$-equivariant morphism
\begin{equation} \label{piso}
\Hc[n - {[F:\Qp]} \left( \dt_r - \dt' \right)]\! \left( \Nt'_0, \pi \right) \otimes \left( \omega^{-1} \circ \left( \deltat_r - \deltat' \right) \right) \to \Hc[n]\! \left( \Nt_0^{(r)}, \pi \right).
\end{equation}
The result is trivial for $r=m-m'$.
We assume $r<m-m'$ and the result true for $r+1$.
Let $n \in \N$.
Since $\dim_F (\Nt^{(r)}/\Nt^{(r+1)}) = \dt_r - \dt_{r+1}$, we deduce from \cite[Lemma 3.5.4]{Em2} that \eqref{HSSS} yields a natural $\Lt^+$-equivariant morphism
\begin{equation} \label{piso1}
\Hc[{[F:\Qp]} \left( \dt_r - \dt_{r+1} \right)]\! \left( \Nt_0^{(r)} / \Nt_0^{(r+1)}, \Hc[n - {[F:\Qp]} \left( \dt_r - \dt_{r+1} \right)]\! \left( \Nt_0^{(r+1)}, \pi \right) \right) \\
\to \Hc[n]\! \left( \Nt_0^{(r)}, \pi \right)
\end{equation}
whose kernel and cokernel are built out of subquotients of $\Hc[i](\Nt_0^{(r)}/\Nt_0^{(r+1)},\Hc[j](\Nt_0^{(r+1)},\pi))$ with $i,j \in \N$ such that $i<[F:\Qp] (\dt_r-\dt_{r+1})$.
Furthermore, Lemma \ref{lemm:Zfin} (i) with $\Nbt^{(r)}/\Nbt^{(r+1)}$, $\Hc[n-{[F:\Qp]} (\dt_r-\dt_{r+1})](\Nt_0^{(r+1)},\pi)$, $[F:\Qp] (\dt_r-\dt_{r+1})$ and $\deltat_r-\deltat_{r+1}$ instead of $\Nbt$, $\pi$, $n$ and $\deltat$ respectively yields a natural $\Lt^+$-equivariant surjection
\begin{multline} \label{piso2}
\Hc[n - {[F:\Qp]} \left( \dt_r - \dt_{r+1} \right)]\! \left( \Nt_0^{(r+1)}, \pi \right) \otimes \left( \omega^{-1} \circ \left( \deltat_r - \deltat_{r+1} \right) \right) \\
\twoheadrightarrow \Hc[{[F:\Qp]} \left( \dt_r - \dt_{r+1} \right)]\! \left( \Nt_0^{(r)} / \Nt_0^{(r+1)}, \Hc[n - {[F:\Qp]} \left( \dt_r - \dt_{r+1} \right)]\! \left( \Nt_0^{(r+1)}, \pi \right) \right).
\end{multline}
Finally by the induction hypothesis with $n-[F:\Qp] (\dt_r-\dt_{r+1})$ instead of $n$, there is a natural $\Lt^+$-equivariant morphism
\begin{multline} \label{piso3}
\Hc[n - {[F:\Qp]} \left( \dt_r - \dt' \right)]\! \left( \Nt'_0, \pi \right) \otimes \left( \omega^{-1} \circ \left( \deltat_r - \deltat' \right) \right) \\
\to \Hc[n - {[F:\Qp]} \left( \dt_r - \dt_{r+1} \right)]\! \left( \Nt_0^{(r+1)}, \pi \right) \otimes \left( \omega^{-1} \circ \left( \deltat_r - \deltat_{r+1} \right) \right).
\end{multline}
The composition of \eqref{piso1}, \eqref{piso2} and \eqref{piso3} yields the natural $\Lt^+$-equivariant morphism \eqref{piso}.

Now, we assume that the $\Lt^+$-representations $\Hc(\Nt'_0,\pi)$ are locally $\zt$-finite and we prove by induction on $r$ that for all $n \in \N$, the Hecke action of $\zt$ on the kernel and cokernel of \eqref{piso} is locally nilpotent, or equivalently that the localisation of \eqref{piso} with respect to $\zt^\N$ is an isomorphism.
The result is trivial for $r=m-m'$.
We assume $r<m-m'$ and the result true for $r+1$.
Let $n \in \N$.
By composition, it is enough to prove that the Hecke action of $\zt$ on the kernels and cokernels of \eqref{piso1}, \eqref{piso2} and \eqref{piso3} is locally nilpotent.
By the induction hypothesis with $j$ instead of $n$, the Hecke action of $\zt$ on the kernel and cokernel of the natural $\Lt^+$-equivariant morphism
\begin{equation*}
\Hc[j - {[F:\Qp]} \left( \dt_{r+1} - \dt' \right)]\! \left( \Nt'_0, \pi \right) \otimes \left( \omega^{-1} \circ \left( \deltat_{r+1} - \deltat' \right) \right) \to \Hc[j]\! \left( \Nt_0^{(r+1)}, \pi \right)
\end{equation*}
is locally nilpotent for all $j \in \N$.
With $j=n-[F:\Qp] (\dt_r-\dt_{r+1})$, we deduce that the Hecke action of $\zt$ on the kernel and cokernel of \eqref{piso3} is locally nilpotent.
Furthermore, we deduce that $\Hc[j] (\Nt_0^{(r+1)},\pi)$ is locally $\zt$-finite for all $j \in \N$ and we use Lemma \ref{lemm:Zfin} with $\Nbt^{(r)}/\Nbt^{(r+1)}$, $\Hc[j] (\Nt_0^{(r+1)},\pi)$ and $i$ instead of $\Nbt$, $\pi$ and $n$ respectively: we deduce from (i) with $i=[F:\Qp] (\dt_r-\dt_{r+1})$ that the Hecke action of $\zt$ on the kernel of \eqref{piso2} is locally nilpotent, and we deduce from (ii) that the Hecke action of $\zt$ on the kernel and cokernel of \eqref{piso1} is locally nilpotent (since the Hecke action of $\zt$ on $\Hc[i](\Nt_0^{(r)}/\Nt_0^{(r+1)},\Hc[j](\Nt_0^{(r+1)},\pi))$ is locally nilpotent for all $i,j \in \N$ such that $i<[F:\Qp] (\dt_r-\dt_{r+1})$).
\end{proof}

We end this subsection by reviewing and generalising the construction of Emerton's $\delta$-functor of derived ordinary parts (cf. \cite[§~3.3]{Em2}).
Let $\Zbt_{\Lbt}$ denote the centre of $\Lbt$.
Assume that $\Zt_{\Lt}$ is generated by $\Zt_{\Lt}^+ \coloneqq \Zt_{\Lt} \cap \Lt^+$ as a group, and that $\Lt$ is generated by $\Lt^+$ and $\Zt_{\Lt}$ as a monoid.
Then, the product induces a group isomorphism $\Lt^+ \times_{\Zt_{\Lt}^+} \Zt_{\Lt} \iso \Lt$ (cf. \cite[Proposition 3.3.6]{Eme06}).
Thus, for any $\Lt^+$-representation $\pi$, the $A$-module $\Hom_{A[\Zt_{\Lt}^+]}(A[\Zt_{\Lt}],\pi)^{\Zt_{\Lt}-\lfin}$ is naturally an $\Lt$-representation (cf. \cite[Lemma 3.1.7]{Em1}).
Therefore, we obtain an $A$-linear left-exact functor $\Mod_{\Lt^+}^\sm(A) \to \Mod_{\Lt}^\sm(A)^{\Zt_{\Lt}-\lfin}$ which commutes with inductive limits (cf. \cite[Lemma 3.2.2]{Em1}).

\begin{rema} \label{rema:Loczt}
Let $\zt \in \Zt_{\Lt}^+$.
Assume that $\Zt_{\Lt}$ is generated by $\Zt_{\Lt}^+$ and $\zt^{-1}$ as a monoid.
Then, for any locally finite $\Zt_{\Lt}^+$-representation $\pi$, there is a natural $\Zt_{\Lt}$-equivariant isomorphism $\Hom_{A[\Zt_{\Lt}^+]}(A[\Zt_{\Lt}],\pi)^{\Zt_{\Lt}-\lfin} \iso A[\zt^{\pm1}] \otimes_{A[\zt]} \pi$ (cf. \cite[Lemma 3.2.1]{Em2}).
Thus, the functor $\Hom_{A[\Zt_{\Lt}^+]}(A[\Zt_{\Lt}],-)^{\Zt_{\Lt}-\lfin}$ restricted to the category $\Mod_{\Lt^+}^\sm(A)^{\Zt_{\Lt}^+-\lfin}$ is isomorphic to the localisation with respect to $\zt^\N$.
In particular, it is exact.
\end{rema}

\begin{defi}
For a connected algebraic group $\Pbt/F$ with unipotent radical $\Nbt$ such that $\Pbt \cong \Lbt \ltimes \Nbt$, we define $A$-linear functors $\Mod_{\Pt}^\sm(A) \to \Mod_{\Lt}^\sm(A)^{\Zt_{\Lt}-\lfin}$ which commute with inductive limits by setting
\begin{equation*}
\HOrd_{\Pt} \coloneqq \Hom_{A[\Zt_{\Lt}^+]} \left( A[\Zt_{\Lt}], \Hc\! \left( \Nt_0, - \right) \right)^{\Zt_{\Lt}-\lfin}.
\end{equation*}
If $\Bbt \subseteq \Pbt$ is a connected closed subgroup containing $\Nbt$ and $\Zbt_{\Lbt}$, then $\Bt_{\Lt} \coloneqq \Bt \cap \Lt$ is generated by $\Bt_{\Lt}^+ \coloneqq \Bt_{\Lt} \cap \Lt^+$ and $\Zt_{\Lt}$ as a monoid, so that $\HOrd_{\Pt}$ naturally extend to $A$-linear functors $\Mod_{\Bt}^\sm(A) \to \Mod_{\Bt_{\Lt}}^\sm(A)^{\Zt_{\Lt}-\lfin}$ which commute with inductive limits.
\end{defi}

\subsection{Computations on the associated graded representations} \label{ssec:calc}

Let $J \subseteq \Delta$.
We fix a totally decomposed\footnote{Given a closed subgroup $\Ubt \subseteq \Ub$ stable under conjugation by $\Sb$, we say that a compact open subgroup $\Ut_0 \subseteq \Ut$ is \emph{totally decomposed} if the product induces a homeomorphism $\prod_{\alpha \in \Phi_0^+} (U_\alpha \cap \Ut_0) \iso \Ut_0$ for any order on $\Phi_0^+$ (e.g. $\Ut_0 = \Ut \cap K$ where $K \subseteq G$ is a maximal compact subgroup which is special with respect to $\Zc$, cf. \cite[§~6.6, Remark 2]{HV}).}
standard compact open subgroup $N_{J,0} \subseteq N_J$ and we define an open submonoid of $L_J$ by setting
\begin{equation*}
L_J^+ \coloneqq \left\{ l \in L_J \middlevert lN_{J,0}l^{-1} \subseteq N_{J,0} \right\}.
\end{equation*}
We let $Z_J^+ \subseteq Z_J$ be the open submonoid $Z_J \cap L_J^+$.
Note that $Z_J$ is generated by $Z_J^+$ as a group and $L_J$ is generated by $L_J^+$ and $Z_J$ as a monoid (cf. \cite[Proposition 3.3.2]{Eme06}).
Moreover, any $\lambda \in \X_*(\Sb)$ associated to $\Pb_J$ has its image contained in the maximal split subtorus $\Sb_J$ of $\Zb_J^\circ$ and satisfies $\langle \alpha, \lambda \rangle > 0$ for all $\alpha \in \Phi^+ \backslash \Phi_J^+$, thus the assumption of §~\ref{ssec:dev} with $\Nbt=\Nb_J$ and $\Zbt=\Sb_J$ is satisfied.
We fix $z \in Z_J^+$ strictly contracting $N_{J,0}$ (equivalently $Z_J$ is generated by $Z_J^+$ and $z^{-1}$ as a monoid).

Let $I \subseteq \Delta$ and $\Iw \in \IW$.
We write $\Iw = \Iw^J w_J$ with $\Iw^J \in \IW^J$ and $w_J \in W_J$.
Let $\sigma$ be an $L_I$-representation.
We set\footnote{The naturality of a morphism involving $\pi_{\Iw}$ will mean its functoriality with respect to $\sigma$.}
\begin{equation*}
\pi_{\Iw} \coloneqq \cind_{P_I^-}^{P_I^- \Iw B} \sigma.
\end{equation*}
We use the notation of §~\ref{ssec:isocind}.
The subgroup $\Nb_{J,\Iw} \subseteq \Nb_J$ is stable under conjugation by $\Bb_{J,w_J}$, and we have a semidirect product $\Nb_{J,\Iw} = \Nb''_{J,\Iw} \ltimes \Nb'_{J,\Iw}$.
The subgroup $\Nb'_{J,\Iw}$ is stable under conjugation by $\Bb_{J,w_J}$, and we endow $\Nb''_{J,\Iw}$ (which may not be stable under conjugation by $\Bb_{J,w_J}$) with the quotient action of $\Bb_{J,w_J}$ via the isomorphism $\Nb''_{J,\Iw} \cong \Nb_{J,\Iw} / \Nb'_{J,\Iw}$.
We let $N_{J,\Iw,0} \subseteq N_{J,\Iw}$ (resp. $N'_{J,\Iw,0} \subseteq N'_{J,\Iw}$, $N''_{J,\Iw,0} \subseteq N''_{J,\Iw}$) be the totally decomposed standard compact open subgroup $N_{J,\Iw} \cap N_{J,0}$ (resp. $N'_{J,\Iw} \cap N_{J,0}$, $N''_{J,\Iw} \cap N_{J,0}$) and $B_{J,w_J}^+ \subseteq B_{J,w_J}$ be the open submonoid $B_{J,w_J} \cap L_J^+$.
Since $N_{J,\Iw,0}$ is totally decomposed, we have a short exact sequence of topological groups
\begin{equation} \label{dev}
1 \to N'_{J,\Iw,0} \to N_{J,\Iw,0} \to N''_{J,\Iw,0} \to 1.
\end{equation}
In particular, $N''_{J,\Iw,0}$ is stable under the quotient action of $B_{J,w_J}^+$ on $N''_{J,\Iw}$.

\begin{lemm} \label{lemm:calc1}
For all $n \in \N$, the inflation map is a natural $B_{J,w_J}^+$-equivariant isomorphism
\begin{equation*}
\Hc[n]\! \left( N''_{J,\Iw,0}, \pi_{\Iw}^{N'_{J,\Iw,0}} \right) \iso \Hc[n]\! \left( N_{J,\Iw,0}, \pi_{\Iw} \right).
\end{equation*}
\end{lemm}

\begin{proof}
The Lyndon--Hochschild--Serre spectral sequence associated to \eqref{dev} is naturally a spectral sequence of $B_{J,w_J}^+$-representations (cf. \cite[(2.3)]{JHB})
\begin{equation} \label{SS}
\Hc[i]\! \left( N''_{J,\Iw,0}, \Hc[j]\! \left( N'_{J,\Iw,0}, \pi_{\Iw} \right) \right) \Rightarrow \Hc[i+j]\! \left( N_{J,\Iw,0}, \pi_{\Iw} \right).
\end{equation}
The inflation maps are the edge maps of \eqref{SS} for $j=0$, thus they are $B_{J,w_J}^+$-equivariant and in order to prove that they are bijective, it is enough to show that \eqref{SS} degenerates, i.e. that $\Hc[j](N'_{J,\Iw,0},\pi_{\Iw}) = 0$ for all integers $j>0$.

Since the left cosets $N'_{J,\Iw}/N'_{J,\Iw,0}$ form an open partition of $N'_{J,\Iw}$, we deduce from \eqref{isocind2} a natural $N'_{J,\Iw,0}$-equivariant isomorphism
\begin{equation*}
\pi_{\Iw} \cong \bigoplus_{n' \in N'_{J,\Iw} / N'_{J,\Iw,0}} \Clis\! \left( n' N'_{J,\Iw,0}, \Clisc\! \left( U'_{J,w_J}, \sigma^{\Iw} \right) \right)
\end{equation*}
where $N'_{J,\Iw,0}$ acts by right translation on the terms of the direct sum.
The latter are $N'_{J,\Iw,0}$-acyclic by Shapiro's lemma (since they are induced discrete $A[N'_{J,\Iw,0}]$-modules) and the $N'_{J,\Iw,0}$-cohomology commutes with direct sums (since the image of a locally constant cochain is finite by compactness), thus $\pi_{\Iw}$ is $N'_{J,\Iw,0}$-acyclic.
\end{proof}

There is a natural smooth $A$-linear action of $B''_{J,w_J} \ltimes (U'_{J,w_J} \times N''_{J,\Iw})$ on $\Clisc(U'_{J,w_J},\sigma^{\Iw})$: we already defined the action of $B_{J,w_J} = B''_{J,w_J} \ltimes U'_{J,w_J}$ in §~\ref{ssec:isocind} and we define the action of $n'' \in N''_{J,\Iw}$ on $f \in \Clisc(U'_{J,w_J},\sigma^{\Iw})$ by setting
\begin{equation*}
\left( n'' \cdot f \right) \left( u' \right) \coloneqq n'' \cdot f \left( u' \right)
\end{equation*}
for all $u' \in U'_{J,w_J}$.

\begin{lemm} \label{lemm:calc2}
For all $n \in \N$, there is a natural $B_{J,w_J}^+$-equivariant morphism
\begin{equation*}
\Hc[n]\! \left( N''_{J,\Iw,0}, \pi_{\Iw}^{N'_{J,\Iw,0}} \right) \to \Hc[n]\! \left( N''_{J,\Iw,0}, \Clisc\! \left( U'_{J,w_J}, \sigma^{\Iw} \right) \right)
\end{equation*}
such that the Hecke action of $z$ on its kernel and cokernel is locally nilpotent.
\end{lemm}

\begin{proof}
We will implicitly make use of the isomorphism \eqref{isocind2}.
For each $n' \in N'_{J,\Iw}/N'_{J,\Iw,0}$, evaluation at $n'$ induces a natural $A$-linear surjection
\begin{equation*}
\ev_{n'} : \pi_{\Iw}^{N'_{J,\Iw,0}} \twoheadrightarrow \Clisc\! \left( U'_{J,w_J}, \sigma^{\Iw} \right).
\end{equation*}
We define a natural $A$-linear surjection
\begin{equation*}
\Ev \coloneqq \sum_{n' \in N'_{J,\Iw} / N'_{J,\Iw,0}} \ev_{n'} : \pi_{\Iw}^{N'_{J,\Iw,0}} \twoheadrightarrow \Clisc\! \left( U'_{J,w_J}, \sigma^{\Iw} \right).
\end{equation*}
We prove that $\Ev$ is $B_{J,w_J}^+$-equivariant: for any $f \in \pi_{\Iw}^{N'_{J,\Iw,0}}$ and $b \in B_{J,w_J}^+$, we have
\begin{align*}
\Ev \left( b \h f \right) &= \sum_{n' \in N'_{J,\Iw} / N'_{J,\Iw,0}} \sum_{n_0' \in N'_{J,\Iw,0} / bN'_{J,\Iw,0}b^{-1}} \left( n_0'b \cdot f \right) \left( n' \right) \\
&= \sum_{n' \in N'_{J,\Iw} / N'_{J,\Iw,0}} \sum_{n_0' \in N'_{J,\Iw,0} / bN'_{J,\Iw,0}b^{-1}} b \cdot f \left( b^{-1}n'n_0'b \right) \\
&= \sum_{n' \in N'_{J,\Iw} / bN'_{J,\Iw,0}b^{-1}} b \cdot f \left( b^{-1}n'b \right) \\
&= b \cdot \Ev \left( f \right)
\end{align*}
where the last equality results from the change of variable $n' \mapsto bn'b^{-1}$.
We prove that $\Ev$ is also $N''_{J,\Iw,0}$-equivariant: for any $f \in \pi_{\Iw}^{N'_{J,\Iw,0}}$, $n'' \in N''_{J,\Iw,0}$ and $u' \in U'_{J,w_J}$, we have
\begin{align*}
\Ev \left( n'' \cdot f \right) \left( u' \right) &= \sum_{n' \in N'_{J,\Iw} / N'_{J,\Iw,0}} \left( n'' \cdot f \right) \left( n' \right) \left( u' \right) \\
&= \sum_{n' \in N'_{J,\Iw} / N'_{J,\Iw,0}} n'' \cdot f \left( u'^{-1}n''^{-1}u'n'n'' \right) \left( u' \right) \\
&= n'' \cdot \Ev \left( f \right) \left( u' \right),
\end{align*}
where the last equality results from the fact that when $n'$ runs among $N'_{J,\Iw}/N'_{J,\Iw,0}$ we have $u'^{-1}n''^{-1}u'n'n'' = (u'^{-1}n''^{-1}u'n'')(n''^{-1}n'n'')$ with on the one hand $n''^{-1}n'n''$ running among $N'_{J,\Iw}/N'_{J,\Iw,0}$ and on the other hand $u'^{-1}n''^{-1}u'n'' \in N'_{J,\Iw}$ being constant.
We deduce that $\Ev$ induces natural $B_{J,w_J}^+$-equivariant morphisms in $N''_{J,\Iw,0}$-cohomology.

We prove that the Hecke action of $z$ on the kernel of $\Ev$ is locally nilpotent: for any $f \in \pi_{\Iw}^{N'_{J,\Iw,0}}$ there exists $i \in \N$ such that $\supp(f) \subseteq z^{-i} N'_{J,\Iw,0} z^i$ (since $z$ is strictly contracting $N'_{J,\Iw,0}$ which is open in $N'_{J,\Iw}$), thus for any $n' \in N'_{J,\Iw}/N'_{J,\Iw,0}$, we have
\begin{align*}
\left( z^i \h f \right) \left( n' \right) &= \sum_{n_0' \in N'_{J,\Iw,0} / z^i N'_{J,\Iw,0} z^{-i}} \left( n_0' z^i \cdot f \right) \left( n' \right) \\
&= \sum_{n_0' \in N'_{J,\Iw,0} / z^i N'_{J,\Iw,0} z^{-i}} z^i \cdot f \left( \left( z^{-i} n' z^i \right) \left( z^{-i} n_0' z^i \right) \right) \\
&=
\begin{cases}
z^i \cdot \Ev \left( f \right) &\text{if } n' \in N'_{J,\Iw,0}, \\
0 &\text{if } n' \not \in N'_{J,\Iw,0}.
\end{cases}
\end{align*}
Using the long exact sequence of $N''_{J,\Iw,0}$-cohomology, we deduce that the Hecke action of $z$ on the kernels and cokernels of the morphisms induced by $\Ev$ in $N''_{J,\Iw,0}$-cohomology is locally nilpotent.
\end{proof}

The subgroup $\Bb''_{J,w_J} \subseteq \Bb_{J,w_J}$ normalises $\Nb''_{J,\Iw}$ and the conjugation action coincides with the action induced by the quotient action of $\Bb_{J,w_J}$ on $\Nb''_{J,\Iw}$.
We define an open submonoid $B_{J,w_J}''^+ \subseteq B''_{J,w_J}$ by setting
\begin{equation*}
B_{J,w_J}''^+ \coloneqq \left\{ b'' \in B''_{J,w_J} \middlevert b''N''_{J,\Iw,0}b''^{-1} \subseteq N''_{J,\Iw,0} \right\}.
\end{equation*}

\begin{lemm} \label{lemm:B+}
We have $B_{J,w_J}^+ \subseteq B_{J,w_J}''^+ \ltimes U'_{J,w_J}$.
\end{lemm}

\begin{proof}
We have a semidirect product $\Bb_{J,w_J} = \Bb''_{J,w_J} \ltimes \Ub'_{J,w_J}$.
Let $b \in B_{J,w_J}^+$.
We write $b=b''u'$ with $b'' \in B''_{J,w_J}$ and $u' \in U'_{J,w_J}$.
We prove that $b'' \in B_{J,w_J}''^+$.
Let $n'' \in N''_{J,\Iw,0}$.
Proceeding as in the proof of Lemma \ref{lemm:action}, we see that $u'n''u'^{-1}n''^{-1} \in N'_{J,\Iw}$ so that $u'n''u'^{-1}=n'n''$ with $n' \in N'_{J,\Iw}$.
Thus $bn''b^{-1} = (b''n'b''^{-1})(b''n''b''^{-1}) \in N_{J,\Iw,0}$, and since $N_{J,\Iw,0}$ is totally decomposed, we deduce that $b''n''b''^{-1} \in N''_{J,\Iw,0}$.
\end{proof}

\begin{lemm} \label{lemm:calc3}
For all $n \in \N$, there is a natural $B_{J,w_J}''^+ \ltimes U'_{J,w_J}$-equivariant isomorphism
\begin{equation*}
\Hc[n]\! \left( N''_{J,\Iw,0}, \Clisc\! \left( U'_{J,w_J}, \sigma^{\Iw} \right) \right) \cong \Clisc\! \left( U'_{J,w_J}, \Hc[n]\! \left( N''_{J,\Iw,0}, \sigma^{\Iw} \right) \right).
\end{equation*}
\end{lemm}

\begin{proof}
Let $\sigmat$ be a $B_{J,w_J}''^+ \ltimes N''_{J,\Iw,0}$-representation.
The $A$-modules $\Hc(N''_{J,\Iw,0},\Clisc(U'_{J,w_J},\sigmat))$ and $\Clisc(U'_{J,w_J},\Hc(N''_{J,\Iw,0},\sigmat))$ are naturally $B_{J,w_J}''^+ \ltimes U'_{J,w_J}$-representations.
The identity of $\Clisc(U'_{J,w_J},\sigmat)$ induces a natural $U'_{J,w_J}$-equivariant isomorphism
\begin{equation*}
\iota : \Clisc\! \left( U'_{J,w_J}, \sigmat \right)^{N''_{J,\Iw,0}} \iso \Clisc\! \left( U'_{J,w_J}, \sigmat^{N''_{J,\Iw,0}} \right).
\end{equation*}
We prove that $\iota$ is also $B_{J,w_J}''^+$-equivariant: for any $f \in \Clisc(U'_{J,w_J},\sigmat)^{N''_{J,\Iw,0}}$, $b'' \in B_{J,w_J}''^+$ and $u' \in U'_{J,w_J}$, we have
\begin{align*}
\iota \left( b'' \h f \right) \left( u' \right) &= \sum_{n'' \in N''_{J,\Iw,0} / b''N''_{J,\Iw,0}b''^{-1}} \iota \left( n'' b'' \cdot f \right) \left( u' \right) \\
&= \sum_{n'' \in N''_{J,\Iw,0} / b''N''_{J,\Iw,0}b''^{-1}} n'' b'' \cdot \iota \left( f \right) \left(b''^{-1} u' b'' \right) \\
&= b'' \h \iota \left( f \right) \left( b''^{-1} u' b'' \right) \\
&= \left( b'' \cdot \iota \left( f \right) \right) \left( u' \right).
\end{align*}
We will prove that deriving $\iota$ yields the desired isomorphisms with $\sigmat=\sigma^{\Iw}$.

The functor $\Clisc(U'_{J,w_J},-)$ is $A$-linear and exact and the $\delta$-functor $\Hc(N''_{J,\Iw,0},-)$ is universal, thus denoting by $\Rc$ the right derived functors on the category $\Mod_{B_{J,w_J}''^+ \ltimes N''_{J,\Iw,0}}^\sm(A)$, we have morphisms of functors
\begin{gather*}
\Rc\! \left( \Clisc\! \left( U'_{J,w_J}, \left( - \right)^{N''_{J,\Iw,0}} \right) \right) \cong \Clisc\! \left( U'_{J,w_J}, \Hc\! \left( N''_{J,\Iw,0}, - \right) \right), \\
\Rc\! \left( \Clisc\! \left( U'_{J,w_J}, - \right)^{N''_{J,\Iw,0}} \right) \to \Hc\! \left( N''_{J,\Iw,0}, \Clisc\! \left( U'_{J,w_J}, - \right) \right).
\end{gather*}
In order to show that the second one is also an isomorphism, it is enough to prove that $\Clisc(U'_{J,w_J},-)$ takes injective objects of $\Mod_{B_{J,w_J}''^+ \ltimes N''_{J,\Iw,0}}^\sm(A)$ to $N''_{J,\Iw,0}$-acyclic objects.
If $\sigmat$ is an $A$-module, then we have a natural $N''_{J,\Iw,0}$-equivariant isomorphism
\begin{equation*}
\Clisc\! \left( U'_{J,w_J}, \Clis\! \left( N''_{J,\Iw,0}, \sigmat \right) \right) \cong \Clis\! \left( N''_{J,\Iw,0}, \Clisc\! \left( U'_{J,w_J}, \sigmat \right) \right),
\end{equation*}
so that $\Clisc(U'_{J,w_J},\Clis( N''_{J,\Iw,0}, \sigmat))$ is $N''_{J,\Iw,0}$-acyclic.
Now if $\sigmat$ is an injective object of $\Mod_{B_{J,w_J}''^+ \ltimes N''_{J,\Iw,0}}^\sm(A)$, then it is also an injective object of $\Mod_{N''_{J,\Iw,0}}^\sm(A)$ (cf. \cite[Proposition 2.1.11]{Em2} and \cite[Lemme 3.1.1]{JH}), thus the natural $N''_{J,\Iw,0}$-equivariant injection $\sigmat \hookrightarrow \Clis(N''_{J,\Iw,0},\sigmat)$ defined by $v \mapsto (n'' \mapsto n'' \cdot v)$ admits an $N''_{J,\Iw,0}$-equivariant retraction, so that $\sigmat$ is a direct factor of $\Clis(N''_{J,\Iw,0},\sigmat)$, and therefore $\sigmat$ is $N''_{J,\Iw,0}$-acyclic.
\end{proof}

We now assume that $\sigma$ is locally admissible.

\begin{lemm} \label{lemm:calc4}
For all $n \in \N$, there is a natural $B_{J,w_J}''^+ \ltimes U'_{J,w_J}$-equivariant morphism
\begin{equation*}
\Clisc\! \left( U'_{J,w_J}, \Hc[n]\! \left( N''_{J,\Iw,0}, \sigma^{\Iw} \right) \right) \to \cind_{L_J \cap P_{J \cap \Iw^J{}^{-1}(I)}^-}^{L_J \cap P_{J \cap \Iw^J{}^{-1}(I)}^- w_J B_J} \left( \HOrd[n]_{L_I \cap P_{I \cap \Iw^J(J)}} \sigma \right)^{\Iw^J}
\end{equation*}
such that the action of $z$ on its kernel and cokernel is locally nilpotent.
\end{lemm}

\begin{proof}
We have natural $B_{J,w_J}''^+$-equivariant isomorphisms
\begin{equation} \label{Iw}
\Hc\! \left( N''_{J,\Iw,0}, \sigma^{\Iw} \right) \cong \Hc\! \left( \Iw N''_{J,\Iw,0} \Iw^{-1}, \sigma \right)^{\Iw}.
\end{equation}
Since $\Iw \Nb''_{J,\Iw} \Iw^{-1} = \Ub_I \cap \Iw^J \Nb_J \Iw^J{}^{-1}$ is the unipotent radical of $\Lb_I \cap \Pb_{I \cap \Iw^J(J)}$ (see Lemma \ref{lemm:IwJ} (iii)), we define an open submonoid of $L_{I \cap \Iw^J(J)}$ by setting
\begin{equation*}
L_{I \cap \Iw^J(J)}^+ \coloneqq \left\{ l \in L_{I \cap \Iw^J(J)} \middlevert l \Iw N''_{J,\Iw,0} \Iw^{-1} l^{-1} \subseteq \Iw N''_{J,\Iw,0} \Iw^{-1} \right\}.
\end{equation*}
We have $\Iw B_{J,w_J}''^+ \Iw^{-1} = \Iw B''_{J,w_J} \Iw^{-1} \cap L_{I \cap \Iw^J(J)}^+$.
We let $Z_{I \cap \Iw^J(J)}^+ \subseteq Z_{I \cap \Iw^J(J)}$ be the open submonoid $Z_{I \cap \Iw^J(J)} \cap L_{I \cap \Iw^J(J)}^+$.
Since $\sigma$ is locally admissible, $\Hc(\Iw N''_{J,\Iw,0} \Iw^{-1},\sigma)$ is locally $Z_{I \cap \Iw^J(J)}^+$-finite by \cite[Theorem 3.4.7 (1)]{Em2}, and thus locally $\Iw z \Iw^{-1}$-finite.
Note that $\Iw z \Iw^{-1} \in Z_{I \cap \Iw^J(J)}^+$ is strictly contracting $\Iw N''_{J,\Iw,0} \Iw^{-1}$.
Therefore, localising with respect to $(\Iw z \Iw^{-1})^\N$ gives rise to $L_{I \cap \Iw^J(J)}^+$-equivariant morphisms
\begin{equation*}
\Hc\! \left( \Iw N''_{J,\Iw,0} \Iw^{-1}, \sigma \right) \to \HOrd_{L_I \cap P_{I \cap \Iw^J(J)}} \sigma
\end{equation*}
such that the action of $\Iw z \Iw^{-1}$ on their kernels and cokernel is locally nilpotent (see Remark \ref{rema:Loczt}).
Using \eqref{Iw}, we deduce $B_{J,w_J}''^+$-equivariant morphisms
\begin{equation*}
\Hc\! \left( N''_{J,\Iw,0}, \sigma^{\Iw} \right) \to \left( \HOrd_{L_I \cap P_{I \cap \Iw^J(J)}} \sigma \right)^{\Iw}
\end{equation*}
such that the action of $z$ on their kernels and cokernels is locally nilpotent.
Applying the functor $\Clisc(U'_{J,w_J},-)$, we obtain $B_{J,w_J}''^+ \ltimes U'_{J,w_J}$-equivariant morphisms
\begin{equation*}
\Clisc\! \left( U'_{J,w_J}, \Hc\! \left( N''_{J,\Iw,0}, \sigma^{\Iw} \right) \right) \to \Clisc\! \left( U'_{J,w_J}, \left( \HOrd_{L_I \cap P_{I \cap \Iw^J(J)}} \sigma \right)^{\Iw} \right)
\end{equation*}
such that the action of $z$ on their kernel and cokernel is still locally nilpotent (because the functions in their sources and targets have finite images).
We conclude using the inverse of the $B_{J,w_J}$-equivariant isomorphism \eqref{isocindJ} with $\sigmat=(\HOrd_{L_I \cap P_{I \cap \Iw^J(J)}} \sigma)^{\Iw^J}$.
\end{proof}

We combine the previous results into the following one.

\begin{prop} \label{prop:calc}
Let $\sigma$ be a locally admissible $L_I$-representation and $\Iw \in \IW$.
We write $\Iw = \Iw^J w_J$ with $\Iw^J \in \IW^J$ and $w_J \in W_J$.
For all $n \in \N$, there is a natural $B_{J,w_J}^+$-equivariant morphism
\begin{equation*}
\Hc[n] \left( N_{J,\Iw,0}, \cind_{P_I^-}^{P_I^- \Iw B} \sigma \right) \to \cind_{L_J \cap P_{J \cap {\Iw^J}^{-1}(I)}^-}^{L_J \cap P_{J \cap {\Iw^J}^{-1}(I)}^- w_J B_J} \left( \HOrd[n]_{L_I \cap P_{I \cap \Iw^J(J)}} \sigma \right)^{\Iw^J}
\end{equation*}
such that the action of $z$ on its kernel and cokernel is locally nilpotent.
Furthermore, this morphism is even $L_J^+ \cap P_{J \cap {\Iw^J}^{-1}(I)}$-equivariant when $w_J=1$ (see Remark \ref{rema:Fil1}).
\end{prop}

\begin{proof}
Combining Lemmas \ref{lemm:calc1}, \ref{lemm:calc2}, \ref{lemm:calc3}, \ref{lemm:calc4} and using Lemma \ref{lemm:B+}, we obtain the desired morphism.
If $w_J=1$, then the previous lemmas and their proofs are true verbatim with $\Lb_J \cap \Pb_{J \cap {\Iw^J}^{-1}(I)}$ and $\Lb_{J \cap {\Iw^J}^{-1}(I)}$ instead of $\Bb_{J,w_J}$ and $\Bb''_{J,w_J}$ respectively (see Remark \ref{rema:isocind}), thus the morphism is $L_J^+ \cap P_{J \cap {\Iw^J}^{-1}(I)}$-equivariant.
\end{proof}

\subsection{Computations on parabolically induced representations} \label{ssec:HOrd}

Let $I,J \subseteq \Delta$, $\sigma$ be a locally admissible $L_I$-representation and $n \in \N$.
For any lower set $\IW^J_1 \subseteq \IW^J$, the natural $P_J$-equivariant injection $\Fil_{P_J}^{\IW^J_1}(\Ind_{P_I^-}^G \sigma) \hookrightarrow \Ind_{P_I^-}^G \sigma$ induces an $L_J$-equivariant morphism
\begin{equation} \label{injIWJ}
\HOrd[n]_{P_J} \left( \Fil_{P_J}^{\IW^J_1} \left( \Ind_{P_I^-}^G \sigma \right) \right) \to \HOrd[n]_{P_J} \left( \Ind_{P_I^-}^G \sigma \right),
\end{equation}
and by taking its image we define an $L_J$-subrepresentation
\begin{equation*}
\Fil_{P_J}^{\IW^J_1} \left( \HOrd[n]_{P_J} \left( \Ind_{P_I^-}^G \sigma \right) \right) \subseteq \HOrd[n]_{P_J} \left( \Ind_{P_I^-}^G \sigma \right).
\end{equation*}

\begin{prop} \label{prop:FilPJ}
The $L_J$-subrepresentations $\Fil_{P_J}^\bullet(\HOrd[n]_{P_J}(\Ind_{P_I^-}^G \sigma))$ form a natural filtration of $\HOrd[n]_{P_J} (\Ind_{P_I^-}^G \sigma)$ indexed by $\IW^J$.
Furthermore, for all $\Iw^J \in \IW^J$ there is a natural $L_J$-equivariant isomorphism
\begin{equation*}
\Gr_{P_J}^{\Iw^J} \left( \HOrd[n]_{P_J} \left( \Ind_{P_I^-}^G \sigma \right) \right) \cong \HOrd[n]_{P_J} \left( \cind_{P_I^-}^{P_I^- \Iw^J P_J} \sigma \right).
\end{equation*}
\end{prop}

\begin{proof}
First, we prove for any lower sets $\IW^J_2 \subseteq \IW^J_1 \subseteq \IW^J$, the short exact sequence of $P_J$-representations
\begin{multline} \label{SEFil}
0 \to \Fil_{P_J}^{\IW^J_2} \left( \Ind_{P_I^-}^G \sigma \right) \to \Fil_{P_J}^{\IW^J_1} \left( \Ind_{P_I^-}^G \sigma \right) \\
\to \Fil_{P_J}^{\IW^J_1} \left( \Ind_{P_I^-}^G \sigma \right) / \Fil_{P_J}^{\IW^J_2} \left( \Ind_{P_I^-}^G \sigma \right) \to 0
\end{multline}
induces a short exact sequence of $L_J$-representations
\begin{multline} \label{SEFilHOrd}
0 \to \HOrd[n]_{P_J} \left( \Fil_{P_J}^{\IW^J_2} \left( \Ind_{P_I^-}^G \sigma \right) \right) \to \HOrd[n]_{P_J} \left( \Fil_{P_J}^{\IW^J_1} \left( \Ind_{P_I^-}^G \sigma \right) \right) \\
\to \HOrd[n]_{P_J} \left( \Fil_{P_J}^{\IW^J_1} \left( \Ind_{P_I^-}^G \sigma \right) / \Fil_{P_J}^{\IW^J_2} \left( \Ind_{P_I^-}^G \sigma \right) \right) \to 0.
\end{multline}
In particular, \eqref{injIWJ} is injective and \eqref{GrPJ} induces the isomorphism in the statement.

Let $N_{J,0} \subseteq N_J$, $L_J^+ \subseteq L_J$, $Z_J^+ \subseteq Z_J$ and $z \in Z_J^+$ be as in §~\ref{ssec:calc}.
Proceeding as in the proof of \cite[Proposition 2.2.3]{JH}, we see that the first non-trivial morphism of \eqref{SEFil} induces an injection in $N_{J,0}$-cohomology.
Using the long exact sequence of $N_{J,0}$-cohomology, we deduce that \eqref{SEFil} induces a short exact sequence of $L_J^+$-representations
\begin{multline} \label{SEFilH}
0 \to \Hc[n]\! \left( N_{J,0}, \Fil_{P_J}^{\IW^J_2} \left( \Ind_{P_I^-}^G \sigma \right) \right) \to \Hc[n]\! \left( N_{J,0}, \Fil_{P_J}^{\IW^J_1} \left( \Ind_{P_I^-}^G \sigma \right) \right) \\
\to \Hc[n]\! \left( N_{J,0}, \Fil_{P_J}^{\IW^J_1} \left( \Ind_{P_I^-}^G \sigma \right) / \Fil_{P_J}^{\IW^J_2} \left( \Ind_{P_I^-}^G \sigma \right) \right) \to 0.
\end{multline}
Since $\sigma$ is locally admissible, $\Ind_{P_I^-}^G \sigma$ is locally admissible (cf. \cite[Proposition 4.1.7]{Em1}), thus $\Hc(N_{J,0},\Ind_{P_I^-}^G \sigma)$ is locally $Z_J^+$-finite by \cite[Theorem 3.4.7 (1)]{Em2}.
We deduce that each term of \eqref{SEFilH} is locally $Z_J^+$-finite (as a subquotient).
We conclude that localising \eqref{SEFilH} with respect to $z^\N$ yields \eqref{SEFilHOrd} (see Remark \ref{rema:Loczt}).

We now prove that $\Fil_{P_J}^\bullet(\HOrd[n]_{P_J}(\Ind_{P_I^-}^G \sigma))$ is a filtration of $\HOrd[n]_{P_J}(\Ind_{P_I^-}^G \sigma)$ indexed by $\IW^J$.
Since $\IW^J$ is finite and $\Fil_{P_J}^\bullet(\HOrd[n]_{P_J}(\Ind_{P_I^-}^G \sigma))$ is inclusion-preserving with $\Fil_{P_J}^\emptyset(\HOrd[n]_{P_J}(\Ind_{P_I^-}^G \sigma))=0$ and $\Fil_{P_J}^{\IW^J}(\HOrd[n]_{P_J}(\Ind_{P_I^-}^G \sigma))=\HOrd[n]_{P_J}(\Ind_{P_I^-}^G \sigma)$ by construction, it remains to prove that for any lower sets $\IW^J_1, \IW^J_2 \subseteq \IW^J$, the natural short exact sequence of $P_J$-representations
\begin{multline*}
0 \to \Fil_{P_J}^{\IW^J_1 \cap \IW^J_2} \left( \Ind_{P_I^-}^G \sigma \right) \to \Fil_{P_J}^{\IW^J_1} \left( \Ind_{P_I^-}^G \sigma \right) \oplus \Fil_{P_J}^{\IW^J_2} \left( \Ind_{P_I^-}^G \sigma \right) \\
\to \Fil_{P_J}^{\IW^J_1 \cup \IW^J_2} \left( \Ind_{P_I^-}^G \sigma \right) \to 0
\end{multline*}
induces a short exact sequence of $L_J$-representations
\begin{multline*}
0 \to \Fil_{P_J}^{\IW^J_1 \cap \IW^J_2} \left( \HOrd[n]_{P_J} \left( \Ind_{P_I^-}^G \sigma \right) \right) \\
\to \Fil_{P_J}^{\IW^J_1} \left( \HOrd[n]_{P_J} \left( \Ind_{P_I^-}^G \sigma \right) \right) \oplus \Fil_{P_J}^{\IW^J_2} \left( \HOrd[n]_{P_J} \left( \Ind_{P_I^-}^G \sigma \right) \right) \\
\to \Fil_{P_J}^{\IW^J_1 \cup \IW^J_2} \left( \HOrd[n]_{P_J} \left( \Ind_{P_I^-}^G \sigma \right) \right) \to 0.
\end{multline*}
This follows from the same arguments as above.
\end{proof}

Let $\Iw^J \in \IW^J$.
For any lower set $W'_J \subseteq \JIW_J$, the natural $B$-equivariant (resp. $P_{J \cap \Iw^J{}^{-1}(I)}$-equivariant when $W_J'=\{1\}$, see Remark \ref{rema:Fil1}) injection $\Fil_B^{W_J'}(\cind_{P_I^-}^{P_I^- \Iw^J P_J} \sigma) \allowbreak \hookrightarrow \cind_{P_I^-}^{P_I^- \Iw^J P_J} \sigma$ induces a $B_J$-equivariant (resp. $L_J \cap P_{J \cap \Iw^J{}^{-1}(I)}$-equivariant when $W_J'=\{1\}$) morphism
\begin{equation} \label{injWJ}
\HOrd[n]_{P_J} \left( \Fil_B^{W_J'} \left( \cind_{P_I^-}^{P_I^- \Iw^J P_J} \sigma \right) \right) \to \HOrd[n]_{P_J} \left( \cind_{P_I^-}^{P_I^- \Iw^J P_J} \sigma \right),
\end{equation}
and by taking its image we define a $B_J$-subrepresentation (resp. $L_J \cap P_{J \cap \Iw^J{}^{-1}(I)}$-sub\-rep\-res\-ent\-a\-tion when $W_J'=\{1\}$)
\begin{equation*}
\Fil_B^{W_J'} \left( \HOrd[n]_{P_J} \left( \cind_{P_I^-}^{P_I^- \Iw^J P_J} \sigma \right) \right) \subseteq \HOrd[n]_{P_J} \left( \cind_{P_I^-}^{P_I^- \Iw^J P_J} \sigma \right).
\end{equation*}
Proceeding as in the proof of Proposition \ref{prop:FilPJ} and using \eqref{GrB}, we prove that \eqref{injWJ} is injective and the following result.

\begin{prop} \label{prop:FilB}
The $B_J$-subrepresentations $\Fil_B^\bullet(\HOrd[n]_{P_J}(\cind_{P_I^-}^{P_I^- \Iw^J P_J} \sigma))$ form a natural filtration of $\HOrd[n]_{P_J}(\cind_{P_I^-}^{P_I^- \Iw^J P_J} \sigma)$ indexed by $\JIW_J$.
Furthermore, for all $w_J \in \JIW_J$ there is a natural $B_J$-equivariant isomorphism
\begin{equation*}
\Gr_B^{w_J} \left( \HOrd[n]_{P_J} \left( \cind_{P_I^-}^{P_I^- \Iw^J P_J} \sigma \right) \right) \cong \HOrd[n]_{P_J} \left( \cind_{P_I^-}^{P_I^- \Iw^J w_J B} \sigma \right)
\end{equation*}
which is even $L_J \cap P_{J \cap \Iw^J{}^{-1}(I)}$-equivariant when $w_J=1$ (see Remark \ref{rema:Fil1}).
\end{prop}

We now state the main result of this section using Notation \ref{nota:dw} and Remark \ref{rema:dIwJ}.

\begin{theo} \label{theo:HOrd}
Let $\sigma$ be a locally admissible $L_I$-representation, $\Iw^J \in \IW^J$ and $n \in \N$.
For all $w_J \in \JIW_J$, there is a natural $B_{J,w_J}$-equivariant isomorphism
\begin{multline*}
\Gr_B^{w_J} \left( \HOrd[n]_{P_J} \left( \cind_{P_I^-}^{P_I^- \Iw^J P_J} \sigma \right) \right) \\
\cong \Gr_{B_J}^{w_J} \left( \Ind_{L_J \cap P_{J \cap \Iw^J{}^{-1}(I)}^-}^{L_J} \left( \left( \HOrd[n-{[F:\Qp]} d_{\Iw^J}]_{L_I \cap P_{I \cap \Iw^J(J)}} \sigma \right)^{\Iw^J} \otimes \left( \omega^{-1} \circ \delta_{\Iw^J} \right) \right) \right)
\end{multline*}
which is even $L_J \cap P_{J \cap \Iw^J{}^{-1}(I)}$-equivariant when $w_J=1$ (see Remark \ref{rema:Fil1}).
\end{theo}

\begin{proof}
We use the notation of §~\ref{ssec:isocind}.
We let $w_J \in \JIW_J$ and we put $\Iw \coloneqq \Iw^J w_J$.
We let $N_{J,0} \subseteq N_J$, $L_J^+ \subseteq L_J$, $Z_J^+ \subseteq Z_J$, $z \in Z_J^+$ and $\pi_{\Iw}$ be as in §~\ref{ssec:calc}.
In the course of the proof of Proposition \ref{prop:FilB}, we see that $\Hc[n](N_{J,0},\pi_{\Iw})$ is locally $Z_J^+$-finite (as we saw it for $\Hc[n](N_{J,0},\cind_{P_I^-}^{P_I^- \Iw^J P_J} \sigma)$ in the course of the proof of Proposition \ref{prop:FilPJ}).

Since $\sigma$ is locally admissible, the $L_{I \cap \Iw^J(J)}$-representations $\HOrd_{L_I \cap P_{I \cap \Iw^J(J)}} \sigma$ are locally admissible by \cite[Theorem 3.4.7 (2)]{Em2}, thus locally $Z_{I \cap \Iw^J(J)}$-finite by \cite[Lemma 2.3.4]{Em1}.
Therefore, the $B_J$-representations
\begin{equation*}
\cind_{L_J \cap P_{J \cap {\Iw^J}^{-1}(I)}^-}^{L_J \cap P_{J \cap {\Iw^J}^{-1}(I)}^- w_J B_J}(\HOrd_{L_I \cap P_{I \cap \Iw^J(J)}} \sigma)^{\Iw^J}
\end{equation*}
are locally $Z_J$-finite, thus locally $z$-finite.
We deduce from Proposition \ref{prop:calc} that the $B_{J,w_J}^+$-representations $\Hc(N_{J,\Iw,0},\pi_{\Iw})$ are locally $z$-finite and that there is a natural $B_{J,w_J}^+$-equivariant (resp. $L_J^+ \cap P_{J \cap {\Iw^J}^{-1}(I)}$-equivariant when $w_J=1$) morphism
\begin{multline} \label{calc}
\Hc[n-{[F:\Qp]} d_{\Iw^J}]\! \left( N_{J,\Iw,0}, \pi_{\Iw} \right) \otimes \left( \omega^{-1} \circ \delta_{\Iw^J} \right)^{w_J} \\
\to \left( \cind_{L_J \cap P_{J \cap {\Iw^J}^{-1}(I)}^-}^{L_J \cap P_{J \cap {\Iw^J}^{-1}(I)}^- w_J B_J} \left( \HOrd[n-{[F:\Qp]} d_{\Iw^J}]_{L_I \cap P_{I \cap \Iw^J(J)}} \sigma \right)^{\Iw^J} \right) \otimes \left( \omega^{-1} \circ \delta_{\Iw^J} \right)^{w_J}
\end{multline}
such that the action of $z$ on its kernel and cokernel is locally nilpotent.

Using Proposition \ref{prop:dev} with $\Lbt=\Bb_{J,w_J}$ (resp. $\Lbt=\Lb_J \cap \Pb_{J \cap \Iw^J{}^{-1}(I)}$ when $w_J=1$), $\Nbt=\Nb_J$, $\Nbt'=\Nb_{J,\Iw}$ (so that $\dt-\dt' = d_{\Iw^J}$ and $\deltat-\deltat' = w_J^{-1}(\delta_{\Iw^J})$ since conjugation by $w_J$ induces an isomorphism of $F$-varieties $\Nb_J / \Nb_{J,\Iw} \iso \Nb_J / \Nb_{J,\Iw^J}$), $\zt=z$ and $\pi=\pi_{\Iw}$, we deduce a natural $B_{J,w_J}^+$-equivariant (resp. $L_J^+ \cap P_{J \cap \Iw^J{}^{-1}(I)}$-equivariant when $w_J=1$) morphism
\begin{equation} \label{calc0}
\Hc[n-{[F:\Qp]} d_{\Iw^J}]\! \left( N_{J,\Iw,0}, \pi_{\Iw} \right) \otimes \left( \omega^{-1} \circ \delta_{\Iw^J} \right)^{w_J} \to \Hc[n]\! \left( N_{J,0}, \pi_{\Iw} \right)
\end{equation}
and the Hecke action of $z$ on its kernel and cokernel is locally nilpotent.

Using Proposition \ref{prop:FilB}, \eqref{GrBJ} with $\sigmat=(\HOrd[n-{[F:\Qp]} d_{\Iw^J}]_{L_I \cap P_{I \cap \Iw^J(J)}} \sigma)^{\Iw^J} \otimes (\omega^{-1} \circ \delta_{\Iw^J})$ and the natural $B_J$-equivariant (resp. $L_J \cap P_{J \cap \Iw^J{}^{-1}(I)}$-equivariant when $w_J=1$) isomorphism
\begin{multline*}
\left( \cind_{L_J \cap P_{J \cap {\Iw^J}^{-1}(I)}^-}^{L_J \cap P_{J \cap {\Iw^J}^{-1}(I)}^- w_J B_J} \left( \HOrd[n-{[F:\Qp]} d_{\Iw^J}]_{L_I \cap P_{I \cap \Iw^J(J)}} \sigma \right)^{\Iw^J} \right) \otimes \left( \omega^{-1} \circ \delta_{\Iw^J} \right)^{w_J} \\
\iso \cind_{L_J \cap P_{J \cap {\Iw^J}^{-1}(I)}^-}^{L_J \cap P_{J \cap {\Iw^J}^{-1}(I)}^- w_J B_J} \left( \left( \HOrd[n-{[F:\Qp]} d_{\Iw^J}]_{L_I \cap P_{I \cap \Iw^J(J)}} \sigma \right)^{\Iw^J} \otimes \left( \omega^{-1} \circ \delta_{\Iw^J} \right) \right),
\end{multline*}
the localisation of \eqref{calc} with respect to $z^\N$ and the inverse of the localisation of \eqref{calc0} with respect to $z^\N$ yield the desired isomorphism (see Remark \ref{rema:Loczt}).
\end{proof}

In particular with $w_J=1$ and $\sigmat \coloneqq (\HOrd[n-{[F:\Qp]} d_{\Iw^J}]_{L_I \cap P_{I \cap \Iw^J(J)}} \sigma)^{\Iw^J} \otimes (\omega^{-1} \circ \delta_{\Iw^J})$, there is a natural $L_J \cap P_{J \cap \Iw^J{}^{-1}(I)}$-equivariant injection
\begin{equation*}
\Gr_{B_J}^1 \left( \Ind_{L_J \cap P_{J \cap \Iw^J{}^{-1}(I)}^-}^{L_J} \sigmat \right) \hookrightarrow \HOrd[n]_{P_J} \left( \cind_{P_I^-}^{P_I^- \Iw^J P_J} \sigma \right),
\end{equation*}
hence a natural $L_J$-equivariant morphism
\begin{equation*}
A[L_J] \otimes_{A[L_J \cap P_{J \cap \Iw^J{}^{-1}(I)}]} \Gr_{B_J}^1 \left( \Ind_{L_J \cap P_{J \cap \Iw^J{}^{-1}(I)}^-}^{L_J} \sigmat \right) \to \HOrd[n]_{P_J} \left( \cind_{P_I^-}^{P_I^- \Iw^J P_J} \sigma \right).
\end{equation*}
In the proof of \cite[Theorem 4.4.6]{Em1}, it is shown that such a morphism factors uniquely through the natural $L_J$-equivariant surjection
\begin{equation*}
A[L_J] \otimes_{A[L_J \cap P_{J \cap \Iw^J{}^{-1}(I)}]} \Gr_{B_J}^1 \left( \Ind_{L_J \cap P_{J \cap \Iw^J{}^{-1}(I)}^-}^{L_J} \sigmat \right) \twoheadrightarrow \Ind_{L_J \cap P_{J \cap \Iw^J{}^{-1}(I)}^-}^{L_J} \sigmat.
\end{equation*}
Thus, the previous injection naturally extends to an $L_J$-equivariant morphism
\begin{multline} \label{HOrd}
\Ind_{L_J \cap P_{J \cap \Iw^J{}^{-1}(I)}^-}^{L_J} \left( \left( \HOrd[n-{[F:\Qp]} d_{\Iw^J}]_{L_I \cap P_{I \cap \Iw^J(J)}} \sigma \right)^{\Iw^J} \otimes \left( \omega^{-1} \circ \delta_{\Iw^J} \right) \right) \\
\to \HOrd[n]_{P_J} \left( \cind_{P_I^-}^{P_I^- \Iw^J P_J} \sigma \right).
\end{multline}

\begin{conj} \label{conj:HOrd}
The natural morphism \eqref{HOrd} is an isomorphism.
\end{conj}

We prove Conjecture \ref{conj:HOrd} in some special cases.

\begin{prop} \phantomsection \label{prop:HOrd1}
\begin{enumerate}
\item If $\HOrd[n-{[F:\Qp]} d_{\Iw^J}]_{L_I \cap P_{I \cap \Iw^J(J)}} \sigma = 0$, then
\begin{equation*}
\HOrd[n]_{P_J} \left( \cind_{P_I^-}^{P_I^- \Iw^J P_J} \sigma \right) = 0.
\end{equation*}
\item If $\Iw^J(J) \subseteq I$, then \eqref{HOrd} is a natural $L_J$-equivariant isomorphism
\begin{multline*}
\left( \HOrd[n-{[F:\Qp]} d_{\Iw^J}]_{L_I \cap P_{I \cap \Iw^J(J)}} \sigma \right)^{\Iw^J} \otimes \left( \omega^{-1} \circ \delta_{\Iw^J} \right) \\
\iso \HOrd[n]_{P_J} \left( \cind_{P_I^-}^{P_I^- \Iw^J P_J} \sigma \right).
\end{multline*}
\item If $n=0$ and $\Iw^J=1$, then \eqref{HOrd} is a natural $L_J$-equivariant isomorphism
\begin{equation*}
\Ind_{L_J \cap P_I^-}^{L_J} \left( \Ord_{L_I \cap P_J} \sigma \right) \iso \Ord_{P_J} \left( \cind_{P_I^-}^{P_I^- P_J} \sigma \right).
\end{equation*}
\end{enumerate}
\end{prop}

\begin{proof}
We use Theorem \ref{theo:HOrd}: if $\HOrd[n-{[F:\Qp]} d_{\Iw^J}]_{L_I \cap P_{I \cap \Iw^J(J)}} \sigma = 0$, then we deduce that $\Gr_B^\bullet(\HOrd[n]_{P_J}(\cind_{P_I^-}^{P_I^- \Iw^J P_J} \sigma))=0$, hence (i); if $\Iw^J(J) \subseteq I$, then we deduce from \cite[Proposition 3.6.1]{Em2} that $\Gr_B^\bullet(\HOrd[n]_{P_J}(\cind_{P_I^-}^{P_I^- \Iw^J P_J} \sigma))$ is concentrated in degree $1$, thus \eqref{HOrd} is an isomorphism, hence (ii).

We now prove (iii).
Since all the functors involved commute with inductive limits, we reduce to the case where $\sigma$ is admissible.
By \cite[Corollaries 4.13 and 5.9]{AHV}, there is a natural $L_J$-equivariant isomorphism
\begin{equation} \label{AHV}
\Ind_{L_J \cap P_I^-}^{L_J} \left( \Ord_{L_I \cap P_J} \sigma \right) \iso \Ord_{P_J} \left( \Ind_{P_I^-}^G \sigma \right).
\end{equation}
Using (i), we deduce from Proposition \ref{prop:FilPJ} with $n=0$ that $\Gr_{P_J}^\bullet(\Ord_{P_J}(\Ind_{P_I^-}^G \sigma))$ is concentrated in degree $1$, hence a natural $L_J$-equivariant isomorphism
\begin{equation} \label{Ord}
\Ord_{P_J} \left( \cind_{P_I^-}^{P_I^- P_J} \sigma \right) \iso \Ord_{P_J} \left( \Ind_{P_I^-}^G \sigma \right).
\end{equation}
The composition of \eqref{HOrd} with $n=0$ and $\Iw^J=1$, \eqref{Ord} and the inverse of \eqref{AHV} yields an $L_J$-equivariant endomorphism $\varphi$ of $\Ind_{L_J \cap P_I^-}^{L_J} (\Ord_{L_I \cap P_J} \sigma)$ which is injective in restriction to $\Fil_{B_J}^1(\Ind_{L_J \cap P_I^-}^{L_J} (\Ord_{L_I \cap P_J} \sigma))$.
From \cite[Lemma 4.3.1 and Proposition 4.3.4]{Em1} and the left-exactness of $\Ord_{L_I \cap P_J}$, we deduce that $\Ord_{L_J \cap P_I} \varphi$ is an injective $L_{I \cap J}$-equivariant endomorphism of $\Ord_{L_I \cap P_J} \sigma$.
Since the latter is admissible by \cite[Theorem 3.3.3]{Em1}, it is Artinian (see §~\ref{ssec:pro} below), and thus co-Hopfian so that $\Ord_{L_I \cap P_J} \varphi$ is an isomorphism.
We deduce that $\varphi$ is an isomorphism using \cite[Proposition 4.3.4 and Theorem 4.4.6]{Em1}.
We conclude that \eqref{HOrd} with $n=0$ and $\Iw^J=1$ is an isomorphism as in the statement.
\end{proof}

\begin{rema} \label{rema:IwJ=1}
Let $\ROrd_{L_I \cap P_J}$ denote the derived functors of $\Ord_{L_I \cap P_J}$ on $\Mod_{L_I}^\ladm(A)$.
By universality of derived functors, the isomorphism in (iii) extends uniquely to a morphism of $\delta$-functors
\begin{equation} \label{homdelta}
\Ind_{L_J \cap P_I^-}^{L_J} \circ \ROrd_{L_I \cap P_J} \to \HOrd_{P_J} \circ \cind_{P_I^-}^{P_I^- P_J}
\end{equation}
(the left-hand side is the derived functor of $\Ind_{L_J \cap P_I^-}^{L_J} \circ \Ord_{L_I \cap P_J}$ by exactness of $\Ind_{L_J \cap P_I^-}^{L_J}$, and the right-hand side is a $\delta$-functor by the same arguments as in the proof of Proposition \ref{prop:FilPJ}).
Now, assume that \cite[Conjecture 3.7.2]{Em2} is true for $L_I \cap P_J$, i.e. $\ROrd_{L_I \cap P_J} \iso \HOrd_{L_I \cap P_J}$.
Then Conjecture \ref{conj:HOrd} for $\Iw^J=1$ is equivalent to \eqref{homdelta} being an isomorphism.
We could prove this if we knew that the isomorphism of Theorem \ref{theo:HOrd} with $\Iw^J=1$ were $B_J$-equivariant for all $w_J \in {}^{J \cap I}W_J$.
\end{rema}

Finally, we compute the derived ordinary parts of a parabolically induced representation in low degree when there is an inclusion between $I$ and $J$.

\begin{prop} \phantomsection \label{prop:HOrd2}
\begin{enumerate}
\item If $I \subseteq J$ and $0 < n < [F:\Qp]$, then $\HOrd[n]_{P_J}(\Ind_{P_I^-}^G \sigma)=0$.
\item If $J \subseteq I$ and $n < [F:\Qp]$, then there is a natural $L_J$-equivariant isomorphism
\begin{equation*}
\HOrd[n]_{L_I \cap P_J} \sigma \iso \HOrd[n]_{P_J} \left( \Ind_{P_I^-}^G \sigma \right).
\end{equation*}
\item If $J \subseteq I$ and $\Ord_{L_I \cap P_{I \cap s_\alpha(J)}} \sigma = 0$ for all $\alpha \in \Delta^1 \backslash (I \cup J^\perp)$, then there is a natural short exact sequence of $L_J$-representations
\begin{multline*}
0 \to \HOrd[{[F:\Qp]}]_{L_I \cap P_J} \sigma \to \HOrd[{[F:\Qp]}]_{P_J} \left( \Ind_{P_I^-}^G \sigma \right) \\
\to \bigoplus_{\alpha \in J^{\perp,1} \backslash I} \left( \Ord_{L_I \cap P_J} \sigma \right)^\alpha \otimes \left( \omega^{-1} \circ \alpha \right) \to 0.
\end{multline*}
\end{enumerate}
\end{prop}

\begin{proof}
We use Proposition \ref{prop:FilPJ} and Lemma \ref{lemm:lin} with $\ell : \IW^J \to \N$ to obtain a filtration $\Fil_{P_J}^{\ell,\bullet}(\HOrd[n]_{P_J}(\Ind_{P_I^-}^G \sigma))$ indexed by $\N$ such that for all $i \in \N$, there is a natural $L_J$-equivariant isomorphism
\begin{equation} \label{GrPJell}
\Gr_{P_J}^{\ell,i} \left( \HOrd[n]_{P_J} \left( \Ind_{P_I^-}^G \sigma \right) \right) \cong \bigoplus_{\ell(\Iw^J)=i} \HOrd[n]_{P_J} \left( \cind_{P_I^-}^{P_I^- \Iw^J P_J} \sigma \right).
\end{equation}

Assume $n<[F:\Qp]$.
If $\Iw^J \neq 1$ (i.e. $d_{\Iw^J}>0$), then $\HOrd[n]_{P_J} (\cind_{P_I^-}^{P_I^- \Iw^J P_J} \sigma) = 0$ by Proposition \ref{prop:HOrd1} (i) since $n - [F:\Qp] d_{\Iw^J} < 0$, thus we deduce from \eqref{GrPJell} that $\Gr_{P_J}^{\ell,\bullet}(\HOrd[n]_{P_J}(\Ind_{P_I^-}^G \sigma))$ is concentrated in degree $0$, so that assuming Conjecture \ref{conj:HOrd} for $\Iw^J=1$, we obtain a natural $L_J$-equivariant isomorphism
\begin{equation} \label{HOrdiso}
\Ind_{L_J \cap P_I^-}^{L_J} \left( \HOrd[n]_{L_I \cap P_J} \sigma \right) \iso \HOrd[n]_{P_J} \left( \Ind_{P_I^-}^G \sigma \right).
\end{equation}
Now, Conjecture \ref{conj:HOrd} is true for $\Iw^J=1$ in the following cases: $n>0$ and $I \subseteq J$ by Proposition \ref{prop:HOrd1} (i) since $\HOrd[n]_{L_I \cap P_J} = \HOrd[n]_{L_I} = 0$ (cf. \cite[Proposition 3.6.1]{Em2}), in which case the source of \eqref{HOrdiso} is zero, hence (i); $J \subseteq I$ by Proposition \ref{prop:HOrd1} (ii), in which case the source of \eqref{HOrdiso} is $\HOrd[n]_{L_I \cap P_J} \sigma$, hence (ii).

Likewise, if $\Iw^J \neq 1$ and $\Iw^J \neq s_\alpha$ for all $\alpha \in \Delta^1 \backslash (I \cup J)$ (i.e. $d_{\Iw^J}>1$), then $\HOrd[{[F:\Qp]}]_{P_J} (\cind_{P_I^-}^{P_I^- \Iw^J P_J} \sigma) = 0$ by Proposition \ref{prop:HOrd1} (i) since $[F:\Qp] - [F:\Qp] d_{\Iw^J} < 0$, thus we deduce from \eqref{GrPJell} that $\Gr_{P_J}^{\ell,\bullet}(\HOrd[{[F:\Qp]}]_{P_J}(\Ind_{P_I^-}^G \sigma))$ is concentrated in degrees $0$ and $1$, so that assuming Conjecture \ref{conj:HOrd} for $n=[F:\Qp]$ and $\Iw^J=1$ or $\Iw^J = s_\alpha$ for all $\alpha \in \Delta^1 \backslash (I \cup J)$, we obtain a short exact sequence of $L_J$-representations
\begin{multline} \label{HOrdSE}
0 \to \Ind_{L_J \cap P_I^-}^{L_J} \left( \HOrd[{[F:\Qp]}]_{L_I \cap P_J} \sigma \right) \to \HOrd[{[F:\Qp]}]_{P_J} \left( \Ind_{P_I^-}^G \sigma \right) \\
\to \bigoplus_{\alpha \in \Delta^1 \backslash (I \cup J)} \Ind_{L_J \cap P_{J \cap s_\alpha(I)}^-}^{L_J} \left( \left( \Ord_{L_I \cap P_{I \cap s_\alpha(J)}} \sigma \right)^\alpha \otimes \left( \omega^{-1} \circ \alpha \right) \right) \to 0.
\end{multline}
Assume $J \subseteq I$ and $\Ord_{L_I \cap P_{I \cap s_\alpha(J)}} \sigma = 0$ for all $\alpha \in \Delta^1 \backslash (I \cup J^\perp)$.
Then Conjecture \ref{conj:HOrd} is indeed true for $n=[F:\Qp]$ in the following cases: $\Iw^J=1$ by Proposition \ref{prop:HOrd1} (ii), and the first non-trivial term of \eqref{HOrdSE} is $\HOrd[{[F:\Qp]}]_{L_I \cap P_J} \sigma$; $\Iw^J=s_\alpha$ with $\alpha \in \Delta^1 \backslash (I \cup J^\perp)$ by Proposition \ref{prop:HOrd1} (i) and the hypothesis on $\sigma$, and the corresponding summand of the last non-trivial term of \eqref{HOrdSE} is zero; $\Iw^J=s_\alpha$ with $\alpha \in J^{\perp,1} \backslash I$ by Proposition \ref{prop:HOrd1} (ii) since $s_\alpha(J)=J \subseteq I$, and the corresponding summand of the last non-trivial term of \eqref{HOrdSE} is $(\Ord_{L_I \cap P_J} \sigma)^\alpha \otimes (\omega^{-1} \circ \alpha)$.
Hence (iii).
\end{proof}

We reformulate Proposition \ref{prop:HOrd2} in the case $I=J$, using the fact that in this case $\HOrd[n]_{L_I \cap P_J}=0$ if $n>0$ (cf. \cite[Proposition 3.6.1]{Em2}).
Note that if $\Pb = \Lb \Nb$ is a standard parabolic subgroup, then for all $\alpha \in \Delta \backslash \Delta_\Lb$ the standard parabolic subgroup of $\Lb$ corresponding to $\Delta_\Lb \cap s_\alpha (\Delta_\Lb)$ is $\Lb \cap s_\alpha \Pb s_\alpha^{-1}$ and it is proper if and only if $\alpha \not \in \Delta_\Lb^\perp$.

\begin{coro} \label{coro:HOrd}
Let $\Pb = \Lb \Nb$ be a standard parabolic subgroup and $\sigma$ be a locally admissible $L$-representation.
\begin{enumerate}
\item For all $n \in \N$ such that $0<n<[F:\Qp]$, we have $\HOrd[n]_P (\Ind_{P^-}^G \sigma) = 0$.
\item If $\Ord_{L \cap s_\alpha P s_\alpha^{-1}} \sigma = 0$ for all $\alpha \in \Delta^1 \backslash (\Delta_\Lb \cup \Delta_\Lb^\perp)$, then there is a natural $L$-equivariant isomorphism
\begin{equation*}
\HOrd[{[F:\Qp]}]_P \left( \Ind_{P^-}^G \sigma \right) \cong \bigoplus_{\alpha \in \Delta_\Lb^{\perp,1}} \sigma^\alpha \otimes \left( \omega^{-1} \circ \alpha \right).
\end{equation*}
\end{enumerate}
\end{coro}

\section{Derived Jacquet functors}

The aim of this section is to study the derived functors of the Jacquet functor.
In §~\ref{ssec:pro}, we review some results on pro-categories.
In §~\ref{ssec:second}, we relate the left derived functors of the Jacquet functor in a pro-category with the derived ordinary parts functors and we construct a new exact sequence to compute extensions by a parabolically induced representation.
In §~\ref{ssec:HJ}, we adapt the results of §~\ref{ssec:HOrd} in order to partially compute the derived Jacquet functors on a parabolically induced representation.

\subsection{Pro-categories} \label{ssec:pro}

Let $H$ be a $p$-adic Lie group.
Let $\Cf$ be the category whose objects are the $A[H]$-modules such that for some (equivalently any) compact open subgroup $H_0 \subseteq H$, the $A[H_0]$-action extends to a structure of $A \llbrack H_0 \rrbrack$-module of finite type, and whose morphisms are the $A[H]$-linear maps.
Since the completed group rings are Noetherian (cf. \cite[Theorem 2.1.2]{Em1}), the category $\Cf$ is $A$-abelian and Noetherian, i.e. it is essentially small\footnote{A category is essentially small if it is equivalent to a small category, i.e. if the isomorphism classes of its objects form a set.} and its objects are Noetherian.
Let $\Cf^\wedge$ be the category of contravariant functors $\Cf \to \Set$ and $\Indt\Cf$ be the full subcategory of $\Cf^\wedge$ whose objects are the functors isomorphic to a small inductive limit in $\Cf^\wedge$ of objects of $\Cf$ (using the Yoneda embedding $\Cf \to \Cf^\wedge$).
By \cite[Theorem 8.6.5]{KS}, the category $\Indt\Cf$ is a Grothendieck category\footnote{A Grothendieck category is an abelian category that admits a generator and small direct sums, and in which inductive limits are exact.} (in particular it has enough injectives, cf. \cite[Theorem 9.6.2]{KS}) and the natural $A$-linear functor $\Cf \to \Indt\Cf$ is fully faithful and exact.

Now, Pontryagin duality induces an equivalence of categories (cf. \cite[(2.2.12)]{Em1})
\begin{equation*}
\Mod^\adm_H(A) \cong \Cf^\op.
\end{equation*}
Thus, the category $\Mod_H^\adm(A)$ is Artinian, the pro-category
\begin{equation*}
\Prot\Mod_H^\adm(A) \coloneqq \left( \Indt\Cf \right)^\op
\end{equation*}
has enough projectives, and the natural $A$-linear functor
\begin{equation} \label{emb}
\Mod_H^\adm(A) \to \Prot\Mod_H^\adm(A)
\end{equation}
is fully faithful and exact.
We let $\Ext_H^\bullet$ and $\Ext_{\Prot H}^\bullet$ denote the bifunctors of Yoneda extensions in the categories $\Mod_H^\adm(A)$ and $\Prot\Mod_H^\adm(A)$ respectively.
By \cite[Theorem 3.5]{Oort}, \eqref{emb} induces $A$-linear isomorphisms
\begin{equation} \label{ExtH}
\Ext_H^\bullet \left( \pi', \pi \right) \iso \Ext_{\Prot H}^\bullet \left( \pi', \pi \right)
\end{equation}
for all objects $\pi,\pi'$ of $\Mod_H^\adm(A)$.

\subsection{A second exact sequence} \label{ssec:second}

Let $\Pb \subseteq \Gb$ be a parabolic subgroup and $\Lb \subseteq \Pb$ be a Levi factor.
We let $\Pb^- \subseteq \Gb$ denote the parabolic subgroup opposed to $\Pb$ with respect to $\Lb$.
There is a natural exact sequence of $A$-modules (cf. \cite[(3.7.6)]{Em2})
\begin{equation} \label{SEExt1}
0 \to \Ext_L^1 \left( \sigma, \Ord_P \pi \right) \to \Ext_G^1 \left( \Ind_{P^-}^G \sigma, \pi \right) \to \Hom_L \left( \sigma, \HOrd[1]_P \pi \right)
\end{equation}
for all objects $\sigma$ and $\pi$ of $\Mod_L^\adm(A)$ and $\Mod_G^\adm(A)$ respectively.
We construct a second exact sequence, in which parabolic induction is on the right.

\medskip

By \cite[Proposition 4.1.5 and Proposition 4.1.7]{Em1}, parabolic induction induces an $A$-linear exact functor
\begin{equation*}
\Ind_P^G : \Mod_L^\adm(A) \to \Mod_G^\adm(A).
\end{equation*}
By \cite[Corollary 3.6.7]{Em2}, taking $N$-coinvariants induces an $A$-linear right-exact functor (the so-called Jacquet functor)
\begin{equation*}
\left( - \right)_N : \Mod_G^\adm(A) \to \Mod_L^\adm(A).
\end{equation*}
By Frobenius reciprocity and the universal property of coinvariants, there is a natural $A$-linear isomorphism
\begin{equation} \label{Frob}
\Hom_G \left( \pi, \Ind_P^G \sigma \right) \cong \Hom_L \left( \pi_N, \sigma \right).
\end{equation}
for all objects $\pi$ and $\sigma$ of $\Mod_G^\adm(A)$ and $\Mod_L^\adm(A)$ respectively.

We deduce from \cite[Proposition 6.1.9]{KS} that these functors and the adjunction relation extend to the corresponding pro-categories.
By \cite[Corollary 8.6.8]{KS}, $\Ind_P^G$ is still exact so that $(-)_N$ still preserves projectives.
Thus, denoting by $\Lh(N,-)$ the left derived functors of $(-)_N$ in $\Prot\Mod_G^\adm(A)$, there is a Grothendieck spectral sequence of $A$-modules
\begin{equation*}
\Ext_{\Prot L}^i \left( \Lh[j]\! \left( N, \pi \right), \sigma \right) \Rightarrow \Ext_{\Prot G}^{i+j} \left( \pi, \Ind_P^G \sigma \right)
\end{equation*}
whose low degree terms form a natural exact sequence of $A$-modules
\begin{multline} \label{SEGroth}
0 \to \Ext_{\Prot L}^1 \left( \pi_N, \sigma \right) \to \Ext_{\Prot G}^1 \left( \pi, \Ind_P^G \sigma \right) \to \Hom_{\Prot L} \left( \Lh[1]\! \left( N, \pi \right), \sigma \right) \\
\to \Ext_{\Prot L}^2 \left( \pi_N, \sigma \right) \to \Ext_{\Prot G}^2 \left( \pi, \Ind_P^G \sigma \right)
\end{multline}
for all objects $\pi$ and $\sigma$ of $\Prot\Mod_G^\adm(A)$ and $\Prot\Mod_L^\adm(A)$ respectively.

\medskip

We let $d$ denote the integer $\dim_F \Nb$ and $\delta \in \X^*(\Lb)$ denote the algebraic character of the adjoint representation of $\Lb$ on $\det_F (\Lie \Nb)$.
We define $A$-linear functors by setting
\begin{equation*}
\Hh\! \left( N, - \right) \coloneqq \HOrd[{[F:\Qp]} d-\bullet]_P \otimes \left( \omega \circ \delta \right).
\end{equation*}
We deduce from \cite[Corollary 3.4.8 and Proposition 3.6.1]{Em2} that we obtain a homological $\delta$-functor
\begin{equation*}
\Hh\! \left( N, - \right) : \Mod_G^\adm(A) \to \Mod_L^\adm(A)
\end{equation*}
and proceeding as in the proof of \cite[Corollary 8.6.8]{KS}, we see that it extends to a homological $\delta$-functor between the corresponding pro-categories.

By \cite[Proposition 3.6.2]{Em2}, there is an isomorphism of functors (hence the notation)
\begin{equation*}
\Hh[0]\! \left( N, - \right) \cong (-)_N
\end{equation*}
which, by universality of derived functors, extends uniquely to a morphism of $\delta$-functors
\begin{equation} \label{delta}
\Hh\! \left( N, - \right) \to \Lh\! \left( N, - \right)
\end{equation}
which is bijective in degree $0$, and thus surjective in degree $1$ (by a dimension-shifting argument).
Using \eqref{ExtH}, we deduce from \eqref{SEGroth} a natural exact sequence of $A$-modules
\begin{equation} \label{SEExt2}
0 \to \Ext_L^1 \left( \pi_N, \sigma \right) \to \Ext_G^1 \left( \pi, \Ind_P^G \sigma \right) \to \Hom_L \left( \Hh[1]\! \left( N, \pi \right), \sigma \right)
\end{equation}
for all objects $\pi$ and $\sigma$ of $\Mod_G^\adm(A)$ and $\Mod_L^\adm(A)$ respectively.

\begin{rema} \phantomsection \label{rema:dual}
\begin{enumerate}
\item Nothing is known (to the author at least) regarding the nature of the morphism \eqref{delta} in degree greater than $1$.
\item Let $H$ be a $p$-adic Lie group.
Taking inductive limits induces an $A$-linear exact functor
\begin{equation*}
\varinjlim : \Indt\Mod_H^\adm(A) \to \Mod_H^\ladm(A)
\end{equation*}
which is essentially surjective, but not faithful nor full in general.
Thus the situation here (i.e. deriving in $\Prot\Mod_H^\adm(A)$) is not exactly dual to that of \cite[§~3.7]{Em2} (i.e. deriving in $\Mod_H^\ladm(A)$).
\end{enumerate}
\end{rema}

\subsection{Adaptation of the computations} \label{ssec:HJ}

Let $I,J \subseteq \Delta$, $\sigma$ be an $L_I$-representation and $n \in \N$.
We let $\Iw_0 = w_{I,0}w_0$ (resp. $\JIw_{J,0} = w_{J \cap \Iw^J{}^{-1}(I),0} w_{J,0}$) denote the image of $w_0$ (resp. $w_{J,0}$) in $\IW$ (resp. $\JIW_J$) and we define an auxiliary subset of $\Delta$ by setting $I' \coloneqq \Iw_0^{-1}(I)$.
We have $\Lb_I = \Iw_0 \Lb_{I'} \Iw_0^{-1}$ and $\Pb_I = \Iw_0 \Pb_{I'}^- \Iw_0^{-1}$, hence a natural $G$-equivariant isomorphism
\begin{equation} \label{Indop}
\Ind_{P_I}^G \sigma \cong \Ind_{P_{I'}^-}^G \sigma^{\Iw_0}
\end{equation}
defined by $f \mapsto (g \mapsto f(\Iw_0 g))$.

\begin{lemm} \label{lemm:IWJop}
The map $\IW^J \to \IpW^J$ defined by $\Iw^J \mapsto \Iw_0^{-1} \Iw^J \JIw_{J,0}$ is an order-reversing bijection.
\end{lemm}

\begin{proof}
First, note that $W_I = w_0 W_{I'} w_0$, so that left translation by $w_0$ induces a bijection
\begin{equation*}
W_I \backslash W / W_J \iso W_{I'} \backslash W / W_J.
\end{equation*}
In particular, $\card \IW^J = \card \IpW^J$.
Thus, it is enough to prove that the order-reversing composite
\begin{equation*}
\IW^J \hookrightarrow W \iso W
\end{equation*}
where the first arrow is defined by $\Iw^J \mapsto w_{I,0} \Iw^J \JIw_{J,0}$ (it is injective since $\IW^J$ is a system of representatives of the double cosets $W_I \backslash W / W_J$, and order-preserving since the projection $W \twoheadrightarrow \IW^J$ is order-preserving) and the second arrow is the left multiplication by $w_0$ (it is an order-reversing bijection, cf. \cite[Proposition 2.3.4 (i)]{BB}), takes values in $\IpW^J$.

Now, let $\Iw^J \in \IW^J$.
For all $\Iw \in \IW$ and $w_{I'} \in W_{I'}$, we have (using \cite[Proposition 2.3.2 (ii)]{BB})
\begin{align*}
\ell \left( w_{I'} \Iw_0^{-1} \Iw \right) &= \ell \left( w_0 \left( w_0 w_{I'} w_0 \right) w_{I,0} \Iw \right) \\
&= \ell \left( w_0 \right) - \left( \left( \ell \left( w_{I,0} \right) - \ell \left( w_0 w_{I'} w_0 \right) \right) + \ell \left( \Iw \right) \right) \\
&= \ell \left( w_{I'} \right) + \ell \left( w_0 w_{I,0} \Iw \right).
\end{align*}
Since $\Iw^J \JIw_{J,0} \in \IW$, we deduce that $\Iw_0^{-1} \Iw^J \JIw_{J,0} \in \IpW$.
Likewise, for all $w^J \in W^J$, we have $w_0 w^J w_{J,0} \in W^J$.
Since
\begin{align*}
\Iw_0^{-1} \Iw^J \JIw_{J,0} &= w_0 w_{I,0} \Iw^J w_{J \cap \Iw^J{}^{-1}(I),0} w_{J,0} \\
&= w_0 w_{I,0} w_{I \cap \Iw^J(J),0} \Iw^J w_{J,0} \\
&= w_0 w_{I,0}^{I \cap \Iw^J(J)} \Iw^J w_{J,0}
\end{align*}
and $w_{I,0}^{I \cap \Iw^J(J)} \Iw^J \in W^J$, we deduce that $\Iw_0^{-1} \Iw^J \JIw_{J,0} \in W^J$.
We conclude that $\Iw_0^{-1} \Iw^J \JIw_{J,0} \in \IpW^J$.
\end{proof}

We deduce from Lemma \ref{lemm:IWJop} that the left translate by $\Iw_0$ of the decomposition $G = \bigsqcup_{\Ipw^J \in \IpW^J} P_{I'}^- \Ipw^J P_J$ is the decomposition $G = \bigsqcup_{\Iw^J \in \IW^J} P_I \Iw^J P_J$ with the opposite closure relations.
Proceeding as in §~\ref{ssec:fil}, we can construct a natural filtration $\Fil_{P_J}^\bullet(\Ind_{P_I}^G \sigma)$ by $P_J$-subrepresentations indexed by $\IW^J$ with the opposite Bruhat order, and there is a natural $P_J$-equivariant isomorphism
\begin{equation*}
\Gr_{P_J}^{\Iw^J} \left( \Ind_{P_I}^G \sigma \right) \cong \cind_{P_I}^{P_I \Iw^J P_J} \sigma
\end{equation*}
for all $\Iw^J \in \IW^J$.
Furthermore, \eqref{Indop} identifies this filtration with $\Fil_{P_J}^\bullet(\Ind_{P_{I'}^-}^G \sigma^{\Iw_0})$, using Lemma \ref{lemm:IWJop} to identifiy the indexing posets, and induces a natural $P_J$-equivariant isomorphism
\begin{equation} \label{cindop}
\cind_{P_I}^{P_I \Iw^J P_J} \sigma \cong \cind_{P_{I'}^-}^{P_{I'}^- \Ipw^J P_J} \sigma^{\Iw_0}
\end{equation}
for all $\Iw^J \in \IW^J$ with $\Ipw^J = \Iw_0^{-1} \Iw^J \JIw_{J,0}$.
We deduce from Proposition \ref{prop:FilPJ} that $\Fil_{P_J}^\bullet(\Ind_{P_I}^G \sigma)$ induces a filtration $\Fil_{P_J}^\bullet(\Hh[n](N_J,\Ind_{P_I}^G \sigma))$ by $L_J$-subrepresentations indexed by $\IW^J$ with the opposite Bruhat order and that there is a natural $L_J$-equivariant isomorphism
\begin{equation*}
\Gr_{P_J}^{\Iw^J} \left( \Hh[n]\! \left( N_J, \Ind_{P_I}^G \sigma \right) \right) \cong \Hh[n]\! \left( N_J, \cind_{P_I}^{P_I \Iw^J P_J} \sigma \right)
\end{equation*}
for all $\Iw^J \in \IW^J$.

Let $\Iw^J \in \IW^J$ and set $\Ipw^J \coloneqq \Iw_0^{-1} \Iw^J \JIw_{J,0}$.
We let $\sigmat$ be an $L_{J \cap \Iw^J{}^{-1}(I)}$-representation.
Note that $J \cap \Iw^J{}^{-1}(I) = \JIw_{J,0}(J \cap \Ipw^J{}^{-1}(I'))$.
We have
\begin{gather*}
\Lb_{J \cap \Iw^J{}^{-1}(I)} = \JIw_{J,0} \Lb_{J \cap \Ipw^J{}^{-1}(I')} \JIw_{J,0}^{-1} \\
\text{and } \Lb_J \cap \Pb_{J \cap \Iw^J{}^{-1}(I)} = \JIw_{J,0} \Lb_J \cap \Pb_{J \cap \Ipw^J{}^{-1}(I')}^- \JIw_{J,0}^{-1},
\end{gather*}
hence a natural $L_J$-equivariant isomorphism
\begin{equation} \label{IndLJop}
\Ind_{L_J \cap P_{J \cap \Iw^J{}^{-1}(I)}}^{L_J} \sigmat \cong \Ind_{L_J \cap P_{J \cap \Ipw^J{}^{-1}(I')}^-}^{L_J} \sigmat^{\JIw_{J,0}}
\end{equation}
defined by $f \mapsto (l \mapsto f(\JIw_{J,0} l))$.
Proceeding as in the proof of Lemma \ref{lemm:IWJop}, we obtain the following result.

\begin{lemm} \label{lemm:JIWJop}
The map $\JIW_J \to \JIpW_J$ defined by $w_J \mapsto \JIw_{J,0}^{-1} w_J$ is an order-reversing bijection.
\end{lemm}

We deduce from Lemma \ref{lemm:JIWJop} that the left translate by $\Iw_0$ (resp. $\JIw_{J,0}$) of the decomposition $P_{I'}^- \Ipw^J P_J = \bigsqcup_{w'_J \in \JIpW_J} P_{I'}^- \Ipw^J w'_J B$
\begin{equation*}
\text{(resp. } L_J = \bigsqcup_{w'_J \in \JIpW_J} L_J \cap P_{J \cap \Ipw^J{}^{-1}(I')}^- w'_J B_J \text{)}
\end{equation*}
is the decomposition $P_I \Iw^J P_J = \bigsqcup_{w_J \in \JIW_J} P_I \Iw^J w_J B$
\begin{equation*}
\text{(resp. } L_J = \bigsqcup_{w_J \in \JIW_J} L_J \cap P_{J \cap \Iw^J{}^{-1}(I)} w_J B_J \text{)}
\end{equation*}
with the opposite closure relations.
Proceeding as in §~\ref{ssec:fil}, we can construct a natural filtration $\Fil_B^\bullet(\cind_{P_I}^{P_I \Iw^J P_J} \sigma)$ (resp. $\Fil_{B_J}^\bullet(\Ind_{L_J \cap P_{J \cap \Iw^J{}^{-1}(I)}}^{L_J} \sigmat)$) by $B$-subrepresentations (resp. $B_J$-subrepresentations) indexed by $\JIW_J$ with the opposite Bruhat order, and there is a natural $B$-equivariant (resp. $B_J$-equivariant) isomorphism
\begin{gather*}
\Gr_B^{w_J} \left( \cind_{P_I}^{P_I \Iw^J P_J} \sigma \right) \cong \cind_{P_I}^{P_I \Iw^J w_J B} \sigma \\
\text{(resp. } \Gr_{B_J}^{w_J} \left( \Ind_{L_J \cap P_{J \cap \Iw^J{}^{-1}(I)}}^{L_J} \sigmat \right) \cong \cind_{L_J \cap P_{J \cap \Iw^J{}^{-1}(I)}}^{L_J \cap P_{J \cap \Iw^J{}^{-1}(I)} w_J B_J} \sigmat \text{)}
\end{gather*}
for all $w_J \in \JIW_J$.
Furthermore, \eqref{cindop} (resp. \eqref{IndLJop}) identifies this filtration with $\Fil_B^\bullet(\cind_{P_{I'}^-}^{P_{I'}^- \Ipw^J P_J} \sigma^{\Iw_0})$ (resp. $\Fil_{B_J}^\bullet(\Ind_{L_J \cap P_{J \cap \Ipw^J{}^{-1}(I')}^-}^{L_J} \sigmat^{\JIw_{J,0}})$), using Lemma \ref{lemm:JIWJop} to identifiy the indexing posets, and induces a natural $B$-equivariant (resp. $B_J$-equivariant) isomorphism
\begin{gather*}
\cind_{P_I}^{P_I \Iw^J w_J B} \sigma \cong \cind_{P_{I'}^-}^{P_{I'}^- \Ipw^J w'_J B} \sigma^{\Iw_0} \\
\text{(resp. } \cind_{L_J \cap P_{J \cap \Iw^J{}^{-1}(I)}}^{L_J \cap P_{J \cap \Iw^J{}^{-1}(I)} w_J B_J} \sigmat \cong \cind_{L_J \cap P_{J \cap \Ipw^J{}^{-1}(I')}^-}^{L_J \cap P_{J \cap \Ipw^J{}^{-1}(I')}^- w'_J B_J} \sigmat^{\JIw_{J,0}} \text{)}
\end{gather*}
for all $w_J \in \JIW_J$ with $w'_J = \JIw_{J,0}^{-1} w_J$.
We deduce from Proposition \ref{prop:FilB} that $\Fil_B^\bullet(\cind_{P_I}^{P_I \Iw^J P_J} \sigma)$ induces a filtration $\Fil_B^\bullet(\Hh[n](N_J,\cind_{P_I}^{P_I \Iw^J P_J} \sigma))$ by $B_J$-subrepresentations indexed by $\JIW_J$ with the opposite Bruhat order and that there is a natural $B_J$-equivariant isomorphism
\begin{equation*}
\Gr_B^{w_J} \left( \Hh[n]\! \left( N_J, \cind_{P_I}^{P_I \Iw^J P_J} \sigma \right) \right) \cong \Hh[n]\! \left( N_J, \cind_{P_I}^{P_I \Iw^J w_J B} \sigma \right)
\end{equation*}
for all $w_J \in \JIW_J$.

\begin{theo} \label{theo:HJ}
Let $\sigma$ be a locally admissible $L_I$-representation, $\Iw^J \in \IW^J$ and $n \in \N$.
For all $w_J \in \JIW_J$, there is a natural $B_{J,\JIw_{J,0}^{-1}w_J}$-equivariant isomorphism
\begin{multline*}
\Gr_B^{w_J} \left( \Hh[n]\! \left( N_J, \cind_{P_I}^{P_I \Iw^J P_J} \sigma \right) \right) \\
\cong \Gr_{B_J}^{w_J} \left( \Ind_{L_J \cap P_{J \cap \Iw^J{}^{-1}(I)}}^{L_J} \left( \Hh[n-{[F:\Qp]} d_{\Iw^J}]\! \left( L_I \cap N_{I \cap \Iw^J(J)}, \sigma \right)^{\Iw^J} \otimes \left( \omega \circ \delta_{\Iw^J} \right) \right) \right)
\end{multline*}
which is even $L_J \cap P_{J \cap \JIw_{J,0}^{-1}\Iw^J{}^{-1}(I)}$-equivariant when $w_J=\JIw_{J,0}$.
\end{theo}

\begin{proof}
We set $\Ipw^J \coloneqq \Iw_0^{-1} \Iw^J \JIw_{J,0}$ and we define an $L_{J \cap \Iw^J{}^{-1}(I)}$-representation by setting
\begin{equation*}
\sigmat \coloneqq \left( \left( \HOrd[{[F:\Qp]} (d_J-d_{\Ipw^J}) - n]_{L_{I'} \cap P_{I' \cap \Ipw^J(J)}} \sigma^{\Iw_0} \right)^{\Ipw^J} \otimes \left( \omega \circ \left( \delta_J - \delta_{\Ipw^J} \right) \right) \right)^{\JIw_{J,0}^{-1}}
\end{equation*}
where $d_J$ denote the integer $\dim_F \Nb_J$ and $\delta_J \in \X^*(\Lb_J)$ denote the algebraic character of the adjoint representation of $\Lb_J$ on $\det_F (\Lie \Nb_J)$.
We prove that there is a natural $L_{J \cap \Iw^J{}^{-1}(I)}$-equivariant isomorphism
\begin{equation} \label{HOrdop}
\sigmat \cong \Hh[n - {[F:\Qp]} d_{\Iw^J}]\! \left( L_I \cap N_{I \cap \Iw^J(J)}, \sigma \right)^{\Iw^J} \otimes \left( \omega \circ \delta_{\Iw^J} \right).
\end{equation}
We have $\Lb_{I \cap \Iw^J(J)} = \Iw_0 \Lb_{I' \cap \Ipw^J(J)} \Iw_0^{-1}$ and $\Lb_I \cap \Pb_{I \cap \Iw^J(J)} = \Iw_0 \Lb_{I'} \cap \Pb_{I' \cap \Ipw^J(J)} \Iw_0^{-1}$, hence natural $L_{I' \cap \Ipw^J(J)}$-equivariant isomorphisms
\begin{equation*}
\HOrd_{L_{I'} \cap P_{I' \cap \Ipw^J(J)}} \sigma^{\Iw_0} \cong \left( \HOrd_{L_I \cap P_{I \cap \Iw^J(J)}} \sigma \right)^{\Iw_0}.
\end{equation*}
Using Lemma \ref{lemm:IwJ} (iii), we have (with notations analogous to $d_J$ and $\delta_J$)
\begin{equation*}
\HOrd_{L_I \cap P_{I \cap \Iw^J(J)}} = \Hh[{[F:\Qp]} (d_{I \cap \Iw^J(J)} - d_I) - \bullet]\! \left( L_I \cap N_{I \cap \Iw^J(J)}, - \right) \otimes \left( \omega^{-1} \circ \left( \delta_{I \cap \Iw^J(J)} - \delta_I \right) \right).
\end{equation*}
Thus in order to prove \eqref{HOrdop}, it remains to check that
\begin{gather*}
d_J = \left( d_{I \cap \Iw^J(J)} - d_I \right) + d_{\Iw^J} + d_{\Ipw^J}, \\
\delta_J = \Iw^J{}^{-1} \left( \delta_{I \cap \Iw^J(J)} - \delta_I \right) + \delta_{\Iw^J} + \JIw_{J,0} \left( \delta_{\Ipw^J} \right).
\end{gather*}
We do these computations on the corresponding Lie algebras: $d_J$ and $\delta_J$ correspond to $\Phi^+ \backslash \Phi_J^+$, $(d_{I \cap \Iw^J(J)} - d_I)$ and $\Iw^J{}^{-1}(\delta_{I \cap \Iw^J(J)} - \delta_I)$ correspond to $(\Phi^+ \backslash \Phi_J^+) \cap \Iw^J{}^{-1}(\Phi_I^+)$, $d_{\Iw^J}$ and $\delta_{\Iw^J}$ correspond to $(\Phi^+ \backslash \Phi_J^+) \cap \Iw^J{}^{-1}(-\Phi^+)$, and $d_{\Ipw^J}$ and $\JIw_{J,0}(\delta_{\Ipw^J})$ correspond to $(\Phi^+ \backslash \Phi_J^+) \cap \Iw^J{}^{-1}(\Phi^+ \backslash \Phi_I^+)$ (noting that $\Iw_0(-\Phi^+) = (-\Phi_I^+) \sqcup (\Phi^+ \backslash \Phi_I^+)$ and $(\Phi^+ \backslash \Phi_J^+) \cap \Iw^J{}^{-1}(-\Phi_I^+) = \emptyset$).
Thus, the two equalities above follow from the partition
\begin{multline*}
\Phi^+ \backslash \Phi_J^+ = \left( \left( \Phi^+ \backslash \Phi_J^+ \right) \cap \Iw^J{}^{-1} \left( \Phi_I^+ \right) \right) \sqcup \left( \left( \Phi^+ \backslash \Phi_J^+ \right) \cap \Iw^J{}^{-1} \left( -\Phi^+ \right) \right) \\
\sqcup \left( \left( \Phi^+ \backslash \Phi_J^+ \right) \cap \Iw^J{}^{-1} \left( \Phi^+ \backslash \Phi_I^+ \right) \right),
\end{multline*}
which is obtained from the partition $\Phi = \Phi_I^+ \sqcup (-\Phi^+) \sqcup (\Phi^+ \backslash \Phi_I^+)$ by applying $\Iw^J{}^{-1}$ and taking the intersection with $\Phi^+ \backslash \Phi_J^+$.

Let $w_J \in \JIW_J$ and set $w'_J \coloneqq \JIw_{J,0}^{-1} w_J$.
By construction, \eqref{cindop} induces a natural $B_J$-equivariant isomorphism
\begin{multline*}
\Gr_B^{w_J} \left( \Hh[n]\! \left( N_J, \cind_{P_I}^{P_I \Iw^J P_J} \sigma \right) \right) \\
\cong \Gr_B^{w'_J} \left( \HOrd[{[F:\Qp]} d_J - n]_{P_J} \left( \cind_{P_{I'}^-}^{P_{I'}^- \Ipw^J P_J} \sigma^{\Iw_0} \right) \otimes \left( \omega \circ \delta_J \right) \right).
\end{multline*}
By Theorem \ref{theo:HOrd}, there is a natural $B_{J,w'_J}$-equivariant isomorphism
\begin{multline*}
\Gr_B^{w'_J} \left( \HOrd[{[F:\Qp]} d_J - n]_{P_J} \left( \cind_{P_{I'}^-}^{P_{I'}^- \Ipw^J P_J} \sigma^{\Iw_0} \right) \otimes \left( \omega \circ \delta_J \right) \right) \\
\cong \Gr_{B_J}^{w'_J} \left( \Ind_{L_J \cap P_{J \cap \Ipw^J{}^{-1}(I')}^-}^{L_J} \sigmat^{\JIw_{J,0}} \right)
\end{multline*}
which is even $L_J \cap P_{J \cap \Ipw^J{}^{-1}(I')}$-equivariant when $w'_J=1$.
By construction, \eqref{IndLJop} and \eqref{HOrdop} induce a natural $B_J$-equivariant isomorphism
\begin{multline*}
\Gr_{B_J}^{w'_J} \left( \Ind_{L_J \cap P_{J \cap \Ipw^J{}^{-1}(I')}^-}^{L_J} \sigmat^{\JIw_{J,0}} \right) \\
\cong \Gr_{B_J}^{w_J} \left( \Ind_{L_J \cap P_{J \cap \Iw^J{}^{-1}(I)}}^{L_J} \left( \Hh[n-{[F:\Qp]} d_{\Iw^J}]\! \left( L_I \cap N_{I \cap \Iw^J(J)}, \sigma \right)^{\Iw^J} \otimes \left( \omega \circ \delta_{\Iw^J} \right) \right) \right).
\end{multline*}
Composing these three isomorphisms yields the result.
\end{proof}

We deduce from Theorem \ref{theo:HJ} with $w_J=\JIw_{J,0}$ a natural $L_J$-equivariant morphism analogous to \eqref{HOrd}
\begin{multline} \label{HJ}
\Ind_{L_J \cap P_{J \cap \Iw^J{}^{-1}(I)}}^{L_J} \left( \Hh[n-{[F:\Qp]} d_{\Iw^J}]\! \left( L_I \cap N_{I \cap \Iw^J(J)}, \sigma \right)^{\Iw^J} \otimes \left( \omega \circ \delta_{\Iw^J} \right) \right) \\
\to \Hh[n]\! \left( N_J, \cind_{P_I}^{P_I \Iw^J P_J} \sigma \right)
\end{multline}
and Conjecture \ref{conj:HOrd} is equivalent to \eqref{HJ} being an isomorphism.
We also have analogues of Propositions \ref{prop:HOrd1} and \ref{prop:HOrd2}.
In the case $I=J$, we obtain the following analogue of Corollary \ref{coro:HOrd}.

\begin{coro} \label{coro:HJ}
Let $\Pb = \Lb \Nb$ be a standard parabolic subgroup and $\sigma$ be a locally admissible $L$-representation.
\begin{enumerate}
\item For all $n \in \N$ such that $0<n<[F:\Qp]$, we have $\Hh[n](N,\Ind_P^G \pi) = 0$.
\item If $\sigma_{L \cap s_\alpha N s_\alpha^{-1}} = 0$ for all $\alpha \in \Delta^1 \backslash (\Delta_\Lb^1 \cup \Delta_\Lb^{\perp,1})$, then there is a natural $L$-equivariant isomorphism
\begin{equation*}
\Hh[{[F:\Qp]}]\! \left( N, \Ind_P^G \sigma \right) \cong \bigoplus_{\alpha \in \Delta_\Lb^{\perp,1}} \sigma^\alpha \otimes \left( \omega \circ \alpha \right).
\end{equation*}
\end{enumerate}
\end{coro}

\begin{rema} \label{rema:HJ}
The results hold true with $P^-$, $N^-$ and $\omega^{-1}$ instead of $P$, $N$ and $\omega$ respectively.
\end{rema}

\section{Application to extensions}

The aim of this section is to compute the extensions between parabolically induced representations of $G$.
In §~\ref{ssec:Extprelim}, we review some cuspidality and genericity properties and we prove some preliminary results on extensions which will be used in the case where $\Gb$ is split and $\Zb$ is connected.
Then, the main results are proved in §~\ref{ssec:Ext}.
Finally, some of these results are lifted to characteristic $0$ in §~\ref{ssec:Extcont}.

\subsection{Preliminaries} \label{ssec:Extprelim}

We fix a standard parabolic subgroup $\Pb = \Lb \Nb$.

\subsubsection*{Cuspidality and genericity properties}

We define some cuspidality properties and discuss the relations between them.

\begin{defi}
We say that an admissible smooth representation $\sigma$ of $L$ over $k$ is:
\begin{itemize}
\item \emph{supersingular} if $\Fpbar \otimes_k \sigma$ is supersingular (in the sense of \cite{AHHV}),
\item \emph{supercuspidal} if it is irreducible and not a subquotient of $\Ind_Q^L \tau$ for any proper parabolic subgroup $\Qb \subset \Lb$ with Levi quotient $\Lb_\Qb$ and any irreducible admissible smooth representation $\tau$ of $L_Q$ over $k$,
\item \emph{right} (resp. \emph{left}) \emph{cuspidal} if $\Ord_Q \sigma = 0$ (resp. $\sigma_{N_Q}=0$) for any proper parabolic subgroup $\Qb \subset \Lb$ with unipotent radical $\Nb_\Qb$.
\end{itemize}
\end{defi}

\begin{rema}
In \cite[Definition 6.3]{AHV}, left and right cuspidality are defined for smooth representations using the left and right adjoint functors of $\Ind_Q^G$, namely $\Lh[Q]^L$ and $\Rc[L]_Q$.
Since $\Lh[Q]^L = (-)_{N_Q}$ and the restriction of $\Rc[L]_Q$ to admissible representations is $\Ord_{Q^-}$ (cf. \cite[Corollary 4.13]{AHV}), these definitions coincide for admissible representations.
\end{rema}

\begin{lemm}
Let $\sigma$ be an irreducible admissible smooth representation of $L$ over $k$.
The following are equivalent.
\begin{enumerate}
\item $\sigma$ is supercuspidal.
\item $\sigma$ is left and right cuspidal.
\item $\Fpbar \otimes_k \sigma$ is a (finite) direct sum of supersingular representations.
\end{enumerate}
In particular, $\sigma$ is supersingular if and only if it is absolutely irreducible and supercuspidal.
\end{lemm}

\begin{proof}
Over $\Fpbar$, the equivalence between (i) and (ii) is \cite[Corollary 6.9]{AHV}, and the equivalence between `supercuspidal' and `supersingular' is \cite[Theorem 5]{AHHV}.
By \cite[Lemma 4.1.2]{Em2}, $\Fpbar \otimes_k \sigma$ is a finite direct sum of irreducible admissible smooth representations of $L$ over $\Fpbar$.
Since $\Ind_Q^L$, $(-)_{N_Q}$ and $\Ord_Q$ commute with $\Fpbar \otimes_k -$, we deduce the equivalences over $k$.
\end{proof}

We now study some genericity property for smooth representations of $L$ over $k$ with central character.
We assume that $\Delta_\Lb^{\perp,1} \neq \emptyset$.

\begin{lemm} \label{lemm:alpha}
Let $\sigma$ be a smooth representation of $L$ over $k$ with central character $\zeta : Z_L \to k^\times$ and $\alpha \in \Delta_\Lb^{\perp,1}$.
If $\zeta \circ \alpha^\vee = \omega^{-1}$, then $\sigma^\alpha \otimes (\omega^{-1} \circ \alpha) \cong \sigma$.
\end{lemm}

\begin{proof}
For convenience, we recall the construction of the representation $\sigma^\alpha \otimes (\omega^{-1} \circ \alpha)$.
Let $\Gb_\alpha \subseteq \Gb$ be the standard Levi subgroup corresponding to $\alpha$.
We fix a representative $n_\alpha \in \Nc$ of $s_\alpha$.
For every $\beta \in \Delta_\Lb$ and for all integers $i,j>0$, $i\alpha+j\beta \not \in \Phi$ (since $\alpha \perp \beta$), thus $\Ub_\alpha$ and $\Ub_\beta$ commute for every $\beta \in \Delta_\Lb$ by \cite[Proposition 2.5]{BT}, or more directly using the Baker--Campbell--Hausdorff formula.
We deduce that $\Gb_\alpha$ and $\Lb$ normalise each other (since $\Gb_\alpha$ and $\Lb$ are generated by $\Zcb$ and respectively $\Ub_{\pm\alpha}$ and $(\Ub_\beta)_{\beta \in \pm \Delta_\Lb}$).
In particular, $n_\alpha$ normalises $L$ (since $n_\alpha \in G_\alpha$) and the $n_\alpha$-conjugate $\sigma^\alpha$ does not depend on the choice of $n_\alpha$ in $n_\alpha \Zc$ up to isomorphism (since $\Zc \subseteq L$).
Furthermore, $\Lb$ normalises $\Ub_\alpha$ and $\alpha$ extends (uniquely) to an algebraic character of $\Lb$ (since $\alpha \in \Delta^1$).

We let $I_\alpha \subseteq L$ be the kernel of $\alpha : L \to F^\times$.
Note that $L=SI_\alpha$.
We may and will assume that $n_\alpha$ lies in the subgroup of $G_\alpha$ generated by $U_{\pm\alpha}$ (cf. \cite[§~II.4]{AHHV}) so that $n_\alpha$ commutes with $I_\alpha$.
Thus, the action of $I_\alpha$ on $\sigma^\alpha \otimes (\omega^{-1} \circ \alpha)$ and $\sigma$ is the same.

Now, assume $\zeta \circ \alpha^\vee = \omega^{-1}$.
For any $\lambda \in \X_*(\Sb)$, $\lambda - s_\alpha(\lambda) = \langle \alpha, \lambda \rangle \alpha^\vee$ so that $\lambda - s_\alpha(\lambda) \in \X_*(\Sb \cap \Zb_\Lb)$ and
\begin{equation*}
\zeta \circ \left( \lambda - s_\alpha(\lambda) \right) = \left( \zeta \circ \alpha^\vee \right)^{\langle \alpha, \lambda \rangle} = \omega^{-\langle \alpha, \lambda \rangle} = \left( \omega^{-1} \circ \alpha \right) \circ \lambda.
\end{equation*}
We deduce that for any $s \in S$, $s(n_\alpha s n_\alpha^{-1})^{-1} \in S \cap Z_L$ and $\zeta(s(n_\alpha s n_\alpha^{-1})^{-1}) = (\omega^{-1} \circ \alpha)(s)$.
Thus, the action of $S$ on $\sigma^\alpha \otimes (\omega^{-1} \circ \alpha)$ and $\sigma$ is the same.
\end{proof}

The following result yields a converse to Lemma \ref{lemm:alpha} when $\Gb$ is split and $\Zb$ is connected (cf. \cite[Proposition 2.1.1]{BH}).

\begin{lemm} \label{lemm:gen}
Let $\sigma$ be a smooth representation of $L$ over $k$ with central character $\zeta : Z_L \to k^\times$ and $\alpha \in \Delta_\Lb^{\perp,1}$.
Assume that there exists $\lambda \in \X_*(\Zb_\Lb)$ such that $\langle\alpha,\lambda\rangle=1$ and $\langle\beta,\lambda\rangle=0$ for all $\beta \in \Delta_\Lb^{\perp,1} \backslash \{\alpha\}$.
If $\zeta \circ \alpha^\vee \neq \omega^{-1}$, then $s_\alpha(\zeta)(\omega^{-1} \circ \alpha) \neq \zeta$ and $s_\alpha(\zeta)(\omega^{-1} \circ \alpha) \neq s_\beta(\zeta)(\omega^{-1} \circ \beta)$ for all $\beta \in \Delta_\Lb^{\perp,1} \backslash \{\alpha\}$.
\end{lemm}

\begin{proof}
We have $(s_\alpha(\zeta)(\omega^{-1} \circ \alpha)) \circ \lambda = (\zeta \circ \lambda)((\zeta \circ \alpha^\vee) \omega)^{-1}$ and $(s_\beta(\zeta)(\omega^{-1} \circ \beta)) \circ \lambda = (\zeta \circ \lambda)$ for all $\beta \in \Delta_\Lb^{\perp,1} \backslash \{\alpha\}$.
Thus, if $s_\alpha(\zeta)(\omega^{-1} \circ \alpha) = \zeta$ or $s_\alpha(\zeta)(\omega^{-1} \circ \alpha) = s_\beta(\zeta)(\omega^{-1} \circ \beta)$ for some $\beta \in \Delta_\Lb^{\perp,1} \backslash \{\alpha\}$, then precomposing each side of the equality with $\lambda$ yields the equality $(\zeta \circ \alpha^\vee) \omega = 1$.
\end{proof}

\subsubsection*{Preliminary results on extensions}

Let $H$ be a $p$-adic Lie group.
For locally admissible smooth representations $\pi,\pi'$ of $L$ over $A$, we let $\Ext_H^\bullet(\pi',\pi)$ denote the $A$-modules of extensions computed in $\Mod_H^\ladm(A)$ à la Yoneda or using an injective resolution of $\pi$.
If $\pi,\pi'$ are admissible, then in degree $1$ it is equivalent to compute $\Ext_H^\bullet(\pi',\pi)$ in $\Mod_H^\adm(A)$ à la Yoneda, but this is not known in higher degree (cf. \cite[Remark 3.7.8]{Em2}), except when $H=\GL_2(\Qp)$ (cf. \cite[Corollary 5.17]{Pas}).
Let $\Zt \subseteq H$ be a central closed subgroup and $\zeta : \Zt \to A^\times$ be a smooth character.
We write $\Mod_{H,\zeta}^\ladm(A)$ for the full subcategory of $\Mod_H^\ladm(A)$ whose objects are the representations on which $\Zt$ acts via $\zeta$.
If $\Zt$ acts on $\pi,\pi'$ via $\zeta$, then we let $\Ext_{H,\zeta}^\bullet(\pi',\pi)$ denote the $A$-modules of extensions computed in $\Mod_{H,\zeta}^\ladm(A)$ à la Yoneda, or equivalently using an injective resolution of $\pi$.

We now assume that $\Gb$ is split and we write $\Tb$ for the maximal split torus $\Sb=\Zcb$.
Using Notation \ref{nota:dw}, we have $d_w=\ell(w)$ for all $w \in W$, so that in particular $\Delta^1=\Delta$.
Let $\Lb' \subseteq \Gb$ be a standard Levi subgroup such that $\Delta_\Lb \perp \Delta_{\Lb'}$.
Note that $\Lb\Lb'$ is the standard Levi subgroup corresponding to $\Delta_\Lb \sqcup \Delta_{\Lb'}$.
Let $\sigma$ be a locally admissible smooth representation of $L$ over $k$ with central character $\zeta : Z_L \to k^\times$.
The following construction was communicated to me by N.~Abe.

First, we assume that $\Gb^\der$ is simply connected and we let $\Zbt \subseteq \Zb$ be a closed subgroup.
Recall that this is equivalent to the existence of fundamental weights $(\mu_\alpha)_{\alpha \in \Delta}$ (cf. \cite[Proposition 2.1.1]{BH}).
We set $\chi \coloneqq \zeta \circ \sum_{\alpha \in \Delta_{\Lb'}} (\alpha^\vee \circ \mu_\alpha)$.
Thus $\chi \circ \alpha^\vee = 1$ for all $\alpha \in \Delta_\Lb$ and $\chi \circ \alpha^\vee = \zeta \circ \alpha^\vee$ for all $\alpha \in \Delta_{\Lb'}$, so that $\chi$ extends uniquely to $L$ and $\sigma_0 \coloneqq \sigma \otimes \chi^{-1}$ extends uniquely to a locally admissible smooth representation of $LL'$ over $k$ by \cite[Lemma 3.2]{Abe}.
We let $\chi' : T \to k^\times$ be a smooth character such that $\chi'_{|Z_{L'}} = \chi_{|Z_{L'}}$, so that $\chi'$ extends uniquely to $L$, and we set $\sigma' \coloneqq \sigma_0 \otimes \chi'$.
There is a commutative diagram of $k$-vector spaces
\begin{equation} \label{diagExt} \begin{tikzcd}
\Ext_{T,\chi_{|Z_{L'}}}^\bullet \left( \chi', \chi \right) \rar \dar & \Ext_{L',\chi_{|Z_{L'}}}^\bullet \left( \Ind_{B_{L'}^-}^{L'} \chi', \Ind_{B_{L'}^-}^{L'} \chi \right) \dar \\
\Ext_{L,\chi_{|\Zt}}^\bullet \left( \chi', \chi \right) \rar \dar & \Ext_{LL',\chi_{|\Zt}}^\bullet \left( \Ind_{LB_{L'}^-}^{LL'} \chi', \Ind_{LB_{L'}^-}^{LL'} \chi \right) \dar \\
\Ext_{L,\zeta_{|\Zt}}^\bullet \left( \sigma', \sigma \right) \rar & \Ext_{LL',\zeta_{|\Zt}}^\bullet \left( \Ind_{LB_{L'}^-}^{LL'} \sigma', \Ind_{LB_{L'}^-}^{LL'} \sigma \right)
\end{tikzcd} \end{equation}
where the horizontal arrows are induced by the functors $\Ind_{B_{L'}^-}^{L'}$ and $\Ind_{LB_{L'}^-}^{LL'}$, the upper vertical arrows are induced by extending representations to $L$ and $LL'$, and the lower vertical arrows are induced by tensoring representations with $\sigma_0$.
Furthermore, the lower horizontal arrow of \eqref{diagExt} composed with the $k$-linear morphism induced by the functor $\Ind_{P^-L'}^G$:
\begin{equation} \label{IndPL'}
\Ext_{LL',\zeta_{|\Zt}}^\bullet \left( \Ind_{LB_{L'}^-}^{LL'} \sigma', \Ind_{LB_{L'}^-}^{LL'} \sigma \right) \to \Ext_{G,\zeta_{|\Zt}}^\bullet \left( \Ind_{P^-}^G \sigma', \Ind_{P^-}^G \sigma \right),
\end{equation}
is the $k$-linear morphism induced by the functor $\Ind_{P^-}^G$:
\begin{equation} \label{IndP}
\Ext_{L,\zeta_{|\Zt}}^\bullet \left( \sigma', \sigma \right) \to \Ext_{G,\zeta_{|\Zt}}^\bullet \left( \Ind_{P^-}^G \sigma', \Ind_{P^-}^G \sigma \right).
\end{equation}

\begin{lemm} \phantomsection \label{lemm:Extprelim}
\begin{enumerate}
\item In degree $1$, there is a $k$-linear injection from the cokernel of the upper horizontal arrow of \eqref{diagExt} into the cokernel of \eqref{IndP}.
\item Assume $\Zb$ connected.
In all degrees, there is a $k$-linear injection from the kernel of the upper horizontal arrow of \eqref{diagExt} into the kernel of \eqref{IndP}.
\end{enumerate}
\end{lemm}

\begin{proof}
We prove (i).
The map in question is induced by the composite right-hand side vertical arrow of \eqref{diagExt} composed with \eqref{IndPL'}.
Let $\Ec$ be an extension of $\Ind_{B_{L'}^-}^{L'} \chi'$ by $\Ind_{B_{L'}^-}^{L'} \chi$ with central character $\chi_{|Z_{L'}}$ (so that $\Ec$ extends to $LL'$).
Then $\Ind_{P^-L'}^G(\sigma_0 \otimes \Ec)$ is an extension of $\Ind_{P^-}^G \sigma'$ by $\Ind_{P^-}^G \sigma$ on which $\Zt$ acts via $\zeta$.
There are $L$-equivariant isomorphisms
\begin{equation*}
\Ord_P \left( \Ind_{P^-L'}^G \left( \sigma_0 \otimes \Ec \right) \right) \cong \Ord_{LB_{L'}} \left( \sigma_0 \otimes \Ec \right) \cong \sigma_0 \otimes \Ord_{B_{L'}} \Ec.
\end{equation*}
The first one results from \cite[Proposition 4.3.4]{Em1} and the second one from the fact that $U_{L'}$ acts trivially on $\sigma_0$ (note that $\Ord_{B_{L'}} \Ec$ extends to $L$).
If the class of $\Ec$ is not in the image of the upper horizontal arrow of \eqref{diagExt}, then there is a $T$-equivariant isomorphism $\Ord_{B_{L'}} \Ec \cong \chi$, hence an $L$-equivariant isomorphism $\Ord_P(\Ind_{P^-L'}^G(\sigma_0 \otimes \Ec)) \cong \sigma$, thus the class of $\Ind_{P^-L'}^G(\sigma_0 \otimes \Ec)$ is not in the image of \eqref{IndP}.

We prove (ii).
The map in question is induced by the left-hand side composite vertical arrow of \eqref{diagExt}.
Thus, it is enough to prove that the latter is injective.
We assume $\Zb$ connected.
Recall that this is equivalent to the existence of fundamental coweights $(\lambda_\alpha)_{\alpha \in \Delta}$ (cf. \cite[Proposition 2.1.1]{BH}).
We let $\Tb' \subseteq \Tb$ be the closed subgroup generated by the images of $(\lambda_\alpha)_{\alpha \in \Delta_{\Lb'}}$, so that $\Tb' \subseteq \Zb_\Lb$ and the product induces an isomorphism $\Tb' \times \Zb_{\Lb'} \iso \Tb$.
There is a commutative diagram of $k$-vector spaces
\begin{equation*} \begin{tikzcd}
\Ext_{T,\chi_{|Z_{L'}}}^\bullet \left( \chi', \chi \right) \rar{\sim} \dar & \Ext_{T'}^\bullet \left( \chi'_{|T'}, \chi_{|T'} \right) \dar[equal] \\
\Ext_{L,\chi_{|\Zt}}^\bullet \left( \chi', \chi \right) \rar \dar & \Ext_{T'}^\bullet \left( \chi'_{|T'}, \chi_{|T'} \right) \dar[hook] \\
\Ext_{L,\zeta_{|\Zt}}^\bullet \left( \sigma', \sigma \right) \rar & \Ext_{T'}^\bullet \left( \sigma'_{|T'}, \sigma_{|T'} \right)
\end{tikzcd} \end{equation*}
where the horizontal arrows are induced by restricting representations to $T'$ (the upper one is bijective with inverse induced by tensoring representations with $\chi_{|Z_{L'}}$, and the middle and lower ones are well-defined since a locally admissible smooth representation of $L$ over $k$ is locally $Z_L$-finite, cf. \cite[Lemma 2.3.4]{Em1}), the left-hand side vertical arrows are the same as in \eqref{diagExt} and the lower right-hand side vertical arrow is induced by tensoring representations with ${\sigma_0}_{|T'}$ (it is injective since $T'$ acts on $\sigma_0$ via $\zeta \chi^{-1}$).
\end{proof}

Now, we do not assume $\Gb^\der$ simply connected.
Instead we take a $z$-extension of $\Gb$, i.e. an exact sequence of affine algebraic $F$-groups
\begin{equation*}
1 \to \Zbt \to \Gbt \to \Gb \to 1
\end{equation*}
such that $\Gbt$ is reductive with simply connected derived subgroup and $\Zbt$ is a central torus (cf. \cite[§~3.1]{CT}).
The projection $\Gbt \twoheadrightarrow \Gb$ identifies the corresponding root systems.
We let $\Pbt \subseteq \Gbt$ be the standard parabolic subgroups corresponding to $\Pb$ and $\Lbt \subseteq \Pbt$ be the standard Levi subgroup corresponding to $\Lb$.
Note that $\Lbt$ is a $z$-extensions of $\Lb$.
We let $\sigma'$ be a locally admissible smooth representation of $L$ over $k$ with central character $\zeta$.
By inflation, we obtain locally admissible smooth representations $\sigmat$ and $\sigmat'$ of $\Lt$ over $k$.
There is a commutative diagram of $k$-vector spaces
\begin{equation} \label{diagInf} \begin{tikzcd}
\Ext_L^\bullet \left( \sigma', \sigma \right) \rar \dar{\vertsim} & \Ext_G^\bullet \left( \Ind_{P^-}^G \sigma', \Ind_{P^-}^G \sigma \right) \dar{\vertsim} \\
\Ext_{\Lt,\zeta_{|\Zt}}^\bullet \left( \sigmat', \sigmat \right) \rar & \Ext_{\Gt,\zeta_{|\Zt}}^\bullet \left( \Ind_{\Pt^-}^{\Gt} \sigmat', \Ind_{\Pt^-}^{\Gt} \sigmat \right)
\end{tikzcd} \end{equation}
where the horizontal arrows are induced by the functors $\Ind_{P^-}^G$ and $\Ind_{\Pt^-}^{\Gt}$ and the vertical arrows are induced by inflating representations to $\Lt$ and $\Gt$ (they are well defined and bijective since $\zeta_{|\Zt}$ is trivial).

\begin{prop} \label{prop:Ext1prelim}
Assume $F=\Qp$ and $\Gb$ split.
Let $\sigma$ be a locally admissible smooth representation of $L$ over $k$ with central character $\zeta : Z_L \to k^\times$.
\begin{enumerate}
\item Assume $\Delta_\Lb^\perp \neq \emptyset$ and let $\alpha \in \Delta_\Lb^\perp$.
If $\zeta \circ \alpha^\vee \neq \omega^{-1}$, then the $k$-linear injection
\begin{equation*}
\Ext_L^1 \left( \sigma^\alpha \otimes \left( \omega^{-1} \circ \alpha \right), \sigma \right) \hookrightarrow \Ext_G^1 \left( \Ind_{P^-}^G \sigma^\alpha \otimes \left( \omega^{-1} \circ \alpha \right), \Ind_{P^-}^G \sigma \right)
\end{equation*}
induced by the functor $\Ind_{P^-}^G$ is not surjective.
\item If $p=2$, then the functor $\Ind_{P^-}^G$ induces a $k$-linear injection
\begin{equation*}
\Ext_L^1 \left( \sigma, \sigma \right) \hookrightarrow \Ext_G^1 \left( \Ind_{P^-}^G \sigma, \Ind_{P^-}^G \sigma \right)
\end{equation*}
whose cokernel is of dimension at least $\card \{ \alpha \in \Delta_\Lb^\perp \mid \zeta \circ \alpha^\vee = 1 \}$.
\end{enumerate}
\end{prop}

\begin{rema}
We expect the results to hold true for a non-split reductive group with $\Delta_\Lb^{\perp,1}$ instead of $\Delta_\Lb^\perp$.
\end{rema}

\begin{proof}
By taking a $z$-extension of $\Gb$ and using \eqref{diagInf}, we can and do assume that $\Gb^\der$ is simply connected and prove analogous results for the morphism \eqref{IndP}.

Assume $\Delta_\Lb^\perp \neq \emptyset$ and let $\alpha \in \Delta_\Lb^\perp$.
We use Lemma \ref{lemm:Extprelim} (i) with $\Lb'$ defined by $\Delta_{\Lb'}=\{\alpha\}$, $\chi = \zeta \circ \alpha^\vee \circ \mu_\alpha$, and $\chi' = s_\alpha(\chi)(\omega^{-1} \circ \alpha)$, so that $\sigma' = \sigma^\alpha \otimes (\omega^{-1} \circ \alpha)$ (since $\sigma_0^\alpha=\sigma_0$ by Lemma \ref{lemm:alpha} with $1$ instead of $\omega$).
If $\zeta \circ \alpha^\vee \neq \omega^{-1}$, then the upper horizontal arrow of \eqref{diagExt} in degree $1$ is not surjective by the mod $p$ analogue of \cite[Lemme 3.1.4]{JHC} (since $\chi \circ \alpha^\vee = \zeta \circ \alpha^\vee = 1$), thus \eqref{IndP} in degree $1$ is not surjective, hence (i).

We use Lemma \ref{lemm:Extprelim} (i) with $\Lb'$ defined by $\Delta_{\Lb'}=\{\alpha \in \Delta_\Lb^\perp \mid \zeta \circ \alpha^\vee = 1\}$, $\chi = \zeta \circ \sum_{\alpha \in \Delta_{\Lb'}} (\alpha^\vee \circ \mu_\alpha)$, and $\chi'=\chi$, so that $\sigma'=\sigma$.
If $p=2$, then the cokernel of the upper horizontal arrow of \eqref{diagExt} in degree $1$ is of dimension at least $\card \Delta_{\Lb'}$ by \cite[Théorème 3.2.4 (ii) and Remarque 3.2.5 (ii)]{JHC} (since $\chi \circ \alpha^\vee = \zeta \circ \alpha^\vee = 1$ so that $s_\alpha(\chi)=\chi$ by Lemma \ref{lemm:alpha} for all $\alpha \in \Delta_{\Lb'}$), noting that all the extensions constructed in loc. cit. have a central character (cf. \cite[Lemme 3.1.5]{JHC}), thus the cokernel of \eqref{IndP} in degree $1$ is of dimension at least $\card \Delta_{\Lb'}$, hence (ii).
\end{proof}

\begin{prop} \label{prop:Ext2prelim}
Assume $F=\Qp$, $\Gb$ split and $\Zb$ connected.
Let $\sigma$ be a locally admissible smooth representation of $L$ over $k$ with central character $\zeta : Z_L \to k^\times$.
If $p \neq 2$, then the functor $\Ind_{P^-}^G$ induces a $k$-linear morphism
\begin{equation*}
\Ext_L^2 \left( \sigma, \sigma \right) \to \Ext_G^2 \left( \Ind_{P^-}^G \sigma, \Ind_{P^-}^G \sigma \right)
\end{equation*}
whose kernel is of dimension at least $\card \{ \alpha \in \Delta_\Lb^\perp \mid \zeta \circ \alpha^\vee = \omega^{-1} \}$.
\end{prop}

\begin{proof}
By taking a $z$-extension of $\Gb$ (noting that the centre of $\Gbt$ is also connected because $\Zbt$ is connected) and using \eqref{diagInf}, we can and do assume that $\Gb^\der$ is simply connected and prove an analogous result for the morphism \eqref{IndP}.

We use Lemma \ref{lemm:Extprelim} (ii) with $\Lb'$ defined by $\Delta_{\Lb'}=\{\alpha \in \Delta_\Lb^\perp \mid \zeta \circ \alpha^\vee = \omega^{-1}\}$ and $\chi'=\chi$, so that $\sigma'=\sigma$.
If $p \neq 2$, then we see in the proof of \cite[Théorème 3.2.4 (i)]{JHC} that the kernel of the upper horizontal arrow of \eqref{diagExt} in degree $2$ is of dimension at least $\card \Delta_{\Lb'}$ (since $\chi \circ \alpha^\vee = \zeta \circ \alpha^\vee = \omega^{-1}$ for all $\alpha \in \Delta_{\Lb'}$), thus the kernel of \eqref{diagExt} in degree $2$ is also of dimension at least $\card \Delta_{\Lb'}$, hence the result.
\end{proof}

\subsection{Extensions between parabolically induced representations} \label{ssec:Ext}

We begin with a result when there is no inclusion between the two parabolic subgroups, assuming a special case of Conjecture \ref{conj:HOrd} (see also Remark \ref{rema:IwJ=1}).

\begin{prop} \label{prop:Ext1}
Let $\Pb = \Lb \Nb,\Pb' = \Lb' \Nb'$ be standard parabolic subgroups and $\sigma,\sigma'$ be admissible smooth representations of $L,L'$ respectively over $k$.
Assume Conjecture \ref{conj:HOrd} is true for $A=k$, $n=1$ and $\Iw^J=1$.
If $\Pb' \not \subseteq \Pb$, $\Pb \not \subseteq \Pb'$, and $\sigma,\sigma'$ are right,left cuspidal respectively, then
\begin{equation*}
\Ext_G^1 \left( \Ind_{P'^-}^G \sigma', \Ind_{P^-}^G \sigma \right) = 0.
\end{equation*}
\end{prop}

\begin{proof}
We put $I \coloneqq \Delta_\Lb$ and $J \coloneqq \Delta_{\Lb'}$.
Using \eqref{AHV}, \eqref{SEExt1} with $\pi=\Ind_{P_I^-}^G \sigma$ and $\Pb_J$, $\Lb_J$, $\sigma'$ instead of $\Pb$, $\Lb$, $\sigma$ respectively yields an exact sequence of $k$-vector spaces
\begin{multline} \label{SEExtIJ}
0 \to \Ext_{L_J}^1 \left( \sigma', \Ind_{L_J \cap P_I^-}^{L_J} \left( \Ord_{L_I \cap P_J} \sigma \right) \right) \to \Ext_G^1 \left( \Ind_{P_J^-}^G \sigma', \Ind_{P_I^-}^G \sigma \right) \\
\to \Hom_{L_J} \left( \sigma', \HOrd[1]_{P_J} \left( \Ind_{P_I^-}^G \sigma \right) \right).
\end{multline}

Assume $I \not \subseteq J$ and $\sigma$ right cuspidal.
Then $\Ord_{L_I \cap P_J} \sigma = 0$ (since $\Lb_I \cap \Pb_J$ is a proper parabolic subgroup of $\Lb_I$) and there is a natural $L_J$-equivariant isomorphism
\begin{equation} \label{H1OrdIJ}
\Ind_{L_J \cap P_I^-}^{L_J} \left( \HOrd[1]_{L_I \cap P_J} \sigma \right) \iso \HOrd[1]_{P_J} \left( \Ind_{P_I^-}^G \sigma \right).
\end{equation}
Indeed, by assumption \eqref{HOrd} is a natural $L_J$-equivariant isomorphism
\begin{equation*}
\Ind_{L_J \cap P_I^-}^{L_J} \left( \HOrd[1]_{L_I \cap P_J} \sigma \right) \iso \HOrd[1]_{P_J} \left( \cind_{P_I-}^{P_I^-P_J} \sigma \right)
\end{equation*}
and by Proposition \ref{prop:HOrd1} (i) with $n=1$, we have $\HOrd[1]_{P_J}(\cind_{P_I-}^{P_I^- \Iw^J P_J} \sigma)=0$ for all $\Iw^J \in \IW^J$ such that $\Iw^J \neq 1$ (since either $d_{\Iw^J}=1$ and $\Ord_{L_I \cap P_{I \cap \Iw^J(J)}} \sigma = 0$, or $d_{\Iw^J}>1$ and $1- [F:\Qp] d_{\Iw^J} < 0$).
Thus, we deduce \eqref{H1OrdIJ} from Proposition \ref{prop:FilPJ}.

Assume $J \not \subseteq I$ and $\sigma'$ left cuspidal.
Then $\sigma'_{L_J \cap N_I^-}=0$ (since $\Lb_J \cap \Pb_I^-$ is a proper parabolic subgroup of $\Lb_J$) and using \eqref{Frob} with $\pi=\sigma'$ and $\Lb_J$, $\Lb_J \cap \Pb_I^-$, $\Lb_{J \cap I}$, $\Lb_J \cap \Nb_I^-$, $\HOrd[1]_{L_I \cap P_J} \sigma$ instead of $\Gb$, $\Pb$, $\Lb$, $\Nb$, $\sigma$ respectively, we obtain
\begin{equation*}
\Hom_{L_J} \left( \sigma', \Ind_{L_J \cap P_I^-}^{L_J} \left( \HOrd[1]_{L_I \cap P_J} \sigma \right) \right) = 0.
\end{equation*}
We conclude using \eqref{SEExtIJ}.
\end{proof}

Now, we prove unconditional results whenever there is an inclusion between the two parabolic subgroups.
We treat the cases $F=\Qp$ and $F \neq \Qp$ separately.

\begin{theo} \label{theo:Ext1Qp}
Assume $F=\Qp$.
Let $\Pb = \Lb \Nb,\Pb' = \Lb' \Nb'$ be standard parabolic subgroups and $\sigma,\sigma'$ be admissible smooth representations of $L,L'$ respectively over $k$.
\begin{enumerate}
\item If $\Pb'=\Pb$, $\sigma,\sigma'$ are supercuspidal and $\sigma' \not \cong \sigma^\alpha \otimes (\omega^{-1} \circ \alpha)$ for all $\alpha \in \Delta_\Lb^{\perp,1}$, then the functor $\Ind_{P^-}^G$ induces a $k$-linear isomorphism
\begin{equation*}
\Ext_L^1 \left( \sigma', \sigma \right) \iso \Ext_G^1 \left( \Ind_{P^-}^G \sigma', \Ind_{P^-}^G \sigma \right).
\end{equation*}
\item If $\Pb' \subsetneq \Pb$ and $\sigma$ is right cuspidal, then the functor $\Ind_{P^-}^G$ induces a $k$-linear isomorphism
\begin{equation*}
\Ext_L^1 \left( \Ind_{L \cap P'^-}^L \sigma', \sigma \right) \iso \Ext_G^1 \left( \Ind_{P'^-}^G \sigma', \Ind_{P^-}^G \sigma \right).
\end{equation*}
\item If $\Pb \subsetneq \Pb'$ and $\sigma'$ is left cuspidal, then the functor $\Ind_{P'^-}^G$ induces a $k$-linear isomorphism
\begin{equation*}
\Ext_{L'}^1 \left( \sigma', \Ind_{L' \cap P^-}^{L'} \sigma \right) \iso \Ext_G^1 \left( \Ind_{P'^-}^G \sigma', \Ind_{P^-}^G \sigma \right).
\end{equation*}
\end{enumerate}
\end{theo}

\begin{rema} \label{rema:Ext1Qp}
Assume $\Pb'=\Pb$ and $\sigma,\sigma'$ irreducible.
In general, we do not know the dimension of the cokernel of the $k$-linear injection $\Ext_L^1(\sigma',\sigma) \hookrightarrow \Ext_G^1(\Ind_{P^-}^G \sigma',\Ind_{P^-}^G \sigma)$ induced by $\Ind_{P^-}^G$, but we prove that it is at most $\card\{\alpha \in \Delta_\Lb^{\perp,1} \mid \sigma' \cong \sigma^\alpha \otimes (\omega^{-1} \circ \alpha)\}$ whenever $\sigma$ is right cuspidal or $\sigma'$ is left cuspidal (see the proof).
If $\sigma,\sigma'$ are supersingular, then letting $\zeta : Z_L \to k^\times$ denote the central character of $\sigma$ (cf. \cite[Lemma 4.1.7]{Em2}), we expect this dimension to be equal to
\begin{equation*}
\card \left\{ \alpha \in \Delta_\Lb^{\perp,1} \middlevert \sigma' \cong \sigma^\alpha \otimes \left( \omega^{-1} \circ \alpha \right) \text{ and } \zeta \circ \alpha^\vee \neq \omega^{-1} \right\}
\end{equation*}
except when $p=2$ and in some exceptional cases (cf. \cite[Remarque 3.2.5]{JHC} when $\Gb$ is split and $\Pb=\Bb$).
We prove this when $\Gb$ is split and $\Zb$ is connected (see Theorem \ref{theo:Ext1Zcnx} below), in which case the cardinal above is equal to $1$ if $\sigma' \cong \sigma^\alpha \otimes (\omega^{-1} \circ \alpha) \not \cong \sigma$ for some $\alpha \in \Delta_\Lb^{\perp,1}$ and $0$ otherwise by Lemma \ref{lemm:gen}.
When $\Gb$ is split but $\Zb$ is not connected, one could prove that the cardinal above is a lower bound using Proposition \ref{prop:Ext1prelim} (i) and some generalisation of \cite[§~2.2]{JHC} for $\Pb \neq \Bb$.
\end{rema}

\begin{proof}
We prove slightly more general results.

Assume $\Pb' \subseteq \Pb$ and $\sigma$ satisfies the condition in Corollary \ref{coro:HOrd} (ii), e.g. $\sigma$ is right cuspidal.
Using \cite[Proposition 4.3.4]{Em1} and Corollary \ref{coro:HOrd} (ii), \eqref{SEExt1} with $\pi = \Ind_{P^-}^G \sigma$ and $\Ind_{L \cap P'^-}^L \sigma'$ instead of $\sigma$ yields an exact sequence of $k$-vector spaces
\begin{multline} \label{SEExtQpR}
0 \to \Ext_L^1 \left( \Ind_{L \cap P'^-}^L \sigma', \sigma \right) \to \Ext_G^1 \left( \Ind_{P'^-}^G \sigma', \Ind_{P^-}^G \sigma \right) \\
\to \bigoplus_{\alpha \in \Delta_\Lb^{\perp,1}} \Hom_L \left( \Ind_{L \cap P'^-}^L \sigma', \sigma^\alpha \otimes \left( \omega^{-1} \circ \alpha \right) \right).
\end{multline}
If $\Pb'=\Pb$ and $\sigma,\sigma'$ are irreducible, then $\sigma^\alpha \otimes (\omega^{-1} \circ \alpha)$ is also irreducible for all $\alpha \in \Delta_\Lb^{\perp,1}$, and thus the last term of \eqref{SEExtQpR} has dimension equal to $\card\{\alpha \in \Delta_\Lb^{\perp,1} \mid \sigma' \cong \sigma^\alpha \otimes (\omega^{-1} \circ \alpha)\}$, hence (i).
If $\Pb' \subsetneq \Pb$ and $\sigma$ is right cuspidal, then $\Lb \cap \Pb'$ is a proper parabolic subgroup of $\Lb$ and
\begin{equation*}
\Ord_{L \cap P'} \left( \sigma^\alpha \otimes \left( \omega^{-1} \circ \alpha \right) \right) \cong \left( \Ord_{L \cap P'} \sigma \right)^\alpha \otimes \left( \omega^{-1} \circ \alpha \right) = 0
\end{equation*}
for all $\alpha \in \Delta_\Lb^{\perp,1}$, thus the last term of \eqref{SEExtQpR} is zero by \cite[Theorem 4.4.6]{Em1}, hence (ii).

Assume $\Pb \subseteq \Pb'$ and $\sigma'$ satisfies the condition in Corollary \ref{coro:HJ} (ii) for $\Pb'^-=\Lb'\Nb'^-$, e.g. $\sigma'$ is left cuspidal.
Using \cite[Theorem 5.3, 3]{VigAdj} and Corollary \ref{coro:HJ} (ii) for $\Pb'^-=\Lb'\Nb'^-$ (see Remark \ref{rema:HJ}), \eqref{SEExt2} with $\pi = \Ind_{P'^-}^G \sigma'$ and $\Pb'^-$, $\Lb'$, $\Nb'^-$, $\Ind_{L' \cap P^-}^{L'} \sigma$ instead of $\Pb$, $\Lb$, $\Nb$, $\sigma$ respectively yields an exact sequence of $k$-vector spaces
\begin{multline} \label{SEExtQpL}
0 \to \Ext_{L'}^1 \left( \sigma', \Ind_{L' \cap P^-}^{L'} \sigma \right) \to \Ext_G^1 \left( \Ind_{P'^-}^G \sigma', \Ind_{P^-}^G \sigma \right) \\
\to \bigoplus_{\alpha \in \Delta_{\Lb'}^{\perp,1}} \Hom_{L'} \left( \sigma'^\alpha \otimes \left( \omega^{-1} \circ \alpha \right), \Ind_{L' \cap P^-}^{L'} \sigma \right).
\end{multline}
If $\Pb'=\Pb$ and $\sigma,\sigma'$ are irreducible, then $\sigma'^\alpha \otimes (\omega^{-1} \circ \alpha)$ is also irreducible for all $\alpha \in \Delta_\Lb^{\perp,1}$, and thus the last term of \eqref{SEExtQpL} has dimension equal to $\card\{\alpha \in \Delta_{\Lb'}^{\perp,1} \mid \sigma' \cong \sigma^\alpha \otimes (\omega^{-1} \circ \alpha)\}$, hence (i).
If $\Pb \subsetneq \Pb'$ and $\sigma'$ is left cuspidal, then $\Lb' \cap \Pb^-$ is a proper parabolic subgroup of $\Lb'$ and
\begin{equation*}
\left( \sigma'^\alpha \otimes \left( \omega^{-1} \circ \alpha \right) \right)_{L' \cap N^-} \cong \left( \sigma'_{L' \cap N^-} \right)^\alpha \otimes \left( \omega^{-1} \circ \alpha \right) = 0
\end{equation*}
for all $\alpha \in \Delta_\Lb^{\perp,1}$, thus the last term of \eqref{SEExtQpL} is zero using \eqref{Frob} with $\pi = \sigma'^\alpha \otimes (\omega^{-1} \circ \alpha)$ and $\Lb'$, $\Lb' \cap \Pb^-$, $\Lb' \cap \Lb$, $\Lb' \cap \Nb^-$ instead of $\Gb$, $\Pb$, $\Lb$, $\Nb$ respectively, hence (iii).
\end{proof}

\begin{theo} \label{theo:Ext1F}
Assume $F \neq \Qp$.
Let $\Pb = \Lb \Nb$ be a standard parabolic subgroup.
The functor $\Ind_{P^-}^G$ induces an $A$-linear isomorphism
\begin{equation*}
\Ext_L^1 \left( \sigma', \sigma \right) \iso \Ext_G^1 \left( \Ind_{P^-}^G \sigma', \Ind_{P^-}^G \sigma \right)
\end{equation*}
for all locally admissible smooth representations $\sigma,\sigma'$ of $L$ over $A$.
\end{theo}

\begin{proof}
Let $\sigma,\sigma'$ be locally admissible smooth representations of $L$ over $A$.
Using \cite[Proposition 4.3.4]{Em1} and Corollary \ref{coro:HOrd} (i), \eqref{SEExt1} with $\pi=\Ind_{P^-}^G \sigma$ and $\sigma'$ instead of $\sigma$ yields the isomorphism in the statement.
\end{proof}

\begin{coro} \label{coro:Ext1F}
Assume $F \neq \Qp$.
Let $\Pb = \Lb \Nb,\Pb' = \Lb' \Nb'$ be standard parabolic subgroups and $\sigma,\sigma'$ be admissible smooth representations of $L,L'$ respectively over $k$.
\begin{enumerate}
\item If $\Pb' \subseteq \Pb$, then the functor $\Ind_{P^-}^G$ induces a $k$-linear isomorphism
\begin{equation*}
\Ext_L^1 \left( \Ind_{L \cap P'^-}^L \sigma', \sigma \right) \iso \Ext_G^1 \left( \Ind_{P'^-}^G \sigma', \Ind_{P^-}^G \sigma \right).
\end{equation*}
\item If $\Pb \subseteq \Pb'$, then the functor $\Ind_{P'^-}^G$ induces a $k$-linear isomorphism
\begin{equation*}
\Ext_{L'}^1 \left( \sigma', \Ind_{L' \cap P^-}^{L'} \sigma \right) \iso \Ext_G^1 \left( \Ind_{P'^-}^G \sigma', \Ind_{P^-}^G \sigma \right).
\end{equation*}
\end{enumerate}
\end{coro}

\begin{rema} \label{rema:Extn}
Theorem \ref{theo:Ext1Qp} (i) and Theorem \ref{theo:Ext1F} are encompassed in a more general (but conditional to a conjecture of Emerton) result.
Let $\Pb = \Lb \Nb$ be a standard parabolic subgroup, $\sigma,\sigma'$ be locally admissible smooth representations of $L$ over $A$ and $n \in \N$.
The functor $\Ind_{P^-}^G$ induces an $A$-linear morphism
\begin{equation} \label{Extn}
\Ext_L^n \left( \sigma', \sigma \right) \to \Ext_G^n \left( \Ind_{P^-}^G \sigma', \Ind_{P^-}^G \sigma \right)
\end{equation}
and there is a spectral sequence of $A$-modules (cf. \cite[(3.7.4)]{Em2})
\begin{equation} \label{SSExt}
\Ext_L^i \left( \sigma', \ROrd[j]_P \left( \Ind_{P^-}^G \sigma \right) \right) \Rightarrow \Ext_G^{i+j} \left( \Ind_{P^-}^G \sigma', \Ind_{P^-}^G \sigma \right)
\end{equation}
where $\ROrd_P$ denotes the right derived functors of $\Ord_P : \Mod_G^\ladm(A) \to \Mod_L^\ladm(A)$.
Now assume that \cite[Conjecture 3.7.2]{Em2} is true, i.e. $\ROrd_P \iso \HOrd_P$.
Using Corollary \ref{coro:HOrd}, one can deduce from \eqref{SSExt} that:
\begin{itemize}
\item if $n < [F:\Qp]$, then \eqref{Extn} is an isomorphism;
\item if $n = [F:\Qp]$, then \eqref{Extn} is injective and if furthermore $\sigma,\sigma'$ are supercuspidal, then the dimension of its cokernel is at most $\card\{\alpha \in \Delta_\Lb^{\perp,1} \mid \sigma' \cong \sigma^\alpha \otimes (\omega^{-1} \circ \alpha)\}$.
\end{itemize}
One can also generalise Proposition \ref{prop:Ext1} and Theorem \ref{theo:Ext1Qp} (ii) and (iii) in all degrees $n \leq [F:\Qp]$.
\end{rema}

Finally, we complete Theorem \ref{theo:Ext1Qp} (i) when $\Gb$ is split and $\Zb$ is connected.

\begin{theo} \label{theo:Ext1Zcnx}
Assume $F=\Qp$, $\Gb$ split and $\Zb$ connected.
Let $\Pb = \Lb \Nb$ be a standard parabolic subgroup and $\sigma,\sigma'$ be supersingular representations of $L$ over $k$.
\begin{enumerate}
\item If $\sigma' \cong \sigma^\alpha \otimes (\omega^{-1} \circ \alpha) \not \cong \sigma$ for some $\alpha \in \Delta_\Lb^\perp$, then $\Ext_L^1(\sigma',\sigma)=0$ and
\begin{equation*}
\dim_k \Ext_G^1 \left( \Ind_{P^-}^G \sigma', \Ind_{P^-}^G \sigma \right) = 1.
\end{equation*}
\item If either $\sigma' \cong \sigma$ and $p \neq 2$, or $\sigma' \not \cong \sigma^\alpha \otimes (\omega^{-1} \circ \alpha)$ for any $\alpha \in \Delta_\Lb^\perp$, then the functor $\Ind_{P^-}^G$ induces a $k$-linear isomorphism
\begin{equation*}
\Ext_L^1 \left( \sigma', \sigma \right) \iso \Ext_G^1 \left( \Ind_{P^-}^G \sigma', \Ind_{P^-}^G \sigma \right).
\end{equation*}
\item If $p=2$, then the functor $\Ind_{P^-}^G$ induces a $k$-linear injection
\begin{equation*}
\Ext_L^1 \left( \sigma', \sigma \right) \hookrightarrow \Ext_G^1 \left( \Ind_{P^-}^G \sigma', \Ind_{P^-}^G \sigma \right)
\end{equation*}
whose cokernel is of dimension $\card \{ \alpha \in \Delta_\Lb^\perp \mid \sigma' \cong \sigma^\alpha \}$.
\end{enumerate}
\end{theo}

\begin{proof}
Since $\sigma$ is absolutely irreducible, it has a central character $\zeta : Z_L \to k^\times$ (cf. \cite[Lemma 4.1.7]{Em2}).

We first assume that $\sigma' \cong \sigma^\alpha \otimes (\omega^{-1} \circ \alpha) \not \cong \sigma$ for some $\alpha \in \Delta_\Lb^\perp$.
We have $\zeta \circ \alpha^\vee \neq \omega^{-1}$ by Lemma \ref{lemm:alpha}, so that $\sigma$ and $\sigma'$ have distinct central characters by Lemma \ref{lemm:gen}, thus $\Ext_L^1(\sigma',\sigma)=0$ (cf. \cite[Proposition 8.1]{PasExt}).
Furthermore, $\sigma' \not \cong \sigma^\beta \otimes (\omega^{-1} \circ \beta)$ for any $\beta \in \Delta_\Lb^\perp \backslash \{\alpha\}$ (since the central characters are distinct by Lemma \ref{lemm:gen}).
Using \eqref{SEExtQpR} with $\Pb'=\Pb$ and $\Lb'=\Lb$, we deduce that
\begin{equation*}
\dim_k \Ext_G^1 \left( \Ind_{P^-}^G \sigma', \Ind_{P^-}^G \sigma \right) \leq 1.
\end{equation*}
Since the left-hand side is non-zero by Proposition \ref{prop:Ext1prelim} (i), this proves (i).

We now prove (ii).
If $\sigma' \not \cong \sigma^\alpha \otimes (\omega^{-1} \circ \alpha)$ for any $\alpha \in \Delta_\Lb^\perp$, then the result follows from Theorem \ref{theo:Ext1Qp} (i).
Assume that $\sigma' \cong \sigma$ and $p \neq 2$.
The terms of low degree of \eqref{SSExt} form an exact sequence of $k$-vector spaces
\begin{multline} \label{SE5}
0 \to \Ext_L^1 \left( \sigma, \sigma \right) \to \Ext_G^1 \left( \Ind_{P^-}^G \sigma, \Ind_{P^-}^G \sigma \right) \to \Hom_L \left( \sigma, \ROrd[1]_P \left( \Ind_{P^-}^G \sigma \right) \right) \\
\to \Ext_L^2 \left( \sigma, \sigma \right) \to \Ext_G^2 \left( \Ind_{P^-}^G \sigma, \Ind_{P^-}^G \sigma \right).
\end{multline}
Since there is an injection of functors $\ROrd[1]_P \hookrightarrow \HOrd[1]_P$ (cf. \cite[Remark 3.7.3]{Em2}), we deduce from Corollary \ref{coro:HOrd} (ii) and Lemma \ref{lemm:gen} that
\begin{equation*}
\dim_k \Hom_L \left( \sigma, \ROrd[1]_P \left( \Ind_{P^-}^G \sigma \right) \right) \leq \card \left\{ \alpha \in \Delta_\Lb^\perp \middlevert \zeta \circ \alpha^\vee = \omega^{-1} \right\}.
\end{equation*}
Thus, we deduce from Proposition \ref{prop:Ext2prelim} that the third arrow of \eqref{SE5} is zero, hence the result.

Finally, we assume $p=2$ and we prove (iii).
By Proposition \ref{prop:Ext1prelim} (ii) we have a lower bound and using \eqref{SEExtQpR} with $\Pb'=\Pb$ and $\Lb'=\Lb$ we obtain an upper bound.
Using Lemmas \ref{lemm:alpha} and \ref{lemm:gen} together with the fact that $\omega=1$, we see that both are equal to $\card \{ \alpha \in \Delta_\Lb^\perp \mid \sigma' \cong \sigma^\alpha \}$.
\end{proof}

\begin{rema} \label{rema:ExtnZcnx}
Theorem \ref{theo:Ext1Zcnx} (i) can also be generalised in the context of Remark \ref{rema:Extn}.
Let $\Pb = \Lb \Nb$ be a standard parabolic subgroup and $\sigma,\sigma'$ be supersingular representations of $L$ over $k$ such that $\sigma' \cong \sigma^\alpha \otimes (\omega^{-1} \circ \alpha) \not \cong \sigma$ for some $\alpha \in \Delta_\Lb^\perp$.
Assume $\Gb$ split, $\Zb$ connected and \cite[Conjecture 3.7.2]{Em2} is true.
Using Corollary \ref{coro:HOrd} and Lemma \ref{lemm:gen}, one can deduce from \eqref{SSExt} that $\Ext^{[F:\Qp]}_L(\sigma',\sigma)=0$ and
\begin{equation*}
\dim_k \Ext^{[F:\Qp]}_G \left( \Ind_{P^-}^G \sigma', \Ind_{P^-}^G \sigma \right) = 1.
\end{equation*}
\end{rema}

\begin{coro} \label{coro:Ext1Zcnx}
Assume $\Gb$ split and $\Zb$ connected.
If Conjecture \ref{conj:HOrd} is true for $A=k$, $n=1$ and $\Iw^J=1$, then \cite[Conjecture 3.17]{JHB} is true.
\end{coro}

\begin{proof}
Even though \cite[Conjecture 3.17]{JHB} is formulated under the hypotheses $\Gb$ split, $\Zb$ connected and $\Gb^\der$ simply connected, we do not need the last one to prove it:
(i) is Proposition \ref{prop:Ext1}, which is conditional to Conjecture \ref{conj:HOrd} for $A=k$, $n=1$ and $\Iw^J=1$;
(ii) is Theorem \ref{theo:Ext1Zcnx} (i);
(iii) and (iv) are Corollary \ref{coro:Ext1F} (i) and (ii) respectively when $F \neq \Qp$, Theorem \ref{theo:Ext1Qp} (ii) and (iii) respectively when $F=\Qp$ and $\Pb' \neq \Pb$, Theorem \ref{theo:Ext1Zcnx} (ii) when $F=\Qp$, $\Pb'=\Pb$ and $p \neq 2$, and Theorem \ref{theo:Ext1Zcnx} (iii) when $F=\Qp$, $\Pb'=\Pb$ and $p=2$ (noting that if $p=2$, then $\omega=1$ and $\Ind_{P^-}^G \sigma$ is irreducible if and only if $\sigma^\alpha \not \cong \sigma$ for all $\alpha \in \Delta_\Lb^\perp$, cf. \cite[Lemma 5.8]{Abe} and Lemma \ref{lemm:gen}).
\end{proof}

\subsection{\texorpdfstring{Results for unitary continuous $p$-adic representations}{Results for unitary continuous p-adic representations}} \label{ssec:Extcont}

Let $H$ be a $p$-adic Lie group.
A \emph{continuous representation} of $H$ over $E$ is an $E$-Banach space $\Pi$ endowed with an $E$-linear action of $H$ such that the map $H \times \Pi \to \Pi$ is continuous.
It is \emph{admissible} if the continuous dual $\Pi^* \coloneqq \Hom_H^\cont(\Pi,E)$ is of finite type over the Iwasawa algebra $E \otimes_\Oc \Oc \llbrack H_0 \rrbrack$ for some (equivalently any) compact open subgroup $H_0 \subseteq H$ (cf. \cite{ST}).
It is \emph{unitary} if there exists an $H$-stable bounded open $\Oc$-lattice $\Pi^0 \subseteq \Pi$.
We write $\Ban_H^\admu(E)$ for the category of admissible unitary continuous representations of $H$ over $E$ and $H$-equivariant $E$-linear continuous morphisms.
It is an $E$-abelian category.

We fix a uniformiser $\varpi$ of $\Oc$.
Following \cite[§~2.4]{Em1}, we let $\Mod_H^{\varpi-\adm}(\Oc)^\fl$ be the category of $\varpi$-torsion-free $\varpi$-adically complete and separated $\Oc$-modules $\Pi^0$ such that $\Pi^0/\varpi\Pi^0$ is admissible as a smooth representation of $H$ over $k$ and $H$-equivariant $\Oc$-linear morphisms.
It is an $\Oc$-abelian category and the localised category $E \otimes_\Oc \Mod_H^{\varpi-\adm}(\Oc)^\fl$ is equivalent to $\Ban_H^\admu(E)$.

The $E$-vector spaces $\Ext_H^1(\Pi',\Pi)$ of Yoneda extensions between admissible unitary continuous representations $\Pi,\Pi'$ of $H$ over $E$ are computed in $\Ban_H^\admu(E)$.
For all $n \geq 1$, the $\Oc/\varpi^n\Oc$-modules $\Ext_H^1(\pi',\pi)$ of Yoneda extensions between admissible smooth representations $\pi,\pi'$ of $H$ over $\Oc/\varpi^n\Oc$ are computed in $\Mod_H^\adm(\Oc/\varpi^n\Oc)$.
The following result is a slight generalisation of \cite[Proposition B.2]{JH}.

\begin{prop} \label{prop:devExt1}
Let $H$ be a $p$-adic Lie group, $\Pi,\Pi'$ be admissible unitary continuous representations of $H$ and $\pi,\pi'$ be the reductions mod $\varpi$ of $H$-stable bounded open $\Oc$-lattices $\Pi^0,\Pi'^0$ of $\Pi,\Pi'$ respectively.
Assume that $\dim_k \Hom_H(\pi',\pi)<\infty$.
There is an $E$-linear isomorphism
\begin{equation*}
\Ext_H^1 \left( \Pi', \Pi \right) \cong E \otimes_\Oc \varprojlim_{n \geq 1} \left( \Ext_H^1 \left( \Pi'^0 / \varpi^n \Pi'^0, \Pi^0 / \varpi^n \Pi^0 \right) \right).
\end{equation*}
Furthermore, $\dim_E \Ext_H^1(\Pi',\Pi) \leq \dim_k \Ext_H^1(\pi', \pi)$.
\end{prop}

\begin{proof}
In the proof of \cite[Proposition B.2]{JH}, the hypothesis that $\pi'$ is of finite length is only used to prove that $\Hom_H(\Pi'^0/\varpi^n\Pi'^0,\Pi^0/\varpi^n\Pi^0)$ is of finite type over $\Oc/\varpi^n\Oc$ for all $n \geq 1$.
But this can be proved by induction using that $\dim_k \Hom_H(\pi',\pi)<\infty$.
\end{proof}

Let $\Pb = \Lb \Nb$ be a parabolic subgroup.
We recall that the continuous parabolic induction functor is defined for any continuous representation $\Sigma$ of $L$ over $E$ by
\begin{equation*}
\Ind_{P^-}^G \Sigma \coloneqq \left\{ f : G \to \Sigma \text{ continuous} \middlevert f(pg) = p \cdot f(g) \ \forall p \in P^-, \forall g \in G \right\}.
\end{equation*}
We obtain an $E$-linear exact functor $\Ind_{P^-}^G : \Ban_L^\admu(E) \to \Ban_G^\admu(E)$ (cf. \cite[§~4.1]{Em1}).
Furthermore, there is a natural $G$-equivariant $E$-linear continuous isomorphism (cf. \cite[Lemma 4.1.3]{Em1})
\begin{equation*}
\Ind_{P^-}^G \Sigma \cong E \otimes_\Oc \varprojlim_{n \geq 1} \left( \Ind_{P^-}^G \left( \Sigma^0 / \varpi^n \Sigma^0 \right) \right).
\end{equation*}
We extend the definition of the ordinary parts functor to any admissible unitary continuous representation $\Pi$ of $G$ over $E$ by setting
\begin{equation*}
\Ord_P \Pi \coloneqq E \otimes_\Oc \varprojlim_{n \geq 1} \left( \Ord_P \left( \Pi^0 / \varpi^n \Pi^0 \right) \right)
\end{equation*}
for some (equivalently any) $G$-stable bounded open $\Oc$-lattice $\Pi^0 \subseteq \Pi$.
We obtain an $E$-linear left-exact functor $\Ord_P : \Ban_G^\admu(E) \to \Ban_L^\admu(E)$ which is a left quasi-inverse and the right adjoint of $\Ind_{P^-}^G$ (cf. \cite[Theorem 3.4.8, Corollary 4.3.5 and Theorem 4.4.6]{Em1}).

\begin{defi}
We say that an admissible unitary continuous representation $\Sigma$ of $L$ over $E$ is \emph{right cuspidal} if $\Ord_Q \Sigma = 0$ for any proper parabolic subgroup $\Qb \subset \Lb$.
\end{defi}

\begin{rema} \label{rema:Jacquetcont}
We also extend the Jacquet functor to continuous representations of $G$ over $E$ by taking the Hausdorff completion of the $N$-coinvariants.
We obtain the left adjoint of $\Ind_P^G$ by Frobenius reciprocity and the universal property of coinvariants.
However, we do not know whether it preserves admissibility.
For unitary representations, it does not behave well with respect to reduction mod $\varpi^n$ ($n \geq 1$).
Nevertheless, we say that an admissible unitary continuous representation $\Sigma$ of $L$ over $E$ is \emph{left cuspidal} if $\Sigma_{N_Q} = 0$ for any proper parabolic subgroup $\Qb \subset \Lb$ with unipotent radical $\Nb_\Qb$.
\end{rema}

We now turn to extensions computations.
Our main tool is the following result, which gives a weak $p$-adic analogue of the exact sequence \eqref{SEExt1}.

\begin{prop} \label{prop:SEExtcont}
Let $\Pb = \Lb \Nb$ be a standard parabolic subgroup, $\Sigma,\Sigma'$ be admissible unitary continuous representations of $L$ respectively over $E$ and $\sigma,\sigma'$ be the reductions mod $\varpi$ of $L$-stable bounded open $\Oc$-lattices $\Sigma^0,\Sigma'^0$ of $\Sigma,\Sigma'$ respectively.
Assume that $\dim_k \Hom_L(\sigma',\sigma)<\infty$.
There is a natural exact sequence of $E$-vector spaces
\begin{multline*}
0 \to \Ext_L^1 \left( \Sigma', \Sigma \right) \to \Ext_G^1 \left( \Ind_{P^-}^G \Sigma', \Ind_{P^-}^G \Sigma \right) \\
\to \Hom_L \left( \Sigma', E \otimes_\Oc \varprojlim_{n \geq 1} \left( \HOrd[1]_P \left( \Ind_{P^-}^G \left( \Sigma^0 / \varpi^n \Sigma^0 \right) \right) \right) \right).
\end{multline*}
\end{prop}

\begin{proof}
For all $n \geq 1$, \eqref{SEExt1} with $A=\Oc/\varpi^n\Oc$, $\pi=\Ind_{P^-}^G(\Sigma^0/\varpi^n\Sigma^0)$ and $\sigma=\Sigma'^0/\varpi^n\Sigma'^0$ yields, using \cite[Proposition 4.3.4]{Em1}, an exact sequence of $\Oc/\varpi^n\Oc$-modules
\begin{multline} \label{SEExt1A}
0 \to \Ext_L^1 \left( \Sigma'^0 / \varpi^n \Sigma'^0, \Sigma^0 / \varpi^n \Sigma^0 \right) \to \Ext_G^1 \left( \Ind_{P^-}^G \left( \Sigma'^0 / \varpi^n \Sigma'^0 \right), \Ind_{P^-}^G \left( \Sigma^0 / \varpi^n \Sigma^0 \right) \right) \\
\to \Hom_L \left( \Sigma'^0 / \varpi^n \Sigma'^0, \HOrd[1]_P \left( \Ind_{P^-}^G \left( \Sigma^0 / \varpi^n \Sigma^0 \right) \right) \right).
\end{multline}
Note that the composite $\HOrd[1]_P \circ \Ind_{P^-}^G : \Mod_L^\adm(\Oc/\varpi^n\Oc) \to \Mod_L^\adm(\Oc/\varpi^n\Oc)$ is left-exact for all $n \geq 1$ by \cite[Proposition 4.3.4]{Em1} and \cite[Corollary 3.4.8]{Em2}.
Thus $\varprojlim_{n \geq 1} (\HOrd[1]_P(\Ind_{P^-}^G(\Sigma^0/\varpi^n\Sigma^0)))$ is a $\varpi$-adically admissible representation of $L$ over $\Oc$ by \cite[Corollary 3.4.5]{Em1}.
Furthermore, it is $\varpi$-torsion-free and the projective limit topology coincide with the $\varpi$-adic topology (cf. \cite[Proposition 3.4.3 (1) and (3)]{Em1}).
Thus
\begin{equation*}
E \otimes_\Oc \varprojlim_{n \geq 1} \left( \HOrd[1]_P \left( \Ind_{P^-}^G \left( \Sigma^0 / \varpi^n \Sigma^0 \right) \right) \right)
\end{equation*}
is an admissible unitary continuous representation of $L$ over $E$.
Taking the projective limit over $n \geq 1$ of \eqref{SEExt1A} and inverting $\varpi$ and using Proposition \ref{prop:devExt1} yields the desired exact sequence.
\end{proof}

\begin{rema} \label{rema:SEExtcont}
In order to obtain an analogue of \eqref{SEExt1} for any admissible unitary continuous representations $\Sigma,\Pi$ of $L,G$ respectively over $E$, one has to prove that the $\varpi$-torsion of $\varprojlim_{n \geq 1} (\HOrd[1]_P(\Pi^0/\varpi^n\Pi^0))$ is of bounded exponent (i.e. annihilated by a power of $\varpi$) for some (equivalently any) $G$-stable bounded open $\Oc$-lattice $\Pi^0 \subseteq \Pi$.
\end{rema}

We now use Proposition \ref{prop:SEExtcont} to compute extensions between parabolically induced representations.

\begin{theo} \label{theo:Ext1contQp}
Assume $F=\Qp$.
Let $\Pb = \Lb \Nb,\Pb' = \Lb' \Nb'$ be standard parabolic subgroups, $\Sigma,\Sigma'$ be admissible unitary continuous representations of $L,L'$ respectively over $E$ and $\sigma,\sigma'$ be the reductions mod $\varpi$ of $L,L'$-stable bounded open $\Oc$-lattices of $\Sigma,\Sigma'$ respectively.
Assume that $\dim_k \Hom_G(\Ind_{P'^-}^G \sigma',\Ind_{P^-}^G \sigma) < \infty$ and $\Sigma$ is right cuspidal.
\begin{enumerate}
\item If $\Pb'=\Pb$, $\Sigma,\Sigma'$ are topologically irreducible and $\Sigma' \not \cong \Sigma^\alpha \otimes (\varepsilon^{-1} \circ \alpha)$ for all $\alpha \in \Delta_\Lb^{\perp,1}$, then the functor $\Ind_{P^-}^G$ induces an $E$-linear isomorphism
\begin{equation*}
\Ext_L^1 \left( \Sigma', \Sigma \right) \iso \Ext_G^1 \left( \Ind_{P^-}^G \Sigma', \Ind_{P^-}^G \Sigma \right).
\end{equation*}
\item If $\Pb' \subsetneq \Pb$, then the functor $\Ind_{P^-}^G$ induces an $E$-linear isomorphism
\begin{equation*}
\Ext_L^1 \left( \Ind_{L \cap P'^-}^L \Sigma', \Sigma \right) \iso \Ext_G^1 \left( \Ind_{P'^-}^G \Sigma', \Ind_{P^-}^G \Sigma \right).
\end{equation*}
\end{enumerate}
\end{theo}

\begin{rema}
Assume $\Pb'=\Pb$ and $\Sigma,\Sigma'$ topologically irreducible.
We do not know the dimension of the cokernel of the $E$-linear injection $\Ext_L^1(\Sigma',\Sigma) \hookrightarrow \Ext_G^1(\Ind_{P^-}^G \Sigma',\Ind_{P^-}^G \Sigma)$ induced by $\Ind_{P^-}^G$, but we prove that it is at most $\card\{\alpha \in \Delta_\Lb^{\perp,1} \mid \Sigma' \cong \Sigma^\alpha \otimes (\varepsilon^{-1} \circ \alpha)\}$ (see the proof).
If $\Sigma,\Sigma'$ are absolutely topologically irreducible and supercuspidal, then letting $\zeta : Z_L \to \Oc^\times \subset E^\times$ be the central character of $\Sigma$ (cf. \cite[Theorem 1.1, 2)]{DS}), we expect this dimension to be equal to
\begin{equation*}
\card \left\{ \alpha \in \Delta_\Lb^{\perp,1} \middlevert \Sigma' \cong \Sigma^\alpha \otimes \left( \varepsilon^{-1} \circ \alpha \right) \text{ and } \zeta \circ \alpha^\vee \neq \varepsilon^{-1} \right\}.
\end{equation*}
\end{rema}

\begin{proof}
Let $\Sigma^0 \subseteq \Sigma$ be an $L$-stable bounded open $\Oc$-lattice.
For all $n \geq 1$, we deduce from Propositions \ref{prop:FilPJ} and \ref{prop:HOrd1} (i) that there is a natural $L$-equivariant $\Oc/\varpi^n\Oc$-linear isomorphism
\begin{equation} \label{H1OrdA}
\HOrd[1]_P \left( \Ind_{P^-}^G \left( \Sigma^0 / \varpi^n \Sigma^0 \right) \right) \cong \bigoplus_{\alpha \in \Delta^1 \backslash \Delta_\Lb} \HOrd[1]_P \left( \cind_{P^-}^{P^- s_\alpha P} \left( \Sigma^0 / \varpi^n \Sigma^0 \right) \right).
\end{equation}
Furthermore, if $\alpha \in \Delta_\Lb^{\perp,1}$, then there is a natural $L$-equivariant $\Oc/\varpi^n\Oc$-linear isomorphism
\begin{equation*}
\HOrd[1]_P \left( \cind_{P^-}^{P^- s_\alpha P} \left( \Sigma^0 / \varpi^n \Sigma^0 \right) \right) \cong \left( \Sigma^0 / \varpi^n \Sigma^0 \right)^\alpha \otimes \left( \omega^{-1} \circ \alpha \right)
\end{equation*}
hence a natural $L$-equivariant $E$-linear continuous isomorphism
\begin{equation*}
E \otimes_\Oc \varprojlim_{n \geq 1} \left( \HOrd[1]_P \left( \cind_{P^-}^{P^- s_\alpha P} \left( \Sigma^0 / \varpi^n \Sigma^0 \right) \right) \right) \cong \Sigma^\alpha \otimes \left( \varepsilon^{-1} \circ \alpha \right),
\end{equation*}
whereas if $\alpha \not \in \Delta_\Lb^\perp$, then there is natural filtration of $\HOrd[1]_P(\cind_{P^-}^{P^- s_\alpha P}(\Sigma^0/\varpi^n\Sigma^0))$ by $B_L$-subrepresentations such that each term of the associated graded representation is isomorphic as an $\Oc/\varpi^n\Oc$-modules to
\begin{equation*}
\Clisc\! \left( U'_L, \Ord_{L \cap s_\alpha P s_\alpha^{-1}} \left( \Sigma^0 / \varpi^n \Sigma^0 \right) \right)
\end{equation*}
for some closed subgroup $\Ub'_\Lb \subseteq \Ub_\Lb$, and since $\Ord_{L \cap s_\alpha P s_\alpha^{-1}} \Sigma = 0$ we deduce that
\begin{equation*}
E \otimes_\Oc \varprojlim_{n \geq 1} \left( \HOrd[1]_P \left( \cind_{P^-}^{P^- s_\alpha P} \left( \Sigma^0 / \varpi^n \Sigma^0 \right) \right) \right) = 0.
\end{equation*}
Thus, taking the projective limit of \eqref{H1OrdA} over $n \geq 1$ and inverting $\varpi$ yields a natural $L$-equivariant $E$-linear continuous isomorphism
\begin{equation} \label{H1Ordcont}
E \otimes_\Oc \varprojlim_{n \geq 1} \left( \HOrd[1]_P \left( \Ind_{P^-}^G \left( \Sigma^0 / \varpi^n \Sigma^0 \right) \right) \right) \cong \bigoplus_{\alpha \in \Delta_\Lb^{\perp,1}} \Sigma^\alpha \otimes \left( \varepsilon^{-1} \circ \alpha \right).
\end{equation}

Now Proposition \ref{prop:SEExtcont} with $\Ind_{L \cap P'^-}^L \Sigma'$ instead of $\Sigma'$ yields, using \eqref{H1Ordcont}, an exact sequence of $E$-vector spaces
\begin{multline} \label{SEExtQpcont}
0 \to \Ext_L^1 \left( \Ind_{L \cap P'^-}^L \Sigma', \Sigma \right) \to \Ext_G^1 \left( \Ind_{P'^-}^G \Sigma', \Ind_{P^-}^G \Sigma \right) \\
\to \bigoplus_{\alpha \in \Delta_\Lb^{\perp,1}} \Hom_L \left( \Ind_{L \cap P'^-}^L \Sigma', \Sigma^\alpha \otimes \left( \varepsilon^{-1} \circ \alpha \right) \right).
\end{multline}
If $\Pb'=\Pb$ and $\Sigma,\Sigma'$ are topologically irreducible, then $\Sigma^\alpha \otimes (\varepsilon^{-1} \circ \alpha)$ is also topologically irreducible for all $\alpha \in \Delta_\Lb^{\perp,1}$, and thus the last term of \eqref{SEExtQpcont} has dimension equal to $\card\{\alpha \in \Delta_\Lb^{\perp,1} \mid \Sigma' \cong \Sigma^\alpha \otimes (\varepsilon^{-1} \circ \alpha)\}$, hence (i).
If $\Pb' \subsetneq \Pb$, then $\Lb \cap \Pb'$ is a proper parabolic subgroup of $\Lb$ so that
\begin{equation*}
\Ord_{L \cap P'} \left( \Sigma^\alpha \otimes \left( \varepsilon^{-1} \circ \alpha \right) \right) = \left( \Ord_{L \cap P'} \Sigma \right)^\alpha \otimes \left( \varepsilon^{-1} \circ \alpha \right) = 0
\end{equation*}
for all $\alpha \in \Delta_\Lb^{\perp,1}$, thus the last term of \eqref{SEExtQpcont} is zero, hence (ii).
\end{proof}

\begin{rema}
Theorem \ref{theo:Ext1Qp} (iii) cannot be directly lifted to characteristic $0$ because we do not have a weak $p$-adic analogue of the exact sequence \eqref{SEExt2} (since it uses the Jacquet functor, see Remark \ref{rema:Jacquetcont}).
However, assuming Conjecture \ref{conj:HOrd} true for $A=\Oc/\varpi^r\Oc$ ($r \geq 1$), $n=1$, $I \subsetneq J$ and $\Iw^J=s_\alpha$ ($\alpha \in \Delta^1 \backslash J$), one can recover this case: with notation and assumptions as in Theorem \ref{theo:Ext1contQp}, if $\Pb \subsetneq \Pb'$ and $\Sigma'$ is left cuspidal, then the functor $\Ind_{P'^-}^G$ induces an $E$-linear isomorphism
\begin{equation*}
\Ext_{L'}^1 \left( \Sigma', \Ind_{L' \cap P^-}^L \Sigma \right) \iso \Ext_G^1 \left( \Ind_{P'^-}^G \Sigma', \Ind_{P^-}^G \Sigma \right).
\end{equation*}
\end{rema}

\begin{theo} \label{theo:Ext1contF}
Assume $F \neq \Qp$.
Let $\Pb = \Lb \Nb$ be a standard parabolic subgroup, $\Sigma,\Sigma'$ be admissible unitary continuous representations of $L$ over $E$ and $\sigma,\sigma'$ be the reductions mod $\varpi$ of $L$-stable bounded open $\Oc$-lattices of $\Sigma,\Sigma'$ respectively.
Assume that $\dim_k \Hom_L(\sigma',\sigma) < \infty$.
Then, the functor $\Ind_{P^-}^G$ induces an $E$-linear isomorphism
\begin{equation*}
\Ext_L^1 \left( \Sigma', \Sigma \right) \iso \Ext_G^1 \left( \Ind_{P^-}^G \Sigma', \Ind_{P^-}^G \Sigma \right).
\end{equation*}
\end{theo}

\begin{proof}
Let $\Sigma^0 \subseteq \Sigma$ be an $L$-stable bounded open $\Oc$-lattice.
By Corollary \ref{coro:HOrd} (i), we have $\HOrd[1]_P(\Ind_{P^-}^G (\Sigma^0/\varpi^n\Sigma^0))=0$ for all $n \geq 1$.
Thus, the result follows from Proposition \ref{prop:SEExtcont}.
\end{proof}

We end with a remark on the case where there is no inclusion between the two parabolic subgroups.

\begin{rema}
Let $\Pb = \Lb \Nb,\Pb' = \Lb' \Nb'$ be standard parabolic subgroups, $\Sigma,\Sigma'$ be admissible unitary continuous representations of $L,L'$ respectively over $E$ and $\sigma,\sigma'$ be the reductions mod $\varpi$ of $L,L'$-stable bounded open $\Oc$-lattices of $\Sigma,\Sigma'$ respectively.
Assume Conjecture \ref{conj:HOrd} is true for $A=\Oc/\varpi^r\Oc$ ($r \geq 1$), $n=1$ and $\Iw^J=1$.
Assume further $\dim_k \Hom_G(\Ind_{P'^-}^G \sigma',\Ind_{P^-}^G \sigma) < \infty$ and the $\varpi$-torsion of $\varprojlim_{n \geq 1} (\HOrd[1]_{L \cap P}(\Sigma^0/\varpi^n\Sigma^0))$ is of bounded exponent (see Remark \ref{rema:SEExtcont}).
Then, one can prove the following $p$-adic analogue of Proposition \ref{prop:Ext1}: if $\Pb' \not \subseteq \Pb$, $\Pb \not \subseteq \Pb'$, and $\Sigma,\Sigma'$ are right,left cuspidal respectively, then
\begin{equation*}
\Ext_G^1 \left( \Ind_{P'^-}^G \Sigma', \Ind_{P^-}^G \Sigma \right) = 0.
\end{equation*}
\end{rema}

\bibliographystyle{alpha}
\bibliography{parabolique}

\end{document}